\documentclass[reqno,11pt]{amsart}
\usepackage{amssymb, amsmath,latexsym,amsfonts,amsbsy, amsthm}
\usepackage{color}
\usepackage{mathrsfs}
\usepackage{esint}
\usepackage{graphics,color}
\usepackage{enumerate}
\usepackage{marginnote}

\newcommand{\beq}{\begin{equation}}
\newcommand{\eeq}{\end{equation}}
\newcommand{\ben}{\begin{eqnarray}}
\newcommand{\een}{\end{eqnarray}}
\newcommand{\beno}{\begin{eqnarray*}}
\newcommand{\eeno}{\end{eqnarray*}}

\renewcommand{\theequation}{\thesection.\arabic{equation}}

\newtheorem{theorem}{Theorem}[section]

\newtheorem{proposition}[theorem]{Proposition}

\newtheorem{Theorem}{Theorem}[section]
\newtheorem{Definition}[Theorem]{Definition}

\newtheorem{Lemma}[Theorem]{Lemma}
\newtheorem{Corollary}[Theorem]{Corollary}
\newtheorem{Remark}[Theorem]{Remark}

\newcommand{\Id}{\mathrm{Id}}

\begin{document}

\title[Boussinesq Equation]
{H\"{o}lder Continuous Solutions Of Boussinesq Equation with compact support }

\author{Tao Tao}
\address{School of  Mathematics Sciences, Peking University, Beijing, China}
\email{taotao@amss.ac.cn}
\author{Liqun Zhang}
\address{Academy of mathematic and system science , Beijing , China}
\email{lqzhang@math.ac.cn}
%
%

\date{\today}
\maketitle

\renewcommand{\theequation}{\thesection.\arabic{equation}}
\setcounter{equation}{0}

\begin{abstract}
We show the existence of H\"{o}lder continuous solution of Boussinesq equations in whole space which has compact support both in space and time.
\end{abstract}


\vskip 0.2cm



\setcounter{equation}{0}
\section{Introduction}
In this paper, we consider the following Boussinesq system
\begin{equation}\label{e:boussinesq equation}
\begin{cases}
v_{t}+\hbox{div}(v\otimes v)+\nabla p=\theta e_{2}, \quad (x,t)\in R^2 \times R,\\
\hbox{div}v=0,\quad  (x,t)\in R^2 \times R,\\
\theta_{t}+\hbox{div}(v\theta)=0, \quad (x,t)\in R^2 \times R.
\end{cases}
\end{equation}
Here $e_2=(0,1)^T$, $v$ is the velocity vector, $p$ is the pressure, $\theta$ is a scalar function.
The Boussinesq equations arises from many geophysical flows, such as atmospheric fronts and ocean circulations (see, for example, \cite{Ma},\cite{Pe}).
To understand the turbulence phenomena in fluid mechanics, one needs to go beyond classical solutions. The pair $(v, p, \theta)$ on $R^2\times R$ is called a weak solution of (\ref{e:boussinesq equation}) if they belong to $L^2_{loc}(R^2\times R)$ and solve (\ref{e:boussinesq equation})  in the following sense:
\begin{align}
\int_R\int_{R^2}(\partial_t\varphi\cdot v+v\otimes v: \nabla\varphi+p\hbox{div}\varphi+\theta e_2\cdot\varphi )dxdt=0,\nonumber
\end{align}
for all $\varphi\in C_c^\infty(R^2\times R;R^2).$
\begin{align}
\int_R\int_{R^2}(\partial_t\phi\theta+v\cdot\nabla\phi\theta)dxdt=0,\nonumber
\end{align}
for all $\phi\in C_c^\infty(R^2\times R;R)$ and
\begin{align}
\int_R\int_{R^2}v\cdot\nabla\psi dxdt=0.\nonumber
\end{align}
for all $\psi\in C_c^\infty(R^2\times R;R).$

The study of weak solutions in fluid dynamics attract more and more peoples interests. One of the famous problem is the Onsager conjecture on Euler equation which says that the incompressible Euler equation admits H\"{o}lder continuous weak solution which dissipates kinetic energy. More precisely, the Onsager conjecture on Euler equation can be stated as following:
\begin{enumerate}
  \item $C^{0,\alpha}$ solution are energy conservative when $\alpha>\frac{1}{3}$.\\
  \item For any $\alpha <\frac{1}{3}$, there exist dissipative solutions with $C^{0,\alpha}$ regularity .
\end{enumerate}

The part (1) has been proved by P. Constantin, E, Weinan and  E. Titi in \cite{CET} and also by P. Constantin, etc. in \cite{CPFR} with slightly weaker assumption.

The part (2) seems more subtle and has been treated by many authors. For weak solutions, the non-uniqueness results have been obtained by V. Scheffer (\cite{VS}), A. Shnirelman (\cite{ASH1,ASH2}) and Camillo De Lellis, L\'{a}szl\'{o} Sz\'{e}kelyhidi (\cite{CDL,CDL0}). In particular, a great progress in the construction of H\"{o}lder continuous solution was made by Camillo De Lellis, L\'{a}szl\'{o} Sz\'{e}kelyhidi etc in recent years. In fact, Camillo De Lellis and L\'{a}szl\'{o} Sz\'{e}kelyhidi developed an iterative scheme  in \cite{CDL2}, together with the aid of Beltrami flow on $T^3$ and Geometric Lemma, and constructed a continuous periodic solution which satisfies the prescribed kinetic energy. The solution is a superposition of infinitely many weakly interacting Beltrami flows.  Building on the iterative techniques in \cite{CDL2} and Nash-Moser mollify techniques, they constructed H\"{o}lder continuous periodic solutions with exponent $\theta<\frac{1}{10}$, which satisfies the prescribed kinetic energy in \cite{CDL3}. P. Isett  in \cite{IS1} constructed H\"{o}lder continuous periodic solutions with  any $\theta<\frac{1}{5}$, and the solution has compact support in time. By introducing some new devices in \cite{BCDL1}, Camillo De Lellis, L\'{a}szl\'{o} Sz\'{e}kelyhidi and T. Buckmaster constructed H\"{o}lder continuous weak solutions with $\theta<\frac{1}{5}$, which satisfies the prescribed kinetic energy, also see \cite{BCDLI}.
In $R^3$, P. Isett and Sung-Jin Oh in \cite{ISOH2} constructed H\"{o}lder continuous solutions  with $\theta<\frac{1}{5}$, which satisfies the prescribed kinetic energy or is a perturbation of smooth Euler flow. Recently, S. Daneri obtained dissipative H\"{o}lder solutions for the Cauchy problem of incompressible Euler flow in \cite{DA}.

Concerning the Onsager conjecture with the critical spatial regularity, namely H\"{o}lder exponent $\theta=\frac{1}{3}$, there are also some interesting results. By time localized estimates and careful choice of the parameters in \cite{TBU}, T. Buckmaster constructed H\"{o}lder continuous periodic solutions with exponent $\theta<\frac{1}{5}$ in time-space, which for almost every time belongs to $C_x^{\theta}$, for any $\theta<\frac{1}{3}$ and is compactly temporal supported. Later, by smoothing Reynolds stress for different time intervals using different approach carefully and introducing some novel ideas in \cite{BCDL2} , Camillo De Lellis, L\'{a}szl\'{o} Sz\'{e}kelyhidi and T. Buckmaster constructed  H\"{o}lder continuous periodic solution which belongs to $L^1_tC_x^\theta$, for any $\theta<\frac{1}{3}$ and has compact support in time.

Motivated by  the above earlier works, we want to know if the similar phenomena can also happen when considering the temperature effects in the incompressible Euler flow which is the Boussinesq system. In \cite{TZ}, we construct continuous solutions for Boussinesq equations on torus which satisfies the prescribed kinetic energy. In this paper, we consider the existence of H\"{o}lder continuous solution with compact support both in space and time for Boussinesq equations. The main difficulty is to deal with the interactions between velocity and temperature. Following the general framework of convex integration method developed  by De Lellis and Sz\'{e}kelyhidi for Euler equations, by establishing the corresponding geometric lemma and constructing oscillatory perturbation which are compatible with Boussinesq equations, we obtained the following results.
New, we  state our theorem:

\begin{theorem}\label{t:main 1}
For any given positive number $r,~\varepsilon\in (0,\frac{1}{28})$ , there exist a triple
\begin{align}
    (v,p,\theta)\in C_c(Q_{2r};R^{2}\times R\times R)\nonumber
\end{align}
such that they
solve the system \eqref{e:boussinesq equation} in the sense of distribution and
\begin{align}
 v\neq 0.\nonumber
\end{align}
Moreover, we have
$$v\in C^{\frac{1}{28}-\varepsilon}_{t,x},\qquad \theta\in  C^{\frac{1}{25}-\varepsilon}_{t,x},\qquad p\in C^{1-\varepsilon}_{t,x}.$$
\end{theorem}
 Here and subsequent, $Q_{r}=\{(t,x)\in R^3:|x|<r, |t|< r\}$.

\begin{Remark}
In our theorem \ref{t:main 1}, if $\theta=0$, then it is the H\"{o}lder continuous Euler flow with compact support and have been constructed by P.Isett and Sung-jin Oh in \cite{ISOH2}. In fact, they construct  H$\ddot{o}$lder continuous Euler flow with  H$\ddot{o}$lder exponent $\frac{1}{5}-\varepsilon$. Moreover, we can construct weak solution such that $\theta\neq0.$
\end{Remark}




\begin{Remark}\label{r:remark second}
Similar results also hold for the 3-dimensional Boussinesq system on $R^3$ with some H\"{o}lder exponent.
\end{Remark}

 We briefly give some comments on our proof. In \cite{ISOH2}, the authors make use of a families of Beltrami flows to control the interference terms between different waves in the construction. For the Boussinesq system, we do not know if there exist the analogous special solutions. Following an idea in \cite{NASH,ISV2}, we make use of a multi-steps iteration scheme and one-dimensional oscillation. More precisely, in each step, we add some plane waves which oscillate along the same direction with different frequency, and thus only remove one component of stress error in each step. To reduce the whole stress errors, we divide the process into several steps. On the other hand, since the velocity and temperature are coupled together in Boussinesq system, we need to reduce two stress errors simultaneously. To achieve this, we need to extend the geometric lemma in \cite{CDS} and add associated plane waves in the velocity and temperature simultaneously in each step. Their coordination is important for us to reduce the temperature stress error and construct continuous temperature.


\setcounter{equation}{0}
\section{Main proposition and outline of the proof}

As in \cite{CDL2}, the proof of Theorem \ref{t:main 1} will be achieved through an iteration procedure. In the following, $\mathcal{S}^{2\times2}$ always denotes the vector space of
symmetric $2\times2$ matrices.
\begin{Definition}\label{d:boussinesq reynold}
Assume $v,p,\theta,R,f$ are smooth and compact supported functions on $R^{2}\times R$ taking values,~respectively,~in $R^{2}, R, R, \mathcal{S}^{2\times2}, R^{2}$.
We say that they solve the Boussinesq-stress system  if
\begin{equation}\label{d:boussinesq reynold}
\begin{cases}
\partial_tv+{\rm div}(v\otimes v)+\nabla p=\theta e_{2}+{\rm div}R,\\
{\rm div}v=0,\\
\theta_{t}+{\rm div}(v\theta)={\rm div}f.
\end{cases}
\end{equation}
\end{Definition}
\subsection{Some notations on (semi)norm}
In the following, $ m=0,1,2,...$, $\beta$ is a multiindex and $D$ is spatial derivative.
First of all, we denote the supremum norm $\|f\|_0$ by
\begin{align}
 \|f\|_{0}:=\hbox{sup}_{R^{2}}|f|.\nonumber
 \end{align}
 Then the $\dot{C}^m$ seminorms are given by
 \begin{align}
 [f]_{m}:=\hbox{max}_{|\beta|=m}\|D^{\beta}f\|_{0}\nonumber
\end{align}
and $C^m$ norms are given by
\begin{align}
\|f\|_{m}:=\sum_{j=0}^{m}[f]_{j}.\nonumber
\end{align}
If $f=f_1+if_2$ is a complex-valued function, then we set $[f]_m:=[f_1]_m+[f_2]_m$ and $\|f\|_m:=\|f_1\|_m+\|f_2\|_m$.

Moreover, for functions depending on space and time, we introduce the following space-time norm:
\begin{align}
\|f\|_m:=\sup_t\|f(t,\cdot)\|_m,\quad \|f\|_{C^1_{t,x}}:=\|f\|_1+\|\partial_tf\|_0.\nonumber
\end{align}
We now state the main proposition of this paper,~of which Theorem \ref{t:main 1} is a corollary.
\begin{proposition}\label{p: iterative 1}
Let $r>0, \varepsilon>0$ be two given positive numbers. Then there exist positive constants $\eta, M$ such that the following property holds:

For any $0< \delta\leq1$,~if $(v_,~ p, ~\theta, ~R, ~f)\in C_c^{\infty}(Q_{r})$ solves Boussinesq-stress system (\ref{d:boussinesq reynold}) and
\begin{align}
     \|R\|_0\leq& \eta\delta,\label{e:reynold initial}\\ \|f\|_0\leq& \eta\delta.\label{e:reynold initial 2}
\end{align}
Set
\begin{align}
\Lambda:=\max\{1, \|R\|_{C^1_{t,x}}, \|f\|_{C^1_{t,x}}, \|v\|_{C^1_{t,x}}, \|\theta\|_{C^1_{t,x}}\}.\nonumber
\end{align}

Then, for any $\bar{\delta}\leq\frac{1}{2}\delta^{\frac{3}{2}}$, we can construct new functions $(\tilde{v},~\tilde{p},~\tilde{\theta},~\tilde{R},~\tilde{f})\in C_c^{\infty}(Q_{r+\delta})$, which also solves Boussinesq-stress system (\ref{d:boussinesq reynold}) and satisfies
\begin{align}
    \|\tilde{R}\|_0\leq&\eta\bar{\delta} ,\label{e:iterative stress estimate}\\
     \|\tilde{f}\|_0\leq& \eta\bar{\delta},\\
    \|\tilde{v}-v\|_0 \leq& M\sqrt{\delta},\\
    \|\tilde{\theta}-\theta\|_0 \leq& M\sqrt{\delta},\\
    \|\tilde{p}-p\|_0 \leq& M\delta,\label{e:iterative pressure difference estimate}
 \end{align}
    and
   \begin{align}\label{e:first derivative estimate for stress term}
   \Lambda_1:=\max\{1, \|\tilde{R}\|_{C^1_{t,x}}, \|\tilde{f}\|_{C^1_{t,x}}, \|\tilde{v}\|_{C^1_{t,x}}, \|\tilde{\theta}\|_{C^1_{t,x}}\}
\leq A\delta^{\frac{\varepsilon^2+3\varepsilon+3}{2}}\Big(\frac{\sqrt{\delta}}{\bar{\delta}}\Big)^{(1+\varepsilon)^{2}(2+\varepsilon)+(2+\varepsilon)^2}
\Lambda^{(1+\varepsilon)^{3}}.
\end{align}
Moreover,
\begin{align}\label{e:terperture derivative estimates}
\|\tilde{p}\|_{C^1_{t,x}}&\leq C_0,\quad\|\tilde{\theta}\|_{C^1_{t,x}}\leq A\delta^{\frac{2+\varepsilon}{2}}
\Big(\frac{\sqrt{\delta}}{\bar{\delta}}\Big)^{4+4\varepsilon+\varepsilon^2}
\Lambda^{(1+\varepsilon)^2
}.
\end{align}
where the constant $A$ depends on $r, \|v\|_0$ linearly and depend on $\varepsilon$.
\end{proposition}
We will prove Proposition \ref{p: iterative 1} in the subsequent sections.

\subsection{Outline of the proof of Proposition \ref{p: iterative 1}}

\indent

The rest of this paper will be dedicated to prove Proposition \ref{p: iterative 1}. The construction of the functions $\tilde{v}, \tilde{\theta}$ consists of a stage which contains three steps. In the first step, we add perturbations to $v_0, \theta_0$  and get new functions $v_{01}, \theta_{01}$ as following:
 \begin{align}
 v_{01}=&v_0+w_{1o}+w_{1c}:=v_0+w_1,\nonumber \\
 \theta_{01}=&\theta_0+\chi_{1o}+\chi_{1c}:=\theta_0+\chi_{1},\nonumber
 \end{align}
where $w_{1o}, w_{1c}, \chi_{1o}, \chi_{1c}$ are highly oscillatory functions with compact support given by explicit formulas. We introduce three parameters $\ell,~~\mu_1,~\lambda_1$ in the construction of perturbation with
\begin{align}
 \qquad 1\ll\mu_1\ll\lambda_1.\nonumber
 \end{align}
After adding these perturbations, we mainly focus on finding functions $R_{01}, p_{01}$ and $f_{01}$ with the desired estimates and solving system (\ref{d:boussinesq reynold}).
 After the first step, the stresses error become smaller in the following sense:

 If
\begin{align}
R_0(t,x)-e(t,x)Id=-\sum\limits_{i=1}^{3}a^2_i(t,x)k_i\otimes k_i,\quad
f_0(t,x)=-\sum\limits_{i=1}^{2}c_i(t,x)k_i,\nonumber
\end{align}
where $e(t,x)$ is a smooth function with compact support, see (\ref{d:definition on p}) and (\ref{d:difinition on e}).
Then
\begin{align}
R_{01}(t,x)=-\sum\limits_{i=2}^{3}a^2_i(t,x)k_i\otimes k_i+\delta R_{01},\quad
f_{01}(t,x)=-c_2(t,x)k_2+\delta f_{01},\nonumber
\end{align}
where $\delta R_{01}, \delta f_0$ can be arbitrary small by the appropriate choice of $\ell, \mu_1, \lambda_1$.

We repeat the above step, till obtain the needed $(\tilde{v},~\tilde{p},~\tilde{\theta},~\tilde{R},~\tilde{f}).$\\

The rest of paper is organized as follows. In section 3, we prove Geometric Lemma and introduce two operators. After these preliminaries, we perform the first step in next three sections. In section 4, we introduce the perturbations $w_{1o}, w_{1c}, \chi_{1o}, \chi_{1c}$ and new stresses $R_{01}, f_{01}$ and prescribe the constant $\eta,~M$ appeared in Proposition 2.1. In sections 5 and 6, we
calculate the main forms of $R_{01}, f_{01}$ and prove the relevant estimates of the various terms involved in the construction, in  term of the parameters
$\lambda_1, \mu_1, \ell$. In sections 7, 8 and 9 we construct $v_{0n}, p_{0n}, \theta_{0n}, R_{0n}, f_{0n}$ for $n=2,3$ and prove relevant estimates by inductions. The construction given in section 7 is similar to that of the first step in section 4. We calculate the main forms of $R_{0n}, f_{0n}$ in section 8 and prove the various error estimates in section 9. After completing the constructions of $(v_{03}, p_{03}, \theta_{03}, R_{03}, f_{03})$ and various estimates, we give a proof of Proposition 2.1 by choosing appropriate parameters $\ell, \mu_n, \lambda_n$ for $1\leq n\leq 3$ in section 10. Finally, in section 11, we give a proof of Theorem \ref{t:main 1}.

\setcounter{equation}{0}
\section{Preliminaries}

\indent

We always make use of the following notations:
$R^{2\times 2}$ denotes the space of $2\times 2$ matrices; $\mathcal{S}^{2\times 2}$, as before,
 denotes the spaces of $2\times 2$ symmetric matrices
and $\Id$ denotes $2\times 2$ identity matrix .
The matrix norm $|R|:=\max_{1\leq i,j\leq 2}|R_{ij}|$, if $R=(R_{ij})_{2\times 2}$.

\subsection{Geometric Lemma}

\indent

The following lemma is a kind of geometric lemma given in \cite{CDS} to our case. Within it, we represent not only a prescribed symmetric matrix $R$, but also a prescribed vector.

\begin{Lemma}[Geometric Lemma]\label{p:split}
There exist $r_{0}>0$, $k_1, k_2, k_3\in R^2\setminus\{0\}$, smooth positive functions
 \begin{align*}
    \gamma_{k_i}\in C^{\infty}(B_{r_{0}}(Id)), \quad \frac{1}{2}\leq \gamma_{k_i}\leq\frac{3}{2},~~~~  i=1,2,3
\end{align*}
and linear functions
 \begin{align*}
    g_{k_i}\in C^{\infty}(R^{2}), i=1,2
 \end{align*}
such that
\begin{enumerate}
 \item for every $R \in B_{r_{0}}(Id)$, we have
  $$R=\sum\limits_{i=1}^3\gamma^2_{k_i}(R)k_i\otimes k_i;$$
   \item for every $f\in R^{2}$,~we have
   $$f=\sum_{i=1}^2g_{k_i}(f)k_i.$$
\end{enumerate}

\end{Lemma}
\begin{proof}
The proof is constructive. We set
\begin{align}\label{d:definition of ki}
k_1:=\Big(\frac{1}{\sqrt{2}},\frac{1}{\sqrt{3}}\Big)^T,\quad k_2:=\Big(-\frac{1}{2},\frac{2}{\sqrt{6}}\Big)^T,\quad k_3:=\Big(\frac{1}{2},0\Big)^T.
\end{align}
A straightforward computation gives
\begin{align}
k_1\otimes k_1=\left(\begin{array}{cc}
\frac{1}{2} & \frac{1}{\sqrt{6}}\\
\frac{1}{\sqrt{6}} & \frac{1}{3}
\end{array}\right),\qquad
k_2\otimes k_2=
\left(\begin{array}{cc}
\frac{1}{4} & -\frac{1}{\sqrt{6}}\\
-\frac{1}{\sqrt{6}} & \frac{2}{3}
\end{array}\right),\qquad
k_3\otimes k_3=
\left(\begin{array}{cc}
\frac{1}{4} & 0\\
0 & 0
\end{array}\right).\nonumber
\end{align}
It's obvious that $k_1\otimes k_1, k_2\otimes k_2, k_3\otimes k_3$ are linearly independent, hence form a basic for $\mathcal{S}^{2\times2}$ and
$$\sum_{i=1}^3k_i\otimes k_i=Id.$$
Thus, taking $r_0> 0$ small, for any symmetric matrix $R\in B_{r_0}(Id)$ , the following equation
$$\sum_{i=1}^3\gamma^2_{k_i}k_i\otimes k_i=R$$
has unique, positive solution $\gamma_{k_i}$  and
$$\|\gamma_{k_i}-1\|_0\leq \frac{1}{2}.$$
Since the representation is unique, the dependence of $\gamma_{k_i}$ on $R$ is smooth.\\
Then,  for any $f\in R^{2}$, we set
  \[
  \left(\begin{array}{cc}
  g_{k_1}(f)\\[12pt]
  g_{k_2}(f)
  \end{array}\right)
  :=(k_1, k_2)^{-1}f,\]
where $(k_1, k_2)^{-1}$ denote the inverse matrix of $(k_1, k_2).$
Thus, $g_{k_1}, g_{k_2}$ are linear functions and it's obvious that
 \begin{align*}
f=\sum_{i=1}^2g_{k_i}(f)k_i, \quad \forall f\in R^2.
  \end{align*}
Thus, we finished the proof of this lemma.
\end{proof}

As in \cite{ISOH2}, we introduce the following operators
in order to deal with the stresses.

\subsection{The operator $\mathcal{R}$ }

\indent

Suppose that the vector function $U(x)=(U_1(x),U_2(x))^T\in C_c^{\infty}(B_r; C^2)$  satisfy
\begin{align}\label{v:vanish 1}
\int_{R^2}U_i(x)dx=0,\quad \int_{R^2}(x_iU_j-x_jU_i)(x)dx=0,\quad i,j=1,2
\end{align} and take the following form
\begin{align}
U(x)=\left(\begin{array}{c}
V_1(x)e^{i\lambda k\cdot x}\\
V_2(x)e^{i\lambda k\cdot x}
\end{array}\right):= V(x)e^{i\lambda k\cdot x},\nonumber
\end{align}
where $\lambda>0,\quad k=(k_1, k_2)\in R^2\setminus\{0\}$.\\
We denote $\mathcal{R}U_o(x)$ by
\begin{align}
\mathcal{R}U_o(x):=\left(\begin{array}{cc}
\frac{M_{11}(x)}{i\lambda}e^{i\lambda k\cdot x}     &   \frac{M_{12}(x)}{i\lambda}e^{i\lambda k\cdot x}\\
\frac{M_{12}(x)}{i\lambda}e^{i\lambda k\cdot x}     &   \frac{M_{22}(x)}{i\lambda}e^{i\lambda k\cdot x}
\end{array}\right),\nonumber
\end{align}
where $M=(M_{11}, M_{12}, M_{22})$ satisfy
\begin{align}\label{e:linear equation}
M_{11}k_1+M_{12}k_2=V_1, \quad M_{12}k_1+M_{22}k_2=V_2.
\end{align}
Obviously, the linear equation (\ref{e:linear equation}) always have a solution
\begin{align}\label{p:property on M}
(M_{11}(x),~M_{12}(x),~M_{22}(x))\in C_c^{\infty}(B_{r},C^3),\quad \|M_{ij}\|_0\leq C_0(k)\|V\|_0.
\end{align}
Here and subsequent in this section, $C_0$ denotes a absolute constant and $C_0(k)$ is a constant depending on $k$.
In fact, we may choose $M_{12}(x)=0$ first, if both $k_1$ and $k_2 $ are not zero, then (\ref{e:linear equation}) gives $(M_{11}, M_{22})$ and they satisfy (\ref{p:property on M}). In the case one of  $k_1$, $k_2$ is zero, for example $k_1=0$, (\ref{e:linear equation}) gives $(M_{12}, M_{22})$ and we set $M_{11}=0$. They also satisfy (\ref{p:property on M}).
It's direct to obtain
\begin{align}
{\rm div}\mathcal{R}U_o(x)=&\left(\begin{array}{c}
V_1(x)e^{i\lambda k\cdot x}\\
V_2(x)e^{i\lambda k\cdot x}
\end{array}\right)
+\left(\begin{array}{c}
\frac{\partial_1M_{11}(x)+\partial_2M_{12}(x)}{i\lambda}e^{i\lambda k\cdot x}  \\
\frac{\partial_1M_{12}(x)+\partial_2M_{22}(x)}{i\lambda}e^{i\lambda k\cdot x}
\end{array}\right)\nonumber
\end{align}
and
\beno
\|\mathcal{R}U_o\|_0\leq C_0(k)\frac{\|V\|_0}{\lambda} ,\quad \|\nabla\mathcal{R}U_o\|_0\leq  C_0(k)\Big(\|V\|_0+\frac{\|V\|_1}{\lambda}\Big).
\eeno
Repeat the above process, there exists $(N_{11}(x),~N_{12}(x),~N_{22}(x))\in C_c^{\infty}(B_{r},C^3)$ such that if we set
\begin{align}
\mathcal{R}U_{c1}(x):=\left(\begin{array}{cc}
\frac{N_{11}(x)}{(i\lambda)^2}e^{i\lambda k\cdot x}     &   \frac{N_{12}(x)}{(i\lambda)^2}e^{i\lambda k\cdot x}\\
\frac{N_{12}(x)}{(i\lambda)^2}e^{i\lambda k\cdot x}     &   \frac{N_{22}(x)}{(i\lambda)^2}e^{i\lambda k\cdot x}
\end{array}\right),\nonumber
\end{align}
then
\begin{align}
{\rm div}\mathcal{R}U_{c1}(x)=
-\left(\begin{array}{c}
\frac{\partial_1M_{11}(x)+\partial_2M_{12}(x)}{i\lambda}e^{i\lambda k\cdot x}  \\
\frac{\partial_1M_{12}(x)+\partial_2M_{22}(x)}{i\lambda}e^{i\lambda k\cdot x}
\end{array}\right)+
\left(\begin{array}{c}
\frac{\partial_1N_{11}(x)+\partial_2N_{12}(x)}{(i\lambda)^2}e^{i\lambda k\cdot x}  \\
\frac{\partial_1N_{12}(x)+\partial_2N_{22}(x)}{(i\lambda)^2}e^{i\lambda k\cdot x}
\end{array}\right)\nonumber
\end{align}
and
\begin{align}
\|N_{ij}\|_0\leq  C_0(k)\|\nabla M\|_0\leq  C_0(k)\|\nabla V\|_0,\quad \|N_{ij}\|_1\leq   C_0(k)\|\nabla^2 V\|_0.\nonumber
\end{align}
Thus, there holds
\beno
\|\mathcal{R}U_{c1}\|_0\leq  C_0(k)\frac{\|V\|_1}{\lambda^2} ,\quad \|\nabla\mathcal{R}U_{c1}\|_0\leq C_0(k)\Big(\frac{\|V\|_1}{\lambda}+\frac{\|V\|_2}{\lambda^2}\Big).
\eeno
A straightforward computation gives
\begin{align}
{\rm div}\Big(\mathcal{R}U_o(x)+\mathcal{R}U_{c1}(x)\Big)&=\left(\begin{array}{c}
V_1(x)e^{i\lambda k\cdot x}\\
V_2(x)e^{i\lambda k\cdot x}
\end{array}\right)
+\left(\begin{array}{c}
\frac{\partial_1N_{11}(x)+\partial_2N_{12}(x)}{(i\lambda)^2}e^{i\lambda k\cdot x}  \\
\frac{\partial_1N_{12}(x)+\partial_2N_{22}(x)}{(i\lambda)^2}e^{i\lambda k\cdot x}
\end{array}\right)
.\nonumber
\end{align}
Performing the above process, for any integer $m\geq 2$, we have symmetric matrix functions $\mathcal{R}U_{ci}\in C_c^{\infty}(B_r): i=1,2,\cdot\cdot\cdot,m-1$ such that
\begin{align}
\|\mathcal{R}U_{ci}\|_0\leq  C_0(k)\frac{\|V\|_i}{\lambda^{i+1}},\quad \|\nabla\mathcal{R}U_{ci}\|_0\leq C_0(k)\Big(\frac{\|V\|_i}{\lambda^i}+\frac{\|V\|_{i+1}}{\lambda^{i+1}}\Big)\nonumber
\end{align}
and
\begin{align}
{\rm div}\Big(\mathcal{R}U_o(x)+\sum_{i=1}^{m-1}\mathcal{R}U_{ci}(x)\Big)&=\left(\begin{array}{c}
V_1(x)e^{i\lambda k\cdot x}\\
V_2(x)e^{i\lambda k\cdot x}
\end{array}\right)
+\left(\begin{array}{c}
\frac{R_1}{(i\lambda)^m}e^{i\lambda k\cdot x}  \\
\frac{R_2}{(i\lambda)^m}e^{i\lambda k\cdot x}
\end{array}\right)
\nonumber
\end{align}
with
$$\|R_1\|_0+\|R_2\|_0\leq  C_0(k)\|V\|_m,\quad \|\nabla R_1\|_0+\|\nabla R_2\|_0\leq  C_0(k)\|V\|_{m+1}.$$
Since $\mathcal{R}U_o(x)+\sum\limits_{i=1}^{m-1}\mathcal{R}U_{ci}(x)\in C_c^{\infty}(B_r)$ is a symmetric matrix, then
\begin{align}\label{v:vanish 2}
\int_{R^2}{\rm div}\Big(\mathcal{R}U_o(x)+\sum_{i=1}^{m-1}\mathcal{R}U_{ci}(x)\Big)=0, \quad \int_{R^2}(x_iH_j-x_jH_i)(x)dx=0,\quad i,j=1,2.
\end{align}
Here we used the notaion ${\rm div}\Big(\mathcal{R}U_o(x)+\sum\limits_{i=1}^{m-1}\mathcal{R}U_{ci}(x)\Big)=(H_1,H_2)^T$.\\
By (\ref{v:vanish 1}) and (\ref{v:vanish 2}), if we set
$$(K_1, K_2)^T:=\Big(\frac{R_1}{(i\lambda)^m}e^{i\lambda k\cdot x}, \frac{R_2}{(i\lambda)^m}e^{i\lambda k\cdot x}\Big)^T,$$
we also have
\begin{align}
K_i\in C_c^{\infty}(B_{r};C),\quad \int_{R^2}K_i(x)dx=0, \quad \int_{R^2}(x_iK_j-x_jK_i)(x)dx=0,\quad i,j=1,2.\nonumber
\end{align}
Following the argument of Section 10 about solving the symmetric divergence eqaution in \cite{ISOH2}, we know that there exists a symmetric matric function $\delta\mathcal{R}[U]\in C_c^{\infty}(B_{r})$ such that
\begin{align}\label{e:estimate on high error}
&{\rm div}\delta\mathcal{R}[U]=-(K_1,K_2)^T,\quad \|\delta\mathcal{R}[U]\|_0\leq C_0(r,k)\frac{\|V\|_m}{\lambda^m},\nonumber\\
&\|\nabla\delta\mathcal{R}[U]\|_0\leq C_0(r,k)\Big(\frac{\|V\|_{m}}{\lambda^{m-1}}+\frac{\|V\|_{m+1}}{\lambda^m}\Big).
\end{align}
Here and subsequent, $C_0(r)$ denote a constant depend on $r$ linearly: $C_0(r)=C_0r+C_0$.\\
In fact, following \cite{ISOH2}, we take a function $\zeta(y)\in C_c^{\infty}(B_r(0))$ such that
\beno
\|\nabla^\beta \zeta\|_0\leq C_\beta r^{-2-|\beta|},\quad \forall |\beta|\geq 0.
\eeno
Then,  define the solution operator $\delta\mathcal{R}[U]$ as following: Let $(\delta\mathcal{R}[U])^{jl}$ denote the $(j, l)$ element of $\delta\mathcal{R}[U]$ and
\beno
(\delta\mathcal{R}[U])^{jl}:=\mathcal{R}_0^{jl}[U]+\mathcal{R}_1^{jl}[U]+\mathcal{R}_2^{jl}[U],
\eeno
where
\beno
\mathcal{R}_0^{jl}[U]&=&-\int_0^1\int_{B_r(0)}\zeta(y)\frac{(x-y)^j}{\sigma}U^l(\frac{x-y}{\sigma}+y)\frac{dy}{\sigma^2}d\sigma\\
&&-\int_0^1\int_{B_r(0)}\zeta(y)\frac{(x-y)^l}{\sigma}U^j(\frac{x-y}{\sigma}+y)\frac{dy}{\sigma^2}d\sigma,\\
\mathcal{R}_1^{jl}[U]&=&\frac12\int_0^1\int_{B_r(0)}(\partial_p\zeta)(y)\frac{(x-y)^j(x-y)^p}{\sigma^2}U^l(\frac{x-y}{\sigma}+y)\frac{dy}{\sigma^2}d\sigma\\
&&+\frac12\int_0^1\int_{B_r(0)}(\partial_p\zeta)(y)\frac{(x-y)^l(x-y)^p}{\sigma^2}U^j(\frac{x-y}{\sigma}+y)\frac{dy}{\sigma^2}d\sigma,\\
\mathcal{R}_2^{jl}[U]&=&-\int_0^1\int_{B_r(0)}(\partial_p\zeta)(y)\frac{(x-y)^j(x-y)^l}{\sigma^2}U^p(\frac{x-y}{\sigma}+y)\frac{dy}{\sigma^2}d\sigma.\\
\eeno
It's easy to see that $(\delta\mathcal{R}[U])^{jl}$ is symmetric in $(j,l)$, depend linearly on $U$ and from the proof of Proposition 10.1 in \cite{ISOH2}, we know that
\beno
\partial_j(\delta\mathcal{R}[U])^{jl}=U.
\eeno
Moreover, it's obvious that ${\rm supp}\delta\mathcal{R}[U]\subseteq B_r(0)$ because ${\rm supp}U\subseteq B_r(0)$. On the other hand, by following the proof of Lemma 10.3 and Lemma 10.4 in  \cite{ISOH2}, we have
\beno
\|\delta\mathcal{R}[U] \|_0\leq C_0r\|U\|_0,\quad \|\nabla\delta\mathcal{R}[U] \|_0\leq C_0(r+1)\|U\|_1.
\eeno
In fact, Lemma 10.3 and Lemma 10.4 in  \cite{ISOH2} also hold for $U=U(x)$ which satisfies vanishing linear and angular moment with ${\rm supp} U\subseteq B_r(0)$. Thus, we obtain (\ref{e:estimate on high error}).\\
Finally, we set $\mathcal{R}U:=\mathcal{R}U_0+\sum\limits_{i=1}^{m-1}\mathcal{R}U_{ci}(x)+\delta\mathcal{R}U$, then $\mathcal{R}U$ is a symmetric matric function and satisfies
\begin{align}
\mathcal{R}U\in C_c^{\infty}(B_{r}),\quad {\rm div}\mathcal{R}U=U.\nonumber
\end{align}
Moreover, there holds
\beno
\|\mathcal{R}U\|_0\leq C_0(r,k)\Big(\sum\limits_{i=0}^{m-1}\frac{\|V\|_i}{\lambda^{i+1}}+\frac{\|V\|_m}{\lambda^m}\Big),\quad
\|\nabla\mathcal{R}U\|_0\leq C_0(r,k)\Big(\sum\limits_{i=0}^{m}\frac{\|V\|_{i}}{\lambda^{i}}+\frac{\|V\|_{m}}{\lambda^{m-1}}+\frac{\|V\|_{m+1}}{\lambda^m}\Big).
\eeno
In fact,
\begin{align}
&\|\mathcal{R}U_{ci}\|_0\leq   C_0(k)\frac{\|V\|_i}{\lambda^{i+1}},\quad \|\nabla\mathcal{R}U_{ci}\|_0\leq C_0(k)\Big(\frac{\|V\|_{i}}{\lambda^{i}}+\frac{\|V\|_{i+1}}{\lambda^{i+1}}\Big),\quad i=0,\cdot\cdot\cdot,m-1,\nonumber\\
&\|\delta\mathcal{R}U\|_0\leq  C_0(r,k)\frac{\|V\|_m}{\lambda^m},\quad
\|\nabla\delta\mathcal{R}U\|_0\leq C_0(r,k)\Big(\frac{\|V\|_{m}}{\lambda^{m-1}}+\frac{\|V\|_{m+1}}{\lambda^m}\Big).\nonumber
\end{align}
Summing them is what we claimed.\\

Now we introduce a vector space.  Put
 \begin{align}
 \Xi:=&\Big\{U(x):U(x)=(U_1(x),U_2(x))^T\in C_c^{\infty}(B_r; C^2),\int_{R^2}U_i(x)dx=0,\nonumber\\
 &\int_{R^2}(x_iU_j-x_jU_i)(x)dx=0,\quad i,j=1,2 ~~{\quad \rm and \quad }~~
U(x)=\sum_{j=0}^7\left(\begin{array}{c}
U_{1j}(x)e^{i\lambda_j k\cdot x}\\
U_{2j}(x)e^{i\lambda_j k\cdot x}
\end{array}\right)\Big\},\nonumber
\end{align}
where $ k\in R^2\setminus\{0\}$ and $\lambda_j>0, j=0,\cdot\cdot\cdot,7$.

\begin{proposition}\label{p:proposition R}
There exists a linear operator $\mathcal{R}$ from $\Xi$ to $C_c^{\infty}(B_r; \mathcal{S}^{2\times2})$ such that for any $U(x)\in \Xi$ with
 $$U(x)=\sum_{j=0}^7U_j(x):=\sum_{j=0}^7\left(\begin{array}{c}
U_{1j}(x)e^{i\lambda_j k\cdot x}\\
U_{2j}(x)e^{i\lambda_j k\cdot x}
\end{array}\right),$$
 there holds, for any integer $m\geq 2$, that
 \begin{align}\label{p:property on inverse operator}
 &{\rm div}\mathcal{R}U(x)=U(x),\quad \|\mathcal{R}U\|_0\leq C_0(r,k)\sum_{j=0}^7\Big(\sum_{i=0}^{m-1}\frac{\|U_{1j}\|_i+\|U_{2j}\|_i}{\lambda_j^{i+1}}
 +\frac{\|U_{1j}\|_m+\|U_{2j}\|_m}{\lambda_j^m}\Big),\nonumber\\
 &\|\nabla\mathcal{R}U\|_0\leq C_0(r,k)\sum_{j=0}^7\Big(\sum_{i=0}^{m}\frac{\|U_{1j}\|_{i}+\|U_{2j}\|_{i}}{\lambda_j^{i}}+\frac{\|U_{1j}\|_{m}+\|U_{2j}\|_{m}}{\lambda_j^{m-1}}
 +\frac{\|U_{1j}\|_{m+1}+\|U_{2j}\|_{m+1}}{\lambda_j^m}\Big).
 \end{align}

\begin{proof}
We have defined the operator $\mathcal{R}$ on function $U_j(x)=\left(\begin{array}{c}
U_{1j}(x)e^{i\lambda_j k\cdot x}\\
U_{2j}(x)e^{i\lambda_j k\cdot x}
\end{array}\right)$
and
\beno
&&\|\mathcal{R}U_j\|_0\leq C_0(r,k)\Big(\sum_{i=0}^{m-1}\frac{\|U_{1j}\|_i+\|U_{2j}\|_i}{\lambda_j^{i+1}}
 +\frac{\|U_{1j}\|_m+\|U_{2j}\|_m}{\lambda_j^m}\Big),\nonumber\\
&&\|\nabla\mathcal{R}U_j\|_0\leq C_0(r,k)\Big(\sum_{i=0}^{m}\frac{\|U_{1j}\|_{i}+\|U_{2j}\|_{i}}{\lambda_j^{i}}
+\frac{\|U_{1j}\|_{m}+\|U_{2j}\|_{m}}{\lambda_j^{m-1}}
 +\frac{\|U_{1j}\|_{m+1}+\|U_{2j}\|_{m+1}}{\lambda_j^m}\Big).\nonumber
 \eeno
Then, set
\begin{align}
\mathcal{R}U:=\sum_{j=0}^7\mathcal{R}U_j.\nonumber
\end{align}
It is obvious that $\mathcal{R}U \in C_c^{\infty}(B_r; \mathcal{S}^{2\times2})$ and satisfies (\ref{p:property on inverse operator}).
\end{proof}
\end{proposition}

\subsection{The operator $\mathcal{G}$ }

\indent

Let $f(x)=\varphi(x)e^{i\lambda k\cdot x}$ with $\varphi(x)\in C_c^{\infty}(B_r;C)$ and $\int_{R^2}\varphi(x)e^{i\lambda k\cdot x}dx=0$, where $\lambda> 0$ and $k\in R^2\setminus\{0\}$. For any $m\geq 2$, set
$$\mathcal{G}f_o:=\sum_{j=0}^{m-1}\frac{-ik(ik\cdot\nabla)^j\varphi}{(\lambda |k|^2)^{j+1}}e^{i\lambda k\cdot x}.$$
A straightforward computation gives
$${\rm div}(\mathcal{G}f_o)=f-\frac{(ik\cdot\nabla)^m \varphi}{(\lambda|k|^2)^m}e^{i\lambda k\cdot x}$$
and
$$\|\mathcal{G}f_{o}\|_0\leq  C_0(k)\sum_{j=0}^{m-1}\frac{\|\varphi\|_j}{\lambda^{j+1}},\quad \|\nabla\mathcal{G}f_{o}\|_0\leq C_0(k)\sum_{j=0}^{m-1}\Big(\frac{\|\varphi\|_{j}}{\lambda^{j}}+\frac{\|\varphi\|_{j+1}}{\lambda^{j+1}}\Big).$$
Since
\begin{align}
\int_{R^2}f(x)dx=0, \quad \int_{R^2}{\rm div}\big(\mathcal{G}f_o\big)dx=0,\nonumber
\end{align}
therefore
\begin{align}
\int_{R^2}\frac{(ik\cdot\nabla)^m\varphi}{(\lambda|k|^2)^m}e^{i\lambda k\cdot x}dx=0.\nonumber
\end{align}
From \cite{GHG}, we know that there exists $\mathcal{G}f_{c}\in C_c^{\infty}(B_r; C^2)$ such that
\beno
&&{\rm div}(\mathcal{G}f_{c})=\frac{(ik\cdot\nabla)^m \varphi}{(\lambda|k|^2)^m}e^{i\lambda k\cdot x},\quad \|\mathcal{G}f_{cm}\|_0\leq C_0(r, k)\frac{\|\nabla^m\varphi\|_0}{\lambda^m},\\
&&\|\nabla\mathcal{G}f_{cm}\|_0\leq C_0(r,k)\Big(\frac{\|\varphi\|_{m}}{\lambda^{m-1}}+\frac{\|\varphi\|_{m+1}}{\lambda^m}\Big).
\eeno
 In fact, the Bogovskii solution operator are bounded from $W_0^{s,p}$ to $W_0^{s+1,p}$ for any $s> 0, 1<p<\infty,$ thus taking $p=4$ and combining Sobolev embedding, we obtain the above estimates.

Finally, we set
$$\mathcal{G}f:=\mathcal{G}f_{o}+\mathcal{G}f_{c},$$
thus we have
\begin{align}
{\rm div}\mathcal{G}f
=f\nonumber
\end{align}
 and
 \begin{align}
 &\mathcal{G}f\in C_c^{\infty}(B_r;C^2),\quad \|\mathcal{G}f\|_0\leq C_0(r,k)\Big(\sum_{i=0}^{m-1}\frac{\|\varphi\|_i}{\lambda^{i+1}}+\frac{\|\varphi\|_m}{\lambda^{m}}\Big)\nonumber\\
  &\|\nabla\mathcal{G}f\|_0\leq C_0(r,k)\Big(\sum_{i=0}^{m}\frac{\|\varphi\|_{i}}{\lambda^{i}}+\frac{\|\varphi\|_{m}}{\lambda^{m-1}}
  +\frac{\|\varphi\|_{m+1}}{\lambda^{m}}\Big).\nonumber
 \end{align}
In conclusion, we have
 \begin{proposition}\label{p:inverse 2}
 Let the vector space $\Psi$ given by
 \begin{align}
 \Psi:=\Big\{H(x): H(x)=\sum_{j=0}^mH_j(x):=\sum_{j=0}^mb_j(x)e^{i\lambda_j k\cdot x},\quad b_j\in C_c^{\infty}(B_{r};C)\quad {\rm and}\quad \int_{R^2}H_j(x)dx=0\Big\},\nonumber
 \end{align}
 then there exists a linear operator $\mathcal{G}: \Psi\rightarrow C_c^{\infty}(B_{r};C^2)$ such that for any positive integer $m\geq 2$ and any $H(x)=\sum\limits_{j=0}^7b_j(x)e^{i\lambda k\cdot x}\in\Psi$, there holds
 \begin{align}
 &{\rm div}\mathcal{G}(H)(x)=H(x),\quad \|\mathcal{G}(H)\|_0\leq C_0(r,k)\sum_{j=0}^7\sum_{i=0}^{m-1} \Big(\frac{\|b_j\|_i}{\lambda_j^{i+1}}+\frac{\|b_i\|_m}{\lambda_j^m}\Big),\nonumber\\
 &\|\nabla\mathcal{G}(H)\|_0\leq C_0(r,k)\sum_{j=0}^7\sum_{i=0}^{m} \Big(\frac{\|b_j\|_{i}}{\lambda_j^{i}}+\frac{\|b_i\|_{m}}{\lambda_j^{m-1}}
 +\frac{\|b_i\|_{m+1}}{\lambda_j^m}\Big).\nonumber
 \end{align}
\end{proposition}
\begin{proof}
The proof is similar to that of Proposition \ref{p:proposition R}, we omit it here.
\end{proof}

\begin{Corollary}\label{pro:support of inverse operator}
For $f=g(t,x)e^{i\lambda k\cdot x}\in C_c^{\infty}(Q_r)$, we have $\partial_t \mathcal{R}(f)=\mathcal{R}(\partial_tf),\quad \partial_t G(f)=G(\partial_tf)$, because the two operators only act space-variable. On the other hand, there also hold ${\rm supp} \mathcal{R}(f)\subseteq Q_r, {\rm supp} G(f)\subseteq Q_r$, because $\mathcal{R}(0)=0, G(0)=0$ and this can be obtained from the construction of $\mathcal{R}, G$ directly.
\end{Corollary}

\setcounter{equation}{0}

\section{The construction of approximate solutions}

The construction of $\tilde{v},~\tilde{p},~\tilde{\theta},~\tilde{R},~\tilde{g}$ from  $v,~p,~\theta,~R,~g$ consists of several steps. The main idea is to decompose the stress errors into some blocks with the add of geometric lemma and remove one block by constructing new approximate solutions in each step. In this section, we perform the first step.

For convenience, we set $v_0:=v,~p_0:=p,~\theta_0:=\theta,~R_0:=R,~g_0:=g$ and in this section $C_0$ denotes a absolute constant.
\subsection{Partition of unity and Conditions on the parameters}
We first introduce a partition of unity. Following the construction given in \cite{CDL2}, we have the following partition of unity. For some constants $a$ such that $\frac{\sqrt{3}}{2}<a<1$ , we have a family of functions $\alpha_l\in C_c^{\infty}(R^3), l\in Z^3$ such that
\begin{align}\label{p:unity}
\sum\limits_{l\in Z^3}\alpha_l^2=1,\quad \hbox{supp}\alpha_l\subseteq B_{a}(l).
\end{align}
Our construction depends on three parameters: $\ell, \mu_1, \lambda_1$ and we assume they satisfy the following inequalities:
  \begin{align}\label{a:assumption on parameter}
  \mu_1\geq \frac{\Lambda}{\delta}\geq 1,\quad \ell^{-1}\geq\frac{\Lambda}{\eta\delta}\geq1,\quad \lambda_1\geq \max\{\mu_1^{1+\varepsilon},\ell^{-(1+\varepsilon)}\}.
\end{align}
\subsection{Decomposition of stress error}
First, we apply Geometric Lemma \ref{p:split} to obtain $r_{0}>0$ and vectors
\begin{align}
k_1=\Big(\frac{1}{\sqrt{2}},\frac{1}{\sqrt{3}}\Big)^T,\quad k_2=\Big(-\frac{1}{2},\frac{2}{\sqrt{6}}\Big)^T,\quad k_3=\Big(\frac{1}{2},0\Big)^T\nonumber
\end{align}
which are given in the proof of lemma \ref{p:split}, see (\ref{d:definition of ki}),
 together with corresponding functions
\begin{align*}
    \gamma_{k_i}\in C^{\infty}(B_{r_{0}}(Id)),\quad g_{k_i}\in C^{\infty}(R^{2}),\quad  i=1,2,3,
\end{align*}
where $g_{k_3}=0$.
Next, we let $\varphi\in C_c^\infty(R^2\times R)$ be a standard nonnegative radial function and denote the corresponding family of mollifiers by
\begin{align}
\varphi_{\ell}(t,x):=\frac{1}{\ell^3}\varphi\Big(\frac{t}{\ell},\frac{x}{\ell}\Big).\nonumber
\end{align}
Then set
\begin{align}
f_{0\ell}(t,x):=f_0\ast\varphi_\ell(t,x),\qquad
R_{0\ell}(t,x):=R_0\ast\varphi_\ell(t,x).\nonumber
\end{align}
By Lemma \ref{p:split}, we decompose $f_{0}$ as
\begin{align}
f_{0}(t,x)=\sum\limits_{i=1}^{2}g_{k_i}(f_{0})(t,x)k_i:=-\sum\limits_{i=1}^{2}c_i(t,x)k_i.\nonumber
\end{align}
Here we denote $c_i(t,x):=-g_{k_i}(f_{0})(t,x)$.
Thus
\begin{align}\label{p:split 2}
f_{0\ell}(t,x)=-\sum\limits_{i=1}^{2}c_{i\ell}(t,x)k_i.
\end{align}
Since $g_{k_1}$ is linear function,
\begin{align}\label{c:the coffieence of termperture error}
c_{i\ell}(t,x)=-g_{k_i}(f_{0\ell})(t,x).
\end{align}
By (\ref{e:reynold initial 2}), we know that
\begin{align}\label{p:property on cil}
c_{i\ell}(t,x)\in C^{\infty}_c(Q_{r+\ell}),~\|c_i\|_0\leq 2\delta.
\end{align}
Then, we introduce $\rho(t,x)$ as following
\begin{align}\label{d:definition on p}
\rho(t,x)\in C_c^\infty(Q_{r+\delta}),\quad\rho(t,x)=\sqrt{2\delta}~~~ {\rm in}~~~ Q_{r+\frac{\delta}{2}},\quad 0\leq \rho(t,x)\leq \sqrt{2\delta},\quad \|\rho\|_{C^k_{t,x}}\leq C_0(k)\delta^{-\frac{1}{2}-(k-1)}
\end{align}
 and set
\begin{align}\label{d:difinition on e}
e(t,x):=\rho^2(t,x).
\end{align}
By (\ref{e:reynold initial}), parameter assumption (\ref{a:assumption on parameter}), (\ref{d:definition on p}) and (\ref{d:difinition on e}), we have
\begin{align}
\Big\|\frac{R_{0\ell}(t,\cdot)}{e(t,\cdot)}\Big\|_0\leq \frac{\eta}{2}\leq\frac{r_0}{2}\nonumber
\end{align}
if we take $\eta=r_0$, thus
by Lemma \ref{p:split}
\begin{align}\label{d:decomposition reynold}
e(t,x)Id-R_{0\ell}(t,x)=&e(t,x)\Big(Id-\frac{R_{0\ell}(t,x)}{e(t,x)}\Big)=e(t,x)
\sum\limits_{i=1}^{3}\gamma^2_{k_i}\Big(Id-
\frac{R_{0\ell}(t,x)}{e(t,x)}\Big)k_i\otimes k_i\nonumber\\
=&\sum\limits_{i=1}^{3}\Big(\rho(t,x)\gamma_{k_i}\Big(Id-
\frac{R_{0\ell}(t,x)}{e(t,x)}\Big)\Big)^2k_i\otimes k_i
:=\sum\limits_{i=1}^{3}a^2_i(t,x)k_i\otimes k_i.
\end{align}
Where $a_i(t,x)=\rho(t,x)\gamma_{k_i}\Big(Id-\frac{R_{0\ell}(t,x)}{e(t,x)}\Big)$ satisfies
\begin{align}\label{b:bound on ai}
 a_i\in C^{\infty}_c(Q_{r+\delta}),\quad \|a_i\|_0\leq \frac{M\sqrt{\delta}}{300}.
 \end{align}
Here we denote constant $M$ by
\begin{align}\label{d:definition on m}
M:=\max\Big\{C_0,600\max_{1\leq i\leq 3}\|\gamma_{k_i}\|_{L^\infty\big(B_{\frac{r_0}{2}}(Id)\big)}, 600\max_{1\leq i\leq 3}\frac{1}{\min_{B_{\frac{r_0}{2}}(Id)}\gamma_{k_i}}\Big\}.
\end{align}

\subsection{Construction of 1-th perturbation on velocity}

\indent

\subsubsection{Main perturbation on velocity}

For any $l\in Z^3$, we set
   \begin{align}\label{d:definit b}
    b_{1l}(t,x):=&\frac{a_1(t,x)\alpha_l(\mu_1t,\mu_1 x)}{\sqrt{2}},
    \end{align}
then, by (\ref{b:bound on ai}), it's easy to obtain
\begin{align}\label{b:bound on componence 1}
\|b_{1l}\|_0\leq \frac{M\sqrt{\delta}}{300}.
\end{align}
As in \cite{ISV2}, we set $[l]:=\sum_{j=0}^22^j[l_j]$, if $l=(l_0, l_1, l_2)$, where
    \begin{align}
    [l_j]=\left\{
    \begin{array}{cc}
    1, ~~~~l_j~~~{\rm is ~~~even},\\
    0,~~~~ l_j~~~{\rm is~~~ odd}.
    \end{array}
    \right.\nonumber
    \end{align}
Thus, $[l]$ can only take values in $\{0,1,\cdot\cdot\cdot,7\}$.

Now we denote main $l$-perturbation $ w_{1ol}$ by
   \begin{align}\label{d:definition w 1ol}
    w_{1ol}(t,x):=b_{1l}(t,x)k_1\Big(e^{i\lambda_1 2^{[l]} k_1^{\perp}\cdot \big(x-v_{0}(\frac{l}{\mu_1})t\big)}+e^{-i\lambda_1 2^{[l]}k_1^{\perp}\cdot \big(x-v_{0}(\frac{l}{\mu_1})t\big)}\Big).
   \end{align}
Here and subsequent, we denote $a^{\perp}=(-a_2,a_1)^T$ if $a=(a_1,a_2)^T.$\\
Then set 1-th main perturbation
   \begin{align}\label{d:1o}
    w_{1o}:=\sum_{l\in Z^3}w_{1ol}=\sum_{j=0}^7\sum_{[l]=j}b_{1l}k_1\Big(e^{i\lambda_1 2^j k_1^{\perp}\cdot \big(x-v_{0}(\frac{l}{\mu_1})t\big)}+e^{-i\lambda_1 2^jk_1^{\perp}\cdot \big(x-v_{0}(\frac{l}{\mu_1})t\big)}\Big).
   \end{align}
Obviously, $w_{1ol},w_{1o}$ are all real 2-dimensional vector-valued functions.

By (\ref{p:unity}), we have $\hbox{supp} \alpha_l\cap \hbox{supp}\alpha_{l'}=\emptyset$  if $|l-l'|\geq2$, hence there are at most 30 nonzero terms at every point $(t,x)\in R^3$ in the summation (\ref{d:1o}), thus by (\ref{b:bound on componence 1})
   \begin{align}\label{b:bounded 3}
  \|w_{1o}\|_0\leq \frac{M\sqrt{\delta}}{6}.
  \end{align}
Furthermore, if $b_{1l}(t,x)\neq0$, then $|(\mu_1t,\mu_1x)-l|\leq 1$ and $|(t,x)|\leq r+\delta$, thus $|l|\leq C_0(r)\mu_1$.
By (\ref{b:bound on ai}), we know that for any $l\in Z^3$,
  \begin{align}\label{p:support}
  b_{1l}\in C^{\infty}_c(Q_{r+\delta}),\quad w_{1ol}\in C^{\infty}_c(Q_{r+\delta}),\quad w_{1o}\in C^{\infty}_c(Q_{r+\delta}).
\end{align}
 \subsubsection{The correction $w_{1c}$ and the 1-th perturbation $w_{1}$ }

   \indent
We define $l$-correction by
   \begin{align}\label{d:definition w 1ocl}
   w_{1cl}:=&\frac{\nabla^{\perp}b_{1l}}{i\lambda_1 2^{[l]}}\Big(e^{i\lambda_1 2^{[l]} k_1^{\perp}\cdot \big(x-v_{0}(\frac{l}{\mu_1})t\big)}-e^{-i\lambda_1 2^{[l]} k_1^{\perp}\cdot \big(x-v_{0}(\frac{l}{\mu_1})t\big)}\Big)\nonumber\\
   &-\nabla^{\bot}\Big(\frac{\nabla b_{1l}\cdot k_1^{\perp}}{\lambda_1^22^{2[l]}|k_1|^2}\Big(e^{i\lambda_1 2^{[l]} k_1^{\perp}\cdot \big(x-v_{0}(\frac{l}{\mu_1})t\big)}+e^{-i\lambda_1 2^{[l]} k_1^{\perp}\cdot \big(x-v_{0}(\frac{l}{\mu_1})t\big)}\Big)\Big),
   \end{align}
   where $\nabla^{\perp}=(-\partial_{x_2},\partial_{x_1})^T.$

Then 1-th correction is given by
   \begin{align}\label{d:w 1oc}
    w_{1c}:=\sum_{l\in Z^3}w_{1cl}.
    \end{align}
Finally, we denote 1-th perturbation $w_1$ by
   \begin{align}
    w_1:=w_{1o}+w_{1c}.\nonumber
   \end{align}
Thus, if we denote $w_{1l}$ by
   \begin{align}
   w_{1l}:=w_{1ol}+w_{1cl},\nonumber
   \end{align}
   then, we have
   \begin{align}
   w_{1l}=\nabla^{\perp}{\rm div}\Big(-\frac{b_{1l}}{\lambda_1^22^{2[l]}|k_1|^2}k_1^{\perp}
   2\cos\Big(\lambda_12^{[l]}k_1^{\perp}\cdot \big(x-v_{0}(\frac{l}{\mu_1})t\big)\Big)\Big)
   \end{align}
and
   \begin{align}
   w_1=\sum\limits_{l\in Z^3}w_{1l},\quad {\rm div}w_{1l}=0,\quad\int_{R^2}w_{1l}dx=0, \quad \int_{R^2}(x_iw_{1lj}-x_jw_{1li})dx=0, \quad i,j=1,2.\nonumber
   \end{align}
In fact
   \begin{align}
   \int_{R^2}(x_1w_{1l2}-x_2w_{1l1})dx=\int_{R^2}(x_1\partial_1{\rm div}\vec{a}_1+x_2\partial_2{\rm div}\vec{a}_1)dx=-2\int_{R^2}{\rm div}\vec{a}_1dx=0,\nonumber
   \end{align}
where $$\vec{a}_1=-\frac{b_{1l}}{\lambda_1^22^{2[l]}|k_1|^2}k_1^{\perp}
2\cos\Big(\lambda_12^{[l]}k_1^{\perp}\cdot \big(x-v_{0}(\frac{l}{\mu_1})t\big)\Big)$$
is a vector-valued smooth function with compact support.\\
Since there are only finite nonzero terms in the summation (\ref{d:1o}) and (\ref{d:w 1oc}), thus
  \begin{align}\label{p:vanishing moment}
   {\rm div}w_{1}=0,\quad\int_{R^2}w_{1}dx=0, \quad \int_{R^2}(x_iw_{1j}-x_jw_{1i})dx=0, \quad i,j=1,2.
   \end{align}
Now we set
\begin{align}\label{d:difinition on k1l}
k_{1l}:=&b_{1l}k_1+\frac{\nabla^{\perp}b_{1l}}{i\lambda_1 2^{[l]}}-\nabla^{\bot}\Big(\frac{\nabla b_{1l}\cdot k_1^{\perp}}{\lambda_1^22^{2[l]}|k_1|^2}\Big)+\frac{\nabla b_{1l}\cdot k_1^{\perp}}{i\lambda_12^{[l]}|k_1|^2}\cdot k_1,\nonumber\\
k_{-1l}:=&b_{1l}k_1+\frac{\nabla^{\perp}b_{1l}}{-i\lambda_1 2^{[l]}}-\nabla^{\bot}\Big(\frac{\nabla b_{1l}\cdot k_1^{\perp}}{\lambda_1^22^{2[l]}|k_1|^2}\Big)+\frac{\nabla b_{1l}\cdot k_1^{\perp}}{-i\lambda_12^{[l]}|k_1|^2}\cdot k_1,
\end{align}
then
\begin{align}
w_{1l}=k_{1l}e^{i\lambda_1 2^{[l]} k_1^{\perp}\cdot \big(x-v_{0}(\frac{l}{\mu_1})t\big)}+k_{-1l}e^{-i\lambda_1 2^{[l]}k_1^{\perp}\cdot \big(x-v_{0}(\frac{l}{\mu_1})t\big)}\nonumber
\end{align}
and
\begin{align}\label{r:second representati on velocity perturbation}
w_1=\sum_{j=0}^7\sum_{[l]=j}\Big(k_{1l}e^{i\lambda_1 2^j k_1^{\perp}\cdot \big(x-v_{0}(\frac{l}{\mu_1})t\big)}+k_{-1l}e^{-i\lambda_1 2^jk_1^{\perp}\cdot \big(x-v_{0}(\frac{l}{\mu_1})t\big)}\Big).
\end{align}
Finally, it's obvious
   \begin{align}
   w_{1ol},~w_{1cl},~w_{1l},~w_{1o},~w_{1c},~w_1,~k_{1l},~k_{-1l}\in C^{\infty}_c(Q_{r+\delta}).\nonumber
   \end{align}
Thus, we complete the construction of perturbation $w_1$.

\subsection{Construction of 1-th perturbation on temperature}

\indent

To construct $\chi_1$, we first
   denote $\beta_{1l}$ by
   \begin{align}\label{d:definition b}
    \beta_{1l}(t,x):=\frac{c_{1\ell}(t,x)\alpha_l(\mu_1t,\mu_1x)}{\sqrt{2e(t,x)}\gamma_{k_1}
    \Big(Id-\frac{R_{0\ell}(t,x)}{e(t,x)}\Big)}.
    \end{align}
Since ${\rm supp}c_{1\ell}\subseteq {\rm supp}e$, so $\beta_{1l}$ is well-defined.
Then we denote main $l$-perturbation $\chi_{1ol}$ by
    \begin{align}
    \chi_{1ol}(t,x):=&\beta_{1l}(t,x)\Big(e^{i\lambda_1 2^{[l]} k_1^{\perp}\cdot \big(x-v_0(\frac{l}{\mu_1})t\big)}+e^{-i\lambda_1 2^{[l]} k_1^{\perp}\cdot \big(x-v_0(\frac{l}{\mu_1})t\big)}\Big)\nonumber
    \end{align}
and $l$-correction $\chi_{1cl}$ by
    \begin{align}
    \chi_{1cl}(t,x):=&\triangle\beta_{1l}(t,x)\Big(\frac{e^{i\lambda_1 2^{[l]} k_1^{\perp}\cdot \big(x-v_0(\frac{l}{\mu_1})t\big)}}{-\lambda_1^22^{2[l]}|k_1|^2}+\frac{e^{-i\lambda_1 2^{[l]} k_1^{\perp}\cdot \big(x-v_0(\frac{l}{\mu_1})t\big)}}{-\lambda_1^22^{2[l]}|k_1|^2}\Big)\nonumber\\
    &+2\nabla\beta_{1l}(t,x)\cdot\nabla\Big(\frac{e^{i\lambda_1 2^{[l]} k_1^{\perp}\cdot \big(x-v_0(\frac{l}{\mu_1})t\big)}}{-\lambda_1^22^{2[l]}|k_1|^2}+\frac{e^{-i\lambda_1 2^{[l]} k_1^{\perp}\cdot \big(x-v_0(\frac{l}{\mu_1})t\big)}}{-\lambda_1^22^{2[l]}|k_1|^2}\Big).\nonumber
   \end{align}
 Finally, the $l$-th perturbation is given by
 \begin{align}
 \chi_{1l}(t,x):=\chi_{1ol}(t,x)+\chi_{1cl}(t,x)
 =\triangle\Big(\beta_{1l}(x,t)\Big(\frac{e^{i\lambda_1 2^{[l]} k_1^{\perp}\cdot \big(x-v_0(\frac{l}{\mu_1})t\big)}}{-\lambda_1^22^{2[l]}|k_1|^2}+\frac{e^{-i\lambda_1 2^{[l]} k_1^{\perp}\cdot \big(x-v_0(\frac{l}{\mu_1})t\big)}}{-\lambda_1^22^{2[l]}|k_1|^2}\Big)\Big).\nonumber
 \end{align}
Set
   \begin{align}
   \chi_{1o}(t,x):=\sum_{l\in Z^3}\chi_{1ol}(t,x),\quad \chi_{1c}(t,x):=\sum_{l\in Z^3}\chi_{1cl}(t,x),\quad \chi_1(t,x):=\sum_{l\in Z^3}\chi_{1l}(t,x).\nonumber
   \end{align}
   Obviously, $\chi_{1ol}, ~\chi_{1cl}, ~\chi_{1l}$ and $\chi_1$ are all real scalar functions and
   \begin{align}
   \chi_{1o}(t,x)=\sum_{j=0}^7\sum_{[l]=j}\beta_{1l}(t,x)\Big(e^{i\lambda_1 2^j k_1^{\perp}\cdot \big(x-v_0(\frac{l}{\mu_1})t\big)}+e^{-i\lambda_1 2^j k_1^{\perp}\cdot \big(x-v_0(\frac{l}{\mu_1})t\big)}\Big).\nonumber
   \end{align}
Moreover, it's easy to get
   \begin{align}
   \int_{R^2} \chi_{1l}(t,x)dx=0,\qquad \int_{R^2} x_1\chi_{1l}(t,x)dx=0.\nonumber
   \end{align}
Since there are only finite terms in the summation of $\chi_1$, therefore
\begin{align}\label{p:some integration poperty on}
   \int_{R^2} \chi_{1}(t,x)dx=0,\qquad \int_{R^2} x_1\chi_{1}(t,x)dx=0.
   \end{align}
 If set
   \begin{align}\label{d:difinition on h1l}
   h_{1l}:=\beta_{1l}-\frac{\triangle\beta_{1l}}{\lambda_1^22^{2[l]}|k_1|^2}+
   2\frac{\nabla\beta_{1l}\cdot k_1^{\perp}}{i\lambda_1 2^{[l]}|k_1|^2},\quad
    h_{-1l}:=\beta_{1l}-\frac{\triangle\beta_{1l}}{\lambda_1^22^{2[l]}|k_1|^2}+
    2\frac{\nabla\beta_{1l}\cdot k_1^{\perp}}{-i\lambda_1 2^{[l]}|k_1|^2},
   \end{align}
   then
   \begin{align}
   \chi_{1l}=h_{1l}e^{i\lambda_1 2^{[l]} k_1^{\perp}\cdot \big(x-v_0(\frac{l}{\mu_1})t\big)}+h_{-1l}e^{-i\lambda_1 2^{[l]} k_1^{\perp}\cdot \big(x-v_0(\frac{l}{\mu_1})t\big)}\nonumber
   \end{align}
and
   \begin{align}\label{r:another representation on perturbation on temperature}
   \chi_{1}=\sum_{j=0}^7\sum_{[l]=j}\Big(h_{1l}e^{i\lambda_1 2^j k_1^{\perp}\cdot \big(x-v_0(\frac{l}{\mu_1})t\big)}+h_{-1l}e^{-i\lambda_1 2^j k_1^{\perp}\cdot \big(x-v_0(\frac{l}{\mu_1})t\big)}\Big).
   \end{align}
Since $c_{1\ell}(t,x)\in C^{\infty}_c(Q_{r+\delta})$, we know that for all $l\in Z^3$
 \begin{align}
 \beta_{1l}\in C^{\infty}_c(Q_{r+\delta}),\qquad h_{1l}\in C^{\infty}_c(Q_{r+\delta})\nonumber
 \end{align}
and
 \begin{align}
 \chi_{1ol},\quad \chi_{1cl},\quad \chi_{1l},\quad \chi_1\in C^{\infty}_c(Q_{r+\delta}).\nonumber
 \end{align}
Then, by (\ref{p:property on cil}), (\ref{d:definition on p}), (\ref{d:difinition on e}), (\ref{d:definition on m}) and (\ref{d:definition b}), we know that
  \begin{align}
  \|\beta_{1l}\|_0\leq \frac{M\sqrt{\delta}}{200}.\nonumber
  \end{align}
Similar to (\ref{b:bounded 3}), we have
\begin{align}\label{b:bound x}
  \|\chi_{1o}\|_0\leq \frac{M\sqrt{\delta}}{4}.
  \end{align}

\subsection{The construction of  $v_{01}$,~$p_{01}$,~$\theta_{01}$,~$R_{01}$,~$f_{01}$ }

\indent

First, we denote
 $M_1$ by
   \begin{align}
   M_1:=\sum\limits_{l\in Z^3}b^2_{1l}k_1\otimes k_1\Big(e^{2i\lambda_1 2^{[l]} k_1^{\perp}\cdot \big(x-v_0(\frac{l}{\mu_1})t\big)}
   +e^{-2i\lambda_1 2^{[l]} k_1^{\perp}\cdot \big(x-v_0(\frac{l}{\mu_1})t\big)}\Big{)}
   +\sum\limits_{l,l'\in Z^3 ,l\neq l'}w_{1ol}\otimes w_{1ol'}\nonumber
   \end{align}


 and
$N_1,K_1$ by
 \begin{align}
   N_1:=&\sum_{l\in Z^3}\Big[w_{1l}\otimes \Big(v_0-v_0\Big(\frac{l}{\mu_{1}}\Big)\Big)
   +\Big(v_0-v_0\Big(\frac{l}{\mu_{1}}\Big)\Big{)}\otimes w_{1l}\Big]+R_0-R_{0\ell},\nonumber\\
   K_1:=&\sum\limits_{l\in Z^3}\beta_{1l}b_{1l}k_1\Big(e^{2i\lambda_1 2^{[l]} k_1^{\perp}\cdot \big(x-v_0(\frac{l}{\mu_1})t\big)}
   +e^{-2i\lambda_1 2^{[l]} k_1^{\perp}\cdot \big(x-v_0(\frac{l}{\mu_1})t\big)}\Big{)}
   +\sum\limits_{l,l'\in Z^3 ,l\neq l'}w_{1ol}\chi_{1ol'}.\nonumber
    \end{align}
  Then set
 \begin{align}\label{d:definition on approximate solution}
    v_{01}(t,x):=v_0(t,x)+w_1(t,x),\quad
    p_{01}(t,x):=&p_0(t,x)-e(t,x),\quad
    \theta_{01}(t,x):=\theta_0(t,x)+\chi_1(t,x),\nonumber\\
     R_{01}(t,x):=-\bar{R}_{0\ell}(t,x)+&2\sum\limits_{l\in Z^3}b^2_{1l}(t,x)k_1\otimes k_1+\delta R_{01}(t,x),\nonumber\\
    f_{01}(t,x):=f_{0\ell}(t,x)+2&\sum\limits_{l\in Z^3}\beta_{1l}(t,x)b_{1l}(t,x)k_1+\delta f_{01}(t,x),
   \end{align}
   where
   \begin{align}
   \bar{R}_{0\ell}(t,x)=-R_{0\ell}(t,x)+e(t,x)Id,\nonumber
   \end{align}
   \begin{align}\label{d:R0small}
   \delta R_{01}:=&\mathcal{R}(\hbox{div}M_1)+N_1-\mathcal{R}(\chi_1e_2)+
   \mathcal{R}\Big\{\partial_tw_{1}
   +\hbox{div}\Big[\sum_{l\in Z^3}\Big(w_{1l}\otimes v_0\Big(\frac{l}{\mu_{1}}\Big)
   +v_0\Big(\frac{l}{\mu_{1}}\Big)\otimes w_{1l}\Big)\Big]\Big\}\nonumber\\
   &+(w_{1o}\otimes w_{1c}+w_{1c}\otimes w_{1o}+w_{1c}\otimes w_{1c})
   \end{align}
and
   \begin{align}\label{d:f0small}
   \delta f_{01}:=&\mathcal{G}(\hbox{div}K_1)
    +\mathcal{G}\Big(\partial_t\chi_{1}+\sum_{l\in Z^3}v_0\Big(\frac{l}{\mu_{1}}\Big)\cdot\nabla\chi_{1l}\Big)+w_{1o}\chi_{1c}+
    f_0-f_{0\ell}+\sum_{l\in Z^3}\Big(v_0-v_0\Big(\frac{l}{\mu_{1}}\Big)\Big)\chi_{1l}\nonumber\\
    &+w_{1c}\chi_1+\sum_{l\in Z^3}w_{1l}\Big(\theta_0-\theta_0\Big(\frac{l}{\mu_2}\Big)\Big).
  \end{align}
By (\ref{p:vanishing moment}) and (\ref{p:some integration poperty on}), we know
$$\hbox{div}M_1,~\chi_1e_2, ~\partial_tw_{1},~\hbox{div}\Big[\sum_{l\in Z^3}\Big(w_{1l}\otimes v_0\Big(\frac{l}{\mu_{1}}\Big)
   +v_0\Big(\frac{l}{\mu_{1}}\Big)\otimes w_{1l}\Big)\Big]\in\Xi,$$
   so $\delta R_{01}$ is well-defined. Notice that
   $$\hbox{div}K_1,~\partial_t\chi_{1}+\sum_{l\in Z^3}v_0\Big(\frac{l}{\mu_1}\Big)\cdot\nabla\chi_{1l}\in\Psi,$$
    thus $\delta f_{01}$ is also well-defined.
By Proposition \ref{p:proposition R} and Corollary \ref{pro:support of inverse operator}, we know that $\delta R_{01}$ is a symmetric matrix and $\delta R_{01}\in
    C_c^{\infty}(Q_{r+\delta})$. Also, by Proposition \ref{p:inverse 2} and Corollary \ref{pro:support of inverse operator}, we have $\delta f_{01}\in C_c^{\infty}(Q_{r+\delta})$.

   Obviously,
   \begin{align*}
    \hbox{div}v_{01}=\hbox{div}v_0+\hbox{div}w_{1}=0.
   \end{align*}
Moreover, from the definition of $(v_{01}, p_{01}, \theta_{01}, R_{01}, f_{01})$ and the fact that $(v_{0}, p_{0},
    \theta_0, R_{0}, f_{0})$ solves the system (\ref{d:boussinesq reynold}), together with Proposition \ref{p:proposition R} we know that
\[\begin{aligned}
    \hbox{div}R_{01}=&\hbox{div}R_0-\nabla e+\partial_tw_1-\chi_1e_2
    +\hbox{div}(w_{1o}\otimes w_{1o}+w_{1}\otimes v_{0}+v_0\otimes w_{1}\nonumber\\
    &+w_{1o}\otimes w_{1c}+w_{1c}\otimes w_{1o}+w_{1c}\otimes w_{1c})\nonumber\\
    =&\partial_tv_{0}+\hbox{div}(v_{0}\otimes v_{0})+\nabla p_{0}-\theta_0e_2-\nabla e+\partial_tw_1
    -\chi_1e_2
    +\hbox{div}(w_{1o}\otimes w_{1o}\nonumber\\
    &+w_{1}\otimes v_{0}+v_0\otimes w_{1}
    +w_{1o}\otimes w_{1c}+w_{1c}\otimes w_{1o}+w_{1c}\otimes w_{1c})\nonumber\\
    =&\partial_tv_{01}+\hbox{div}(v_{01}\otimes v_{01})+\nabla p_{01}-\theta_{01}e_2.
    \end{aligned}\]
Here we use the fact
$$\hbox{div}(M_1)+\hbox{div}\Big(2\sum\limits_{l\in Z^3}b^2_{1l}(x,t)k_1\otimes k_1)\Big)=\hbox{div}(w_{1o}\otimes w_{1o}).$$
Furthermore, by (\ref{d:definition on approximate solution}) and (\ref{d:f0small}), we have
  \[ \begin{aligned}
    f_{01}=&f_0+2\sum\limits_{l\in Z^3}\beta_{1l}b_{1l}k_1+\mathcal{G}(\hbox{div}K_1)
   +w_{1o}\chi_{1c}
   +\mathcal{G}\Big(\partial_t\chi_{1}+\sum_{l\in Z^3}v_0\Big(\frac{l}{\mu_1}\Big)\cdot\nabla\chi_{1l}\Big)\nonumber\\
   &+\sum_{l\in Z^3}\Big(v_0-v_0\Big(\frac{l}{\mu_1}\Big)\Big)\chi_{1l}+w_{1c}\chi_1+\sum_{l\in Z^3}w_{1l}\Big(\theta_0-\theta_0\Big(\frac{l}{\mu_2}\Big)\Big).\nonumber
   \end{aligned}\]
Thus, by Proposition \ref{p:inverse 2} and the fact that $(v_{0}, p_{0},\theta_0, R_{0}, f_{0})$ solves the system (\ref{d:boussinesq reynold}), we have
 \begin{align}
 \hbox{div}f_{01}=&\hbox{div}f_0+\partial_{t}\chi_1
 +\hbox{div}(w_{1o}\chi_1+w_{1c} \chi_1+v_0 \chi_1+w_1 \theta_0)\nonumber\\
 =&\hbox{div}(v_0\theta_0+w_{1o}\chi_1+w_{1c} \chi_1+v_0 \chi_1+w_1 \theta_0)
 +\partial_t(\theta_0+\chi_1)\nonumber\\
 =&\partial_t\theta_{01}+\hbox{div}(v_{01}\theta_{01}).\nonumber
 \end{align}
Here we use the fact
 \begin{align}
 \hbox{div}K_1+\hbox{div}(2\sum\limits_{l\in Z^3}\beta_{1l}(x,t)b_{1l}(x,t)k_1)=\hbox{div}(w_{1o}\chi_{1o}).\nonumber
 \end{align}
 Thus the new functions $(v_{01},p_{01},\theta_{01},R_{01},f_{01})$ solves the system (\ref{d:boussinesq reynold}).

\setcounter{equation}{0}

\section{The 1-th representations}

\indent

In this section, we will calculate the following two terms
$$-\bar{R}_{0\ell}+2\sum\limits_{l\in Z^3}b^2_{1l}k_1\otimes k_1=I$$
and
$$f_{0\ell}+2\sum\limits_{l\in Z^3}\beta_{1l}b_{1l}k_1=II.$$

\subsection{The term I}

\indent

First, by (\ref{d:definit b}) and the fact $\sum\limits_{l\in Z^3}\alpha_l^2=1$, we have
\begin{align}
2\sum\limits_{l\in Z^3}b^2_{1l}(t,x)k_1\otimes k_1
=\sum\limits_{l\in Z^3}\alpha_l^2(\mu_1t,\mu_1 x) a_1^2(t,x)k_1\otimes k_1
=a_1^2(t,x)k_1\otimes k_1.\nonumber
\end{align}
Thus, by (\ref{d:decomposition reynold}),
\begin{align}
-\bar{R}_{0\ell}(x,t)+2\sum\limits_{l\in Z^3}b^2_{1l}(x,t)k_1\otimes k_1
=-\sum_{i=2}^3 a_i^2(t,x)k_i\otimes k_i.\nonumber
\end{align}
Meanwhile, we have
\begin{align}\label{f:form R1}
R_{01}=-\sum_{i=2}^3 a_i^2(t,x)k_i\otimes k_i+\delta R_{01}.
\end{align}
In next section, we will prove that $\delta R_{01}$ is small.

\subsection{The term II}

\indent

By (\ref{d:definit b}) and (\ref{d:definition b}), we have
\begin{align}
2\sum\limits_{l\in Z^3}\beta_{1l}(t,x)b_{1l}(t,x)k_1
=\sum\limits_{l\in Z^3}\alpha_l^2(\mu_1t,\mu_1 x)c_1(t,x)k_1
=c_1(t,x)k_1.\nonumber
\end{align}
By (\ref{p:split 2}),
\begin{align}
f_{0\ell}+2\sum\limits_{l\in Z^3}\beta_{1l}(x,t)b_{1l}(x,t)k_1=-c_2(t,x)k_2.\nonumber
\end{align}
Meanwhile, we have
\begin{align}\label{f:form g1}
f_{01}(t,x)
=-c_2(t,x)k_2+\delta f_{01}.
\end{align}
Again in next section, we will prove that $\delta f_{01}$ is small.\\
\setcounter{equation}{0}
 \section{Estimates on $\delta R_{01}$ and $\delta f_{01}$}

 \indent
In the subsequent estimates, unless otherwise stated, $C_0$ always denotes an absolute constant which only depends on  $r, \|v_0\|_0$ linearly and $C_m$ will in addition to depend on $m$
and both them can vary from line to line.

In the following, we frequently use the elementary inequality
 \begin{align}
[fg]_{m}\leq& C_m\bigl([f]_{m}\|g\|_0+[g]_{m}\|f\|_0\bigr)\label{i:inequality 1}
\end{align}
for any $m\geq0$.\\
Moreover, by the standard estimates on convolution,
\begin{align}
\|c_{i\ell}\|_{C^m_{t,x}}+\|R_{0\ell}\|_{C^m_{t,x}}\leq& C_m\Lambda \ell^{1-m}\quad {\rm for~~any~~m\geq1},~~~ i=1,2,3.\label{e:estimate higher convolution}\\
\|R_{0\ell}-R_0\|_0+\|f_{0\ell}-f_0\|_0\leq& C_0\Lambda \ell.\label{e:estimate different}
\end{align}
Then, we collect a classical estimates on the H\"{o}lder norms of compositions, their proof can be found in \cite{CDL3}:\\
Let $u: R^n\rightarrow R^N$ and $\Psi: R^N\rightarrow R$ be two smooth functions. Then, for every $m\in {\rm N}\setminus {0}$ there is a constant $C_m=C_m(N,n)$ such that
\begin{align}
[\Psi(u)]_m\leq& C_m\sum\limits_{i=1}^m[\Psi]_i[u]_1^{(i-1)\frac{m}{m-1}}[u]_m^{\frac{m-i}{m-1}}.\label{i:composition inequality 2}
\end{align}
In particular,
$$[\Psi(u)]_1\leq C_1[\Psi]_1[u]_1.$$
We summarize the main estimates on $b_{1l}$ and $\beta_{1l}$ in the following lemma.
\begin{Lemma}\label{l:estimate osc}
For any $l\in Z^3$ and integer $m\geq 1$, we have
\begin{align}
\|b_{1l}\|_m+\|\beta_{1l}\|_m\leq& C_m\sqrt{\delta}(\mu_1^m+\mu_1\ell^{-(m-1)}),\label{e:estimate on main perturbation}\\
\|\partial_tb_{1l}\|_m+\|\partial_t\beta_{1l}\|_m\leq& C_m\sqrt{\delta}(\mu_1^{m+1}+\mu_1\ell^{-m}),\label{e:estimate on time derivative}\\
\|\partial_{tt}b_{1l}\|_m+\|\partial_{tt}\beta_{1l}\|_m\leq& C_m\sqrt{\delta}(\mu_1^{m+2}+\mu_1\ell^{-m-1}),\label{e:estimate on two time order}\\
\|k_{\pm 1l}\|_m+\|h_{\pm 1l}\|_m\leq& C_m\sqrt{\delta}(\mu_1^m+\mu_1\ell^{-(m-1)}),\label{e:estimate on transport amp}\\
\|\partial_tk_{\pm 1l}\|_m+\|\partial_th_{\pm 1l}\|_m\leq& C_m\sqrt{\delta}(\mu_1^{m+1}+\mu_1\ell^{-m}),\label{e:estimate on transport time derivative}\\
\|\partial_{tt}k_{\pm 1l}\|_m+\|\partial_{tt}h_{\pm 1l}\|_m\leq &C_m\sqrt{\delta}(\mu_1^{m+2}+\mu_1\ell^{-m-1})\label{e:estimate on b1l}
\end{align}
and
\begin{align}
\|b_{1l}\|_0+\|\beta_{1l}\|_0\leq& C_0\sqrt{\delta},\label{e:zero estimate on main perturbation}\\
\|\partial_tb_{1l}\|_0+\|\partial_t\beta_{1l}\|_0\leq& C_0\sqrt{\delta}\mu_1,\label{e:zero estimate on time derivative}\\
\|\partial_{tt}b_{1l}\|_0+\|\partial_{tt}\beta_{1l}\|_0\leq& C_0\sqrt{\delta}(\mu_1^{2}+\mu_1\ell^{-1}),\label{e:zero estimate on two time order}\\
\|k_{\pm 1l}\|_0+\|h_{\pm 1l}\|_0\leq& C_0\sqrt{\delta},\label{e:zero estimate on transport amp}\\
\|\partial_tk_{\pm 1l}\|_0+\|\partial_th_{\pm 1l}\|_0\leq& C_0\sqrt{\delta}\mu_1,\label{e:zero estimate on transport time derivative}\\
\|\partial_{tt}k_{\pm 1l}\|_0+\|\partial_{tt}h_{\pm 1l}\|_0\leq &C_0\sqrt{\delta}(\mu_1^{2}+\mu_1\ell^{-1}).\label{e:zero estimate on b1l}
\end{align}
\begin{proof}
First, notice the fact $\{(t,x)|\nabla e\neq0\}\cap \{(t,x)|R_{0\ell}\neq 0\}=\emptyset$, we have, for any positive integer $i$,
$$\nabla^i\Big(\frac{R_{0\ell}}{e}\Big)=\frac{\nabla^i( R_{0\ell})}{e},$$
thus for $m\geq 1$, by (\ref{d:definition on p}), (\ref{d:difinition on e}), (\ref{e:estimate higher convolution}), (\ref{i:composition inequality 2}) and parameter assumption (\ref{a:assumption on parameter})
\begin{align}
\Big[\gamma_{k_1}\Big(Id-\frac{R_{0\ell}}{e}\Big)(t,\cdot)\Big]_m\leq & C_m\sum\limits_{i=1}^m\|\nabla^i\gamma_{k_1}\|_0\Big[Id-\frac{R_{0\ell}}{e}(t,\cdot)\Big]_1^{(i-1)
\frac{m}{m-1}}\Big[Id-\frac{R_{0\ell}}{e}(t,\cdot)\Big]_m^{\frac{m-i}{m-1}}\nonumber\\
\leq & C_m(\mu_1^m+\mu_1\ell^{-(m-1)}).\nonumber
\end{align}
It's obvious that
\begin{align}
\Big\|\gamma_{k_1}\Big(Id-\frac{R_{0\ell}}{e}\Big)\Big\|_0
\leq & C_0.\nonumber
\end{align}
Moreover, by (\ref{d:definition on p}) and parameter assumption (\ref{a:assumption on parameter}), for any integer $m$,
\begin{align}
 \|\rho\|_{C^m_{t,x}}\leq C_m\delta^{-\frac{1}{2}-(m-1)}\leq C_m\sqrt{\delta}\mu_1^m.\nonumber
\end{align}
Then, recalling that
 \begin{align}
    b_{1l}(t,x):=\frac{\rho(t,x)}{\sqrt{2}}\gamma_{k_1}\Big(Id-\frac{R_{0\ell}}{e}\Big)(t,x)\alpha_l(\mu_1t,\mu_1x),\nonumber
   \end{align}
thus, for $m\geq 1$, by (\ref{i:inequality 1}) and parameter assumption (\ref{a:assumption on parameter}),
it's easy to get
\begin{align}
\sup_t\big[b_{1l}(t,\cdot)\big]_m\leq C_m\sqrt{\delta}(\mu_1^m+\mu_1\ell^{-(m-1)}).\nonumber
\end{align}
By (\ref{e:estimate higher convolution}) and parameter assumption (\ref{a:assumption on parameter}), for any integer $m\geq1$,
$$\|c_{1\ell}\|_{C^m_{t,x}}\leq C_m\Lambda\ell^{-(m-1)}\leq C_m\delta\mu_1\ell^{-(m-1)}.$$
By (\ref{p:property on cil}), $\|c_{1\ell}\|_0\leq 2\delta$. By  (\ref{d:definition on p}), (\ref{d:difinition on e}), (\ref{i:composition inequality 2}) and parameter assumption (\ref{a:assumption on parameter}), for any integer $m$,
\begin{align}
 \Big\|\frac{1}{\sqrt{e}}\Big\|_{C^m_{t,x}}\leq C_m\delta^{-\frac{1}{2}-m}\leq C_m\delta^{-\frac{1}{2}}\mu_1^m.\nonumber
\end{align}
Notice that $\frac{1}{2}\leq \gamma_{k_1}\leq \frac{3}{2}$, then, by (\ref{d:definition b}), (\ref{i:inequality 1}) and parameter assumption (\ref{a:assumption on parameter}), we also have
\begin{align}
\|\beta_{1l}\|_m\leq C_m\sqrt{\delta}(\mu_1^m+\mu_1\ell^{-(m-1)}).\nonumber
\end{align}
Thus, we complete the proof of (\ref{e:estimate on main perturbation}). (\ref{e:zero estimate on main perturbation}) is a direct result of (\ref{p:property on cil})-(\ref{d:difinition on e}).\\
By the definition (\ref{d:difinition on k1l}) on $k_{\pm 1l}$, the definition (\ref{d:difinition on h1l}) on $h_{\pm 1l}$, estimate (\ref{e:estimate on main perturbation}) and parameter assumption (\ref{a:assumption on parameter}), for $m\geq 1$, it's easy to obtain
\begin{align}
\|k_{\pm 1l}\|_m\leq C_m\sqrt{\delta}(\mu_1^m+\mu_1\ell^{-(m-1)}),\quad
\|h_{\pm 1l}\|_m\leq C_m\sqrt{\delta}(\mu_1^m+\mu_1\ell^{-(m-1)}).\nonumber
\end{align}
Thus, we complete the proof of (\ref{e:estimate on transport amp}). The case of $m=0$ in (\ref{e:zero estimate on transport amp}) is also direct.\\
We introduce function
$$\Gamma(t,x)=\frac{\rho(t,x)}{\sqrt{2}}\gamma_{k_1}\Big(Id-\frac{R_{0\ell}}{e}\Big)(t,x).$$
By (\ref{d:definition on p}), (\ref{d:difinition on e}) and parameter assumption (\ref{a:assumption on parameter})  and notice the fact
$$\partial_t\Big(\frac{R_{0\ell}}{e}\Big)=\frac{(\partial_t R_0)_{\ell}}{e}-\frac{R_{0\ell}\partial_te}{e^2}=\frac{(\partial_t R_0)_{\ell}}{e},\quad
\partial_{tt}\Big(\frac{R_{0\ell}}{e}\Big)=\frac{\partial_t R_0\ast(\partial_t\varphi)_{\ell}\ell^{-1}}{e}$$ we have
\begin{align}
\|\partial_t\Gamma\|_0\leq C_0\sqrt{\delta}\mu_1,\quad
\|\partial_{tt}\Gamma\|_0\leq C_0\sqrt{\delta}(\mu_1^2+\mu_1\ell^{-1}).\nonumber
\end{align}
Moreover, by (\ref{d:definition on p}), (\ref{i:inequality 1}), (\ref{e:estimate higher convolution}), (\ref{i:composition inequality 2}) and parameter assumption (\ref{a:assumption on parameter}), for $m\geq 1$ we have
\begin{align}
\|\partial_t\Gamma\|_m\leq C_m\sqrt{\delta}(\mu_1^{m+1}+\mu_1\ell^{-m}),\quad
\|\partial_{tt}\Gamma\|_m\leq C_m\sqrt{\delta}(\mu_1^{m+2}+\mu_1\ell^{-m-1}).\nonumber
\end{align}
Observe that
$$b_{1l}(t,x)=\Gamma(t,x)\alpha_l(\mu_1t, \mu_1x),$$
thus
\begin{align}
\partial_tb_{1l}=&\partial_t\Gamma \alpha_l(\mu_1t, \mu_1x)+\mu_1\Gamma (\partial_t\alpha)_l(\mu_1t,\mu_1x),\nonumber\\
\partial_{tt}b_{1l}=&\partial_{tt}\Gamma \alpha_l(\mu_1t, \mu_1x)+2\mu_1\partial_t\Gamma (\partial_t\alpha)_l(\mu_1t, \mu_1x)+
\mu_1^2\Gamma (\partial_{tt}\alpha)_l(\mu_1t,\mu_1x).\nonumber
\end{align}
Hence, by (\ref{i:inequality 1}) and the above estimate on $\Gamma$, we obtain
\begin{align}\label{e: infinity estimate}
\|\partial_tb_{1l}\|_0\leq C_0\sqrt{\delta}\mu_1,\quad
\|\partial_{tt}b_{1l}\|_0\leq C_0\sqrt{\delta}(\mu_1^2+\mu_1\ell^{-1})
\end{align}
and
\begin{align}
\|\partial_tb_{1l}\|_m\leq C_m\sqrt{\delta}(\mu_1\ell^{-m}+\mu_1^{m+1}),\quad
\|\partial_{tt}b_{1l}\|_m\leq C_m\sqrt{\delta}(\mu_1\ell^{-m-1}+\mu_1^{m+2}).\nonumber
\end{align}
The same argument gives
\begin{align}
\|\partial_t\beta_{1l}\|_0\leq& C_0\sqrt{\delta}\mu_1,\quad\quad\quad
\|\partial_{tt}\beta_{1l}\|_0\leq C_0\sqrt{\delta}(\mu_1^2+\mu_1\ell^{-1}),\label{e: infinity bil}\\
\|\partial_t\beta_{1l}\|_m\leq& C_m\sqrt{\delta}(\mu_1\ell^{-m}+\mu_1^{m+1}),\quad
\|\partial_{tt}\beta_{1l}\|_m\leq C_m\sqrt{\delta}(\mu_1\ell^{-m-1}+\mu_1^{m+2}).\nonumber
\end{align}
Thus, we obtain (\ref{e:estimate on time derivative}), (\ref{e:estimate on two time order}), (\ref{e:zero estimate on time derivative}) and (\ref{e:zero estimate on two time order}).
Then, by the definition (\ref{d:difinition on k1l}) on $k_{\pm 1l}$, the definition (\ref{d:difinition on h1l}) on $h_{\pm 1l}$ and parameter assumption (\ref{a:assumption on parameter}), it's easy to obtain (\ref{e:estimate on transport time derivative}), (\ref{e:estimate on b1l}), (\ref{e:zero estimate on transport time derivative}) and (\ref{e:zero estimate on b1l}).
Thus, the proof of this lemma is complete.
\end{proof}
\end{Lemma}
Next, we give estimates on perturbations $w_{1o}, w_{1c}, \chi_{1o}, \chi_{1c}$.
 \begin{Lemma}[Estimate on main perturbation and correction]\label{e:estimate on perturbation and correction}
 \begin{align}
 \|w_{1o}\|_0\leq &C_0\sqrt{\delta},\quad  \|w_{1o}\|_{C^1_{t,x}}\leq C_0\sqrt{\delta}\lambda_1,\quad  \|\chi_{1o}\|_0\leq C_0\sqrt{\delta},\quad  \|\chi_{1o}\|_{C^1_{t,x}}\leq C_0\sqrt{\delta}\lambda_1,\label{e:estimate main pertuebations}\\
 \|w_{1c}\|_0\leq& C_0\frac{\sqrt{\delta}\mu_1}{\lambda_1},\quad  \|w_{1c}\|_{C^1_{t,x}}\leq C_0\sqrt{\delta}\mu_1,\quad
\|\chi_{1c}\|_0\leq C_0\frac{\sqrt{\delta}\mu_1}{\lambda_1},\quad  \|\chi_{1c}\|_{C^1_{t,x}}\leq C_0\sqrt{\delta}\mu_1.\label{e:estimate on correction}
\end{align}
 \begin{proof}
  First, by (\ref{b:bounded 3}) and (\ref{b:bound x}), we know $\|w_{1o}\|_0\leq C_0\sqrt{\delta}, \quad\|\chi_{1o}\|_0\leq C_0\sqrt{\delta}$. Since
 \begin{align}
 \partial_tw_{1o}=&\sum_{l\in Z^3}\partial_tb_{1l}k_1\Big(e^{i\lambda_1 2^{[l]} k_1^{\perp}\cdot \big((y,z)-v_0(\frac{l}{\mu_1})t\big)}+e^{-i\lambda_1 2^{[l]} k_1^{\perp}\cdot \big((y,z)-v_0(\frac{l}{\mu_1})t\big)}\Big)\nonumber\\
 &-\sum_{l\in Z^3}b_{1l}i\lambda_1 2^{[l]} k_1^{\perp}\cdot v_0\Big(\frac{l}{\mu_1}\Big)k_1\Big(e^{i\lambda_1 2^{[l]} k_1^{\perp}\cdot \big((y,z)-v_0(\frac{l}{\mu_1})t\big)}-e^{-i\lambda_1 2^{[l]} k_1^{\perp}\cdot \big((y,z)-v_0(\frac{l}{\mu_1})t\big)}\Big),\nonumber
 \end{align}
 thus, by (\ref{e:zero estimate on main perturbation}), (\ref{e:zero estimate on time derivative}) and parameter assumption (\ref{a:assumption on parameter}), we obtain
 \begin{align}
 \|\partial_tw_{1o}\|_0\leq C_0\sqrt{\delta}\lambda_1.\nonumber
\end{align}
The same argument gives
\begin{align}
 \|\nabla w_{1o}\|_0\leq C_0\sqrt{\delta}\lambda_1,\quad  \|\chi_{1o}\|_{C^1_{t,x}}\leq C_0\sqrt{\delta}\lambda_1.\nonumber
\end{align}
Thus, we give a proof of (\ref{e:estimate main pertuebations}).\\
Next, by (\ref{d:definition w 1ocl}), (\ref{e:estimate on main perturbation}), parameter assumption (\ref{a:assumption on parameter}), we get
 \begin{align}
 \|w_{1cl}\|_0
 \leq C_0\frac{\sqrt{\delta}\mu_1}{\lambda_1}.\nonumber
 \end{align}
 Thus, from the property (\ref{p:unity}) on $\alpha_l$, we arrive at
 \begin{align}
 \|w_{1c}\|_0\leq C_0\frac{\sqrt{\delta}\mu_1}{\lambda_1}.\nonumber
 \end{align}
Differentiating (\ref{d:definition w 1ocl}) in time
\begin{align}
 \partial_tw_{1cl}=&\sum_{l\in Z^3}\Big\{\frac{\nabla^{\perp}\partial_tb_{1l}}{i\lambda_1 2^{[l]}}\Big(e^{i\lambda_1 2^{[l]} k_1^{\perp}\cdot \big(x-v_{0}(\frac{l}{\mu_1})t\big)}-e^{-i\lambda_1 2^{[l]} k_1^{\perp}\cdot \big(x-v_{0}(\frac{l}{\mu_1})t\big)}\Big)\nonumber\\
 &-\nabla^{\perp}b_{1l}k_1^{\perp}\cdot v_{0}\Big(\frac{l}{\mu_1}\Big)\Big(e^{i\lambda_1 2^{[l]} k_1^{\perp}\cdot \big(x-v_{0}(\frac{l}{\mu_1})t\big)}+e^{-i\lambda_1 2^{[l]} k_1^{\perp}\cdot \big(x-v_{0}(\frac{l}{\mu_1})t\big)}\Big)\nonumber\\
   &-\nabla^{\bot}\Big(\frac{\nabla \partial_tb_{1l}\cdot k_1^{\perp}}{\lambda_1^22^{2[l]}|k_1|^2}\Big(e^{i\lambda_1 2^{[l]} k_1^{\perp}\cdot \big(x-v_{0}(\frac{l}{\mu_1})t\big)}+e^{-i\lambda_1 2^{[l]} k_1^{\perp}\cdot \big(x-v_{0}(\frac{l}{\mu_1})t\big)}\Big)\nonumber\\
   &-\nabla^{\bot}\Big(\frac{\nabla b_{1l}\cdot k_1^{\perp}}{i\lambda_12^{[l]}|k_1|^2}k_1^{\perp}\cdot v_{0}\Big(\frac{l}{\mu_1}\Big)\Big(e^{i\lambda_1 2^{[l]} k_1^{\perp}\cdot \big(x-v_{0}(\frac{l}{\mu_1})t\big)}-e^{-i\lambda_1 2^{[l]} k_1^{\perp}\cdot \big(x-v_{0}(\frac{l}{\mu_1})t\big)}\Big)\Big\}.\nonumber
 \end{align}
   By (\ref{e:estimate on main perturbation}), (\ref{e:estimate on time derivative}) and parameter assumption (\ref{a:assumption on parameter}), we get
   \begin{align}
 \|\partial_tw_{1c}\|_0\leq C_0\sqrt{\delta}\mu_1.\nonumber
 \end{align}
 Similarly, we have
 \begin{align}
 \|\nabla w_{1c}\|_0\leq C_0\sqrt{\delta}\mu_1,\quad  \|\chi_{1c}\|_{C^1_{t,x}}\leq C_0\sqrt{\delta}\mu_1.\nonumber
 \end{align}
 Collect the above estimates, we complete the proof of (\ref{e:estimate on correction}).
 \end{proof}
 \end{Lemma}

By (\ref{d:definition on p}), (\ref{d:difinition on e}), (\ref{b:bounded 3}), (\ref{b:bound x}), (\ref{d:definition on approximate solution}) and lemma \ref{e:estimate on perturbation and correction}, it's easy to obtain the following estimate:
\begin{Corollary}
\begin{align}\label{e:differerce first estimate}
\|v_{01}-v_0\|_0\leq \frac{M\sqrt{\delta}}{6}+C_0\frac{\sqrt{\delta}\mu_1}{\lambda_1},\quad
\|p_{01}-p_0\|_0\leq& M\delta,\quad\|\theta_{01}-\theta_0\|_0\leq \frac{M\sqrt{\delta}}{4}+C_0\frac{\sqrt{\delta}\mu_1}{\lambda_1},\nonumber\\
\|v_{01}-v_0\|_{C^1_{t,x}}\leq C_0\lambda_1\sqrt{\delta},\qquad
\|p_{01}-p_0\|_{C^1_{t,x}}\leq& C_0,\quad \|\theta_{01}-\theta_0\|_{C^1_{t,x}}\leq C_0\lambda_1\sqrt{\delta}.
\end{align}
\end{Corollary}

 \subsection{Estimates on $\delta R_{01}$}

 \indent

 Recalling that
  \begin{align}
   \delta R_{01}=&\mathcal{R}(\hbox{div}M_1)+N_1-\mathcal{R}(\chi_1e_2)+\mathcal{R}\Big\{\partial_tw_{1}
   +\hbox{div}\Big[\sum_{l\in Z^3}\Big(w_{1l}\otimes v_0\Big(\frac{l}{\mu_{1}}\Big)
   +v_0\Big(\frac{l}{\mu_{1}}\Big)\otimes w_{1l}\Big)\Big]\Big\}\nonumber\\
   &+(w_{1o}\otimes w_{1c}+w_{1c}\otimes w_{1o}+w_{1c}\otimes w_{1c}).\nonumber
   \end{align}
  We split the stress into three parts: \\
  (1)The oscillation part $$\mathcal{R}(\hbox{div}M_1)-\mathcal{R}
   (\chi_1e_2).$$
  (2)The transport part
  \begin{align}
  \mathcal{R}\Big\{\partial_tw_{1}
   +\hbox{div}\Big[\sum_{l\in Z^3}\Big(w_{1l}\otimes v_0\Big(\frac{l}{\mu_{1}}\Big)+v_0\Big(\frac{l}{\mu_{1}}\Big)\otimes w_{1l}\Big)\Big]\Big\}
   =\mathcal{R}\Big(\partial_tw_{1}+\sum_{l\in Z^3}v_0\Big(\frac{l}{\mu_{1}}\Big)\cdot\nabla w_{1l}\Big).\nonumber
   \end{align}
  (3)The error part
  \begin{align}
  N_1+(w_{1o}\otimes w_{1c}+w_{1c}\otimes w_{1o}+w_{1c}\otimes w_{1c}).\nonumber
   \end{align}
In the following we will estimate each term separately.

\begin{Lemma}[The oscillation part]\label{e:oscillate estimate}
\begin{align}
\|\mathcal{R}({\rm div}M_1)\|_0\leq& C_0(\varepsilon)\frac{\delta\mu_1}{\lambda_1},\quad\|\mathcal{R}({\rm div}M_1)\|_{C^1_{t,x}}\leq C_0(\varepsilon)\delta\mu_1.\label{e:oscillate estimate 1}\\
\|\mathcal{R}(\chi_1e_2)\|_0\leq& C_0(\varepsilon)\frac{\sqrt{\delta}}{\lambda_1},\quad\|\mathcal{R}(\chi_1e_2)\|_{C^1_{t,x}}\leq C_0(\varepsilon)\sqrt{\delta}.\label{e:oscillate estimate 2}
\end{align}
\begin{proof}
First, we have
\begin{align}
M_{1}=\sum_{j=0}^7\sum_{[l]=j}k_1\otimes k_1\Big(e^{2i\lambda_1 2^j k_1^{\perp}\cdot \big(x-v_0(\frac{l}{\mu_1})t\big)}
   +e^{-2i\lambda_1 2^j k_1^{\perp}\cdot \big(x-v_0(\frac{l}{\mu_1})t\big)}\Big{)}b^2_{1l}
   +\sum\limits_{l,l'\in Z^3 ,l\neq l'}w_{1ol}\otimes w_{1ol'}.\nonumber
\end{align}
Since $k_1\cdot k_1^{\perp}=0,$ then
\begin{align}
\hbox{div}M_1=M_{11}+M_{12}.\nonumber
\end{align}
where
\begin{align}
M_{11}=&\sum_{j=0}^7\sum_{[l]=j}k_1\otimes k_1\nabla(b^2_{1l})\Big(e^{2i\lambda_1 2^j k_1^{\perp}\cdot \big(x-v_0(\frac{l}{\mu_1})t\big)}
   +e^{-2i\lambda_1 2^j k_1^{\perp}\cdot \big(x-v_0(\frac{l}{\mu_1})t\big)}\Big),\nonumber\\
M_{12}=&\sum\limits_{l,l'\in Z^3 ,l\neq l'}\hbox{div}(w_{1ol}\otimes w_{1ol'}).\nonumber
\end{align}
 By (\ref{e:estimate on main perturbation}),  (\ref{e:zero estimate on main perturbation}), Proposition \ref{p:proposition R} with $m=\Big[1+\frac{1}{\varepsilon}\Big]+1$ and parameter assumption (\ref{a:assumption on parameter}) , we have
\begin{align}
\|\mathcal{R}(M_{11})\|_0
\leq& C_m\sum\limits_{j=0}^7\Big(\sum_{i=0}^{m-1}\frac{\|\sum_{[l]=j}\nabla(b^2_{1l})\|_i}{(\lambda_12^j)^{i+1}}
+\frac{\|\sum_{[l]=j}\nabla(b^2_{1l})\|_m}{(\lambda_12^j)^{m}}\Big)\nonumber\\
\leq& C_m\sum\limits_{j=0}^7\delta\Big(\frac{\mu_1}{\lambda_12^j}+\frac{\mu_1^{m+1}+\mu_1\ell^{-m}}{(\lambda_12^j)^m}\Big)
\leq C_m\frac{\delta\mu_1}{\lambda_1}.\nonumber
\end{align}
where we use the fact: $b_{1l}b_{1l'}= 0$ if $|l-l'|\geq 2$.\\
Notice
\begin{align}
M_{12}=&\sum\limits_{j=0}^7\sum_{[l]=j}\sum\limits_{l'\in Z^3, 1\leq|l'-l|<2}k_1\otimes k_1\nabla(b_{1l}b_{1l'})\Big(e^{i\lambda_1(2^j+2^{[l']})k_1^{\perp}\cdot x-ig_{1,l,l'}(t)}
+e^{i\lambda_1(2^j-2^{[l']})k_1^{\perp}\cdot x-i\overline{g}_{1,l,l'}(t)}\nonumber\\
&+e^{i\lambda_1(2^{[l']}-2^j)k_1^{\perp}\cdot x+i\overline{g}_{1,l,l'}(t)}
+e^{-i\lambda_1(2^j+2^{[l']})k_1^{\perp}\cdot x+ig_{1,l,l'}(t)}\Big),\nonumber
\end{align}
where
$$g_{1,l,l'}(t)=\lambda_1\Big(2^{[l]}k_1^{\perp}\cdot v_0\Big(\frac{l}{\mu_1}\Big)t+2^{[l']}k_1^{\perp}\cdot v_0\Big(\frac{l'}{\mu_1}\Big)t\Big),~~~~~
\overline{g}_{1,l,l'}(t)=\lambda_1\Big(2^{[l]}k_1^{\perp}\cdot v_0\Big(\frac{l}{\mu_1}\Big)-2^{[l']}k_1^{\perp}\cdot v_0\Big(\frac{l'}{\mu_1}\Big)\Big).$$
Following the same strategy as $M_{11}$, we deduce
\begin{align}\label{e:estimate stragety for m12}
\|\mathcal{R}(M_{12})\|_0
\leq& C_m\sum\limits_{j=0}^7\Big(\sum_{i=0}^{m-1}\frac{\Big\|\sum_{[l]=j}\sum\limits_{l'\in Z^3, 1\leq|l'-l|<2} \nabla(b_{1l}b_{1l'})\Big\|_i}{\lambda_1^{i+1}}
+\frac{\Big\|\sum_{[l]=j}\sum\limits_{l'\in Z^3, 1\leq|l'-l|<2}\nabla(b_{1l}b_{1l'})\Big\|_m}{\lambda_1^{m}}\Big)\nonumber\\
\leq& C_m\sum\limits_{j=0}^7\delta\Big(\frac{\mu_1}{\lambda_1}+\frac{\mu_1^{m+1}+\mu_1\ell^{-m}}{\lambda_1^m}\Big)
\leq C_m\frac{\delta\mu_1}{\lambda_1}.
\end{align}
Thus, the first estimate of (\ref{e:oscillate estimate 1}) follows easily.
A direct calculation gives
\begin{align}
\partial_tM_{11}=&\sum_{j=0}^7\sum_{[l]=j}k_1\otimes k_1\nabla\partial_t(b^2_{1l})\Big(e^{2i\lambda_1 2^j k_1^{\perp}\cdot \big(x-v_0(\frac{l}{\mu_1})t\big)}
   +e^{-2i\lambda_1 2^j k_1^{\perp}\cdot \big(x-v_0(\frac{l}{\mu_1})t\big)}\Big)\nonumber\\
   &-\sum_{j=0}^7\sum_{[l]=j}k_1\otimes k_1\nabla(b^2_{1l})2i\lambda_1 2^j k_1^{\perp}\cdot v_0\Big(\frac{l}{\mu_1}\Big)\Big(e^{2i\lambda_1 2^j k_1^{\perp}\cdot \big(x-v_0(\frac{l}{\mu_1})t\big)}
   -e^{-2i\lambda_1 2^j k_1^{\perp}\cdot \big(x-v_0(\frac{l}{\mu_1})t\big)}\Big).\nonumber
\end{align}
Thus, by (\ref{e:estimate on main perturbation}), (\ref{e:estimate on time derivative}), (\ref{e:zero estimate on main perturbation}), Proposition \ref{p:proposition R} with $m=\Big[1+\frac{1}{\varepsilon}\Big]+1$ and parameter assumption (\ref{a:assumption on parameter}), we have
\begin{align}
\|\partial_t\mathcal{R}(M_{11})\|_0
\leq& C_m\sum\limits_{j=0}^7\Big(\sum_{i=0}^{m-1}\frac{\|\sum_{[l]=j}\nabla\partial_t(b^2_{1l})\|_i}{(\lambda_12^j)^{i+1}}
+\frac{\|\sum_{[l]=j}\nabla\partial_t(b^2_{1l})\|_m}{(\lambda_12^j)^{m}}\Big)\nonumber\\
&+ C_m\lambda_1\sum\limits_{j=0}^7\Big(\sum_{i=0}^{m-1}\frac{\|\sum_{[l]=j}\nabla(b^2_{1l})\|_i}{(\lambda_12^j)^{i+1}}
+\frac{\|\sum_{[l]=j}\nabla(b^2_{1l})\|_m}{(\lambda_12^j)^{m}}\Big)\nonumber\\
\leq& C_m\lambda_1\sum\limits_{j=0}^7\delta\Big(\frac{\mu_1}{\lambda_12^j}+\frac{\mu_1^{m+1}+\mu_1\ell^{-m}}{(\lambda_12^j)^m}\Big)
\leq C_m\delta\mu_1.\nonumber
\end{align}
By the same argument, we have
\begin{align}
\|\partial_t\mathcal{R}(M_{12})\|_0
\leq C_m\delta\mu_1.\nonumber
\end{align}
By Proposition \ref{p:proposition R}, a similar computation as above gives
\begin{align}
\|\nabla\mathcal{R}(M_{11})\|_0
\leq C_m\delta\mu_1,\quad
\|\nabla\mathcal{R}(M_{12})\|_0
\leq C_m\delta\mu_1.\nonumber
\end{align}
Putting these estimates together, we obtain the second estimate of (\ref{e:oscillate estimate 1}).\\
Recalling (\ref{r:another representation on perturbation on temperature})
 \begin{align}
   \chi_{1}=\sum_{j=0}^7\sum_{[l]=j}\Big(h_{1l}e^{i\lambda_1 2^j k_1^{\perp}\cdot \big(x-v_0(\frac{l}{\mu_1})t\big)}+h_{-1l}e^{-i\lambda_1 2^j k_1^{\perp}\cdot \big(x-v_0(\frac{l}{\mu_1})t\big)}\Big).\nonumber
   \end{align}
 By  (\ref{e:estimate on transport amp}), (\ref{e:estimate on transport time derivative}), (\ref{e:zero estimate on transport amp}), (\ref{e:zero estimate on transport time derivative}) and using a similar argument as above, we obtain
\begin{align}
\big\|\mathcal{R}(\chi_1e_2)\big\|_0\leq C_m\frac{\sqrt{\delta}}{\lambda_1},\quad \big\|\mathcal{R}(\chi_1e_2)\big\|_{C^1_{t,x}}
\leq C_m\sqrt{\delta}.\nonumber
\end{align}
Thus, the proof of this lemma is complete.
\end{proof}
\end{Lemma}

\begin{Lemma}[The transport part]\label{e:trans estimate}
\begin{align}
\Big\|\mathcal{R}\Big(\partial_tw_{1}+\sum_{l\in Z^3}v_0\Big(\frac{l}{\mu_{1}}\Big)\cdot\nabla w_{1l}\Big)\Big\|_0\leq C_0(\varepsilon)\frac{\sqrt{\delta}\mu_1}{\lambda_1},\nonumber\\
\Big\|\mathcal{R}\Big(\partial_tw_{1}+\sum_{l\in Z^3}v_0\Big(\frac{l}{\mu_{1}}\Big)\cdot\nabla w_{1l}\Big)\Big\|_{C^1_{t,x}}\leq C_0(\varepsilon)\sqrt{\delta}\mu_1.
\end{align}
\begin{proof}
Recalling (\ref{r:second representati on velocity perturbation})
 \begin{align}
w_1=\sum_{j=0}^7\sum_{[l]=j}\Big(k_{1l}e^{i\lambda_1 2^j k_1^{\perp}\cdot \big(x-v_{0}(\frac{l}{\mu_1})t\big)}+k_{-1l}e^{-i\lambda_1 2^jk_1^{\perp}\cdot \big(x-v_{0}(\frac{l}{\mu_1})t\big)}\Big).\nonumber
\end{align}
Thus, using the identity
$$\Big(\partial_t+v_0\Big(\frac{l}{\mu_{1}}\Big)\cdot\nabla\Big)e^{\pm i\lambda_1 2^{[l]} k_1^{\perp}\cdot (x-v_{0}(\frac{l}{\mu_1})t\big)}=0,$$
we deduce
\begin{align}
\partial_tw_{1}+\sum_{l\in Z^3}v_{0}\Big(\frac{l}{\mu_1}\Big)\cdot\nabla w_{1l}
=&\sum_{j=0}^7\sum_{[l]=j}\Big(\Big(\partial_t
+v_{0}\Big(\frac{l}{\mu_1}\Big)\cdot\nabla\Big)k_{1l}e^{i\lambda_1 2^j k_1^{\perp}\cdot (x-v_{0}(\frac{l}{\mu_1})t\big)}\nonumber\\
&+\Big(\partial_t+v_0\Big(\frac{l}{\mu_1}\Big)\cdot\nabla\Big)k_{-1l}e^{-i\lambda_1 2^j k_1^{\perp}\cdot (x-v_{0}(\frac{l}{\mu_1})t\big)}\Big{)}.\nonumber
   \end{align}
By Proposition \ref{p:proposition R} with $m=\Big[1+\frac{1}{\varepsilon}\Big]+1$ , (\ref{e:estimate on transport amp}), (\ref{e:estimate on transport time derivative}), (\ref{e:zero estimate on transport time derivative}) and parameter assumption(\ref{a:assumption on parameter}), we have
\begin{align}\label{e:estimate method for velocity transport term}
&\Big\|\mathcal{R}\Big(\partial_tw_{1}+\sum_{l\in Z^3}v_0\Big(\frac{l}{\mu_{1}}\Big)\cdot\nabla w_{1l}\Big)\Big\|_0\nonumber\\
\leq& C_m\sum\limits_{j=0}^7\Big(\sum_{i=0}^{m-1}\frac{\Big\|\sum_{[l]=j}\Big(\partial_t
+v_{0}\Big(\frac{l}{\mu_1}\Big)\cdot\nabla\Big)k_{1l}\Big\|_i}{(\lambda_12^j)^{i+1}}
+\frac{\Big\|\sum_{[l]=j}\Big(\partial_t
+v_{0}\Big(\frac{l}{\mu_1}\Big)\cdot\nabla\Big)k_{1l}\Big\|_m}{(\lambda_12^j)^{m}}\Big)\nonumber\\
&+C_m\sum\limits_{j=0}^7\Big(\sum_{i=0}^{m-1}\frac{\Big\|\sum_{[l]=j}\Big(\partial_t
+v_{0}\Big(\frac{l}{\mu_1}\Big)\cdot\nabla\Big)k_{-1l}\Big\|_i}{(\lambda_12^j)^{i+1}}
+\frac{\Big\|\sum_{[l]=j}\Big(\partial_t
+v_{0}\Big(\frac{l}{\mu_1}\Big)\cdot\nabla\Big)k_{-1l}\Big\|_m}{(\lambda_12^j)^{m}}\Big)\nonumber\\
\leq& C_m\sum\limits_{j=0}^7\sqrt{\delta}\Big(\frac{\mu_1}{\lambda_12^j}+\frac{\mu_1^{m+1}+\mu_1\ell^{-m}}{(\lambda_12^j)^m}\Big)
\leq C_m\frac{\sqrt{\delta}\mu_1}{\lambda_1}.
\end{align}
A direct calculation gives
\begin{align}
\partial_t\Big(\partial_tw_{1}+\sum_{l\in Z^3}v_{0}\Big(\frac{l}{\mu_1}\Big)\cdot\nabla w_{1l}\Big)\nonumber
=&\sum_{j=0}^7\sum_{[l]=j}\Big(\Big(\partial_t
+v_{0}\Big(\frac{l}{\mu_1}\Big)\cdot\nabla\Big)\partial_tk_{1l}e^{i\lambda_1 2^j k_1^{\perp}\cdot \big(x-v_{0}(\frac{l}{\mu_1})t\big)}\nonumber\\
&+\Big(\partial_t
+v_0\Big(\frac{l}{\mu_1}\Big)\cdot\nabla\Big)\partial_tk_{-1l}e^{-i\lambda_1 2^j k_1^{\perp}\cdot \big(x-v_{0}(\frac{l}{\mu_1})t\big)}\nonumber\\
&-\Big(\partial_t+v_{0}\Big(\frac{l}{\mu_1}\Big)\cdot\nabla\Big)k_{1l}i\lambda_1 2^j k_1^{\perp}\cdot v_{0}\Big(\frac{l}{\mu_1}\Big)e^{i\lambda_1 2^j k_1^{\perp}\cdot \big(x-v_{0}(\frac{l}{\mu_1})t\big)}\nonumber\\
&+\Big(\partial_t+v_{0}\Big(\frac{l}{\mu_1}\Big)\cdot\nabla\Big)k_{-1l}i\lambda_1 2^j k_1^{\perp}\cdot v_{0}\Big(\frac{l}{\mu_1}\Big)e^{i\lambda_1 2^j k_1^{\perp}\cdot \big(x-v_{0}(\frac{l}{\mu_1})t\big)}\Big),\nonumber
\end{align}
thus, by (\ref{e:estimate on transport amp})-(\ref{e:estimate on b1l}), (\ref{e:zero estimate on transport time derivative})-(\ref{e:zero estimate on b1l})  and applying the same argument as above, we arrive at
\begin{align}
\Big\|\partial_t\mathcal{R}\Big(\partial_tw_{1}+\sum_{l\in Z^3}v_{0}\Big(\frac{l}{\mu_1}\Big)\cdot\nabla w_{1l}\Big)\Big\|_0
\leq C_m\sqrt{\delta}\mu_1.\nonumber
\end{align}
By Proposition \ref{p:proposition R}, a similar computation as above gives
\begin{align}
\Big\|\nabla\mathcal{R}\Big(\partial_tw_{1}+\sum_{l\in Z^3}v_{0}\Big(\frac{l}{\mu_1}\Big)\cdot\nabla w_{1l}\Big)\Big\|_0
\leq C_m\sqrt{\delta}\mu_1.\nonumber
\end{align}
Then we proved the Lemma \ref{e:trans estimate}.
\end{proof}
\end{Lemma}
\begin{Lemma}[Estimates on error part I]\label{e:err 1}
\begin{align}
\|N_1\|_0\leq C_0\Big(\sqrt{\delta}\frac{\Lambda}{\mu_1}+\Lambda\ell\Big),\qquad\|N_1\|_{C^1_{t,x}}\leq  C_0\lambda_1\Big(\sqrt{\delta}\frac{\Lambda}{\mu_1}+\Lambda\ell\Big).
\end{align}
\end{Lemma}
\begin{proof}
We may rewrite $N_{1}$ as
\begin{align}
 N_1=N_{11}+N_{12},\nonumber
 \end{align}
 where
 \begin{align}
 N_{11}=\sum_{l\in Z^3}\Big[w_{1l}\otimes \Big(v_0-v_0\Big(\frac{l}{\mu_{1}}\Big)\Big)
   +\Big(v_0-v_0\Big(\frac{l}{\mu_{1}}\Big)\Big{)}\otimes w_{1l}\Big],\quad
N_{12}=R_0-R_{0\ell}.\nonumber
\end{align}
For the term $N_{11}$, by (\ref{r:second representati on velocity perturbation}), we have
\begin{align}
&\sum_{l\in Z^3}w_{1l}\otimes \Big(v_0-v_0\Big(\frac{l}{\mu_{1}}\Big)\Big)\nonumber\\
=&\sum_{l\in Z^3}\Big(k_{1l}\otimes \Big(v_0-v_0\Big(\frac{l}{\mu_{1}}\Big)\Big)e^{i\lambda_1 2^{[l]} k_1^{\perp}\cdot \big(x-v_{0}(\frac{l}{\mu_1})t\big)}+k_{-1l}\otimes \Big(v_0-v_0\Big(\frac{l}{\mu_{1}}\Big)\Big)e^{-i\lambda_1 2^{[l]}k_1^{\perp}\cdot \big(x-v_{0}(\frac{l}{\mu_1})t\big)}\Big).\nonumber
\end{align}
Obviously, $k_{1l}(x,t)\neq0$ implies $|(\mu_1t,\mu_1 x)-l|\leq1$, therefore, by (\ref{e:zero estimate on transport amp})
$$\Big|k_{1l}\otimes\Big(v_0-v_0\Big(\frac{l}{\mu_{1}}\Big)\Big)\Big|\leq C_0\sqrt{\delta}\frac{\|\nabla v_0\|_0}{\mu_1}\leq C_0\sqrt{\delta}\frac{\Lambda}{\mu_1}.$$
By (\ref{p:unity}), it's easy to get
\begin{align}
\Big\|\sum_{l\in Z^3}k_{1l}\otimes \Big(v_0-v_{0}\Big(\frac{l}{\mu_1}\Big)\Big)\Big\|_0\leq C_0\sqrt{\delta}\frac{\Lambda}{\mu_1}.\nonumber
\end{align}
Similarly,
\begin{align}
\Big\|\sum_{l\in Z^3} k_{-1l}\otimes\Big(v_0-v_{0}\Big(\frac{l}{\mu_1}\Big)\Big)\Big\|_0\leq C_0\sqrt{\delta}\frac{\Lambda}{\mu_1}.\nonumber
\end{align}
Thus,
\begin{align}\label{e:first part estimate}
\Big\|\sum_{l\in Z^3} w_{1l}\otimes\Big(v_0-v_{0}\Big(\frac{l}{\mu_1}\Big)\Big)\Big\|_0\leq C_0\sqrt{\delta}\frac{\Lambda}{\mu_1}.
\end{align}
Following the same strategy:
\begin{align}\label{e:second part estimate}
\Big\|\sum_{l\in Z^3}\Big(v_0-v_{0}\Big(\frac{l}{\mu_1}\Big)\Big)\otimes w_{1l}\Big\|_0\leq C_0\sqrt{\delta}\frac{\Lambda}{\mu_1}.
\end{align}
Finally, putting (\ref{e:first part estimate}) and (\ref{e:second part estimate}) together, we arrive at
\begin{align}\label{e:final estimate on first part}
\|N_{11}\|_0\leq C_0\sqrt{\delta}\frac{\Lambda}{\mu_1}.
\end{align}
Moreover, we have
\begin{align}
&\partial_t\Big(\sum_{l\in Z^3}w_{1l}\otimes \Big(v_0-v_0\Big(\frac{l}{\mu_{1}}\Big)\Big)\Big)\nonumber\\
=&\sum_{l\in Z^3}\Big(\partial_tk_{1l}e^{i\lambda_1 2^{[l]} k_1^{\perp}\cdot \big(x-v_{0}(\frac{l}{\mu_1})t\big)}+\partial_tk_{-1l}e^{-i\lambda_1 2^{[l]}k_1^{\perp}\cdot \big(x-v_{0}(\frac{l}{\mu_1})t\big)}\Big{)}\otimes \Big(v_0-v_0\Big(\frac{l}{\mu_{1}}\Big)\Big)\nonumber\\
&+\sum_{l\in Z^3}\Big(k_{1l}e^{i\lambda_1 2^{[l]} k_1^{\perp}\cdot \big(x-v_{0}(\frac{l}{\mu_1})t\big)}+k_{-1l}e^{-i\lambda_1 2^{[l]}k_1^{\perp}\cdot \big(x-v_{0}(\frac{l}{\mu_1})t\big)}\Big{)}\otimes \partial_tv_0\nonumber\\
&+\sum_{l\in Z^3}\Big(-k_{1l}i\lambda_1 2^{[l]} k_1^{\perp}\cdot v_{0}\Big(\frac{l}{\mu_1}\Big)e^{i\lambda_1 2^{[l]} k_1^{\perp}\cdot \big(x-v_{0}(\frac{l}{\mu_1})t\big)}\nonumber\\
&+k_{-1l}i\lambda_1 2^{[l]} k_1^{\perp}\cdot v_{0}\Big(\frac{l}{\mu_1}\Big)e^{-i\lambda_1 2^{[l]}k_1^{\perp}\cdot \big(x-v_{0}(\frac{l}{\mu_1})t\big)}\Big{)}\otimes \Big(v_0-v_0\Big(\frac{l}{\mu_{1}}\Big)\Big),\nonumber
\end{align}
thus, by (\ref{e:zero estimate on transport amp}), (\ref{e:zero estimate on transport time derivative}) and parameter assumption(\ref{a:assumption on parameter})
\begin{align}
\Big\|\partial_t\Big(\sum_{l\in Z^3}w_{1l}\otimes \Big(v_0-v_{0}\Big(\frac{l}{\mu_1}\Big)\Big)\Big)\Big\|_0\leq C_0\sqrt{\delta}\lambda_1\frac{\Lambda}{\mu_1}.\nonumber
\end{align}
Similarly, we have
\begin{align}
\Big\|\partial_t\Big(\sum_{l\in Z^3}\Big(v_0-v_{0}\Big(\frac{l}{\mu_1}\Big)\Big)\otimes w_{1l}\Big)\Big\|_0\leq C_0\sqrt{\delta}\lambda_1\frac{\Lambda}{\mu_1}.\nonumber
\end{align}
Therefore,
\begin{align}\label{e:time derivative estimate on first part}
\|\partial_tN_{11}\|_0\leq C_0\sqrt{\delta}\lambda_1\frac{\Lambda}{\mu_1}.
\end{align}
By a similar argument, we have
\begin{align}\label{e:space derivative estimate on first part}
\|\nabla N_{11}\|_0\leq C_0\sqrt{\delta}\lambda_1\frac{\Lambda}{\mu_1}.
\end{align}
By (\ref{e:estimate different}), we have
\begin{align}
\|R_0-R_{0\ell}\|_0\leq C_0\Lambda\ell,\quad\|\partial_t(R_0-R_{0\ell})\|_0\leq C_0\Lambda,\quad\|\nabla(R_0-R_{0,\ell})\|_0\leq C_0\Lambda.\nonumber
\end{align}
Thus, by parameter assumption(\ref{a:assumption on parameter}), we arrive at
\begin{align}\label{e:final estimate on second part}
\|N_{12}\|_0\leq C_0\Lambda\ell,\quad \|N_{12}\|_{C^1_{t,x}}\leq C_0\lambda_1\Lambda\ell.
\end{align}
Collecting (\ref{e:final estimate on first part}), (\ref{e:time derivative estimate on first part}), (\ref{e:space derivative estimate on first part}) and
(\ref{e:final estimate on second part}), we obtain
\begin{align}
\|N_1\|_0\leq  C_0\Big(\sqrt{\delta}\frac{\Lambda}{\mu_1}+\Lambda\ell\Big),\qquad\|N_1\|_{C^1_{t,x}}\leq  C_0\lambda_1\Big(\sqrt{\delta}\frac{\Lambda}{\mu_1}+\Lambda\ell\Big).\nonumber
\end{align}
We complete our proof of this lemma.
\end{proof}

\begin{Lemma}[Estimates on error part II]\label{e: err 2}
\begin{align}
\|w_{1o}\otimes w_{1c}+w_{1c}\otimes w_{1o}+w_{1c}\otimes w_{1c}\|_0\leq C_0\frac{\delta\mu_1}{\lambda_1},\nonumber\\
\|w_{1o}\otimes w_{1c}+w_{1c}\otimes w_{1o}+w_{1c}\otimes w_{1c}\|_{C^1_{t,x}}\leq C_0\delta\mu_1.
\end{align}
\begin{proof}
By (\ref{e:estimate main pertuebations}) and (\ref{e:estimate on correction}), we have
\begin{align}
\|w_{1o}\otimes w_{1c}+w_{1c}\otimes w_{1o}+w_{1c}\otimes w_{1c}\|_0
\leq C_0(\|w_{1o}\|_0\|w_{1c}\|_0+\|w_{1c}\|_0^2)
\leq C_0\frac{\delta\mu_1}{\lambda_1}\nonumber
\end{align}
and
\begin{align}
\|w_{1o}\otimes w_{1c}+w_{1c}\otimes w_{1o}+w_{1c}\otimes w_{1c}\|_{C^1_{t,x}}\leq C_0(\|w_{1o}\|_0\|w_{1c}\|_{C^1_{t,x}}+\|w_{1o}\|_{C^1_{t,x}}\|w_{1c}\|_0)\leq C_0\delta\mu_1.\nonumber
\end{align}
thus, we complete the proof of this lemma.
\end{proof}
\end{Lemma}

Finally, from  Lemma \ref{e:oscillate estimate}, Lemma \ref{e:trans estimate}, Lemma \ref{e:err 1} and Lemma \ref{e: err 2}, we conclude
\begin{align}\label{e:estimate on first velocity error}
\|\delta R_{01}\|_0\leq C_0(\varepsilon)\Big(\frac{\sqrt{\delta}\mu_1}{\lambda_1}+\sqrt{\delta}\frac{\Lambda}{\mu_1}+\Lambda\ell\Big),\quad
\|\delta R_{01}\|_{C^1_{t,x}}\leq C_0(\varepsilon)\lambda_1\Big(\frac{\sqrt{\delta}\mu_1}{\lambda_1}+\sqrt{\delta}\frac{\Lambda}{\mu_1}+\Lambda\ell\Big).
\end{align}

\subsection{Estimates on $\delta f_{01}$}

 \indent

Recalling that
\begin{align}
   \delta f_{01}=&\mathcal{G}(\hbox{div}K_1)
    +\mathcal{G}\Big(\partial_t\chi_{1}+\sum_{l\in Z^3}v_0\Big(\frac{l}{\mu_{1}}\Big)\cdot\nabla\chi_{1l}\Big)
    +w_{1o}\chi_{1c}+f_0-f_{0\ell}\nonumber\\
   &+\sum_{l\in Z^3}\Big(v_0-v_0\Big(\frac{l}{\mu_{1}}\Big)\Big)\chi_{1l}+w_{1c}\chi_1+\sum_{l\in Z^3}w_{1l}\Big(\theta_0-\theta_0\Big(\frac{l}{\mu_1}\Big)\Big).\nonumber
\end{align}
As before, we split $\delta f_{01}$ into three parts:\\
(1)The oscillation part:
\begin{align}
\mathcal{G}(\hbox{div}K_1).\nonumber
\end{align}
(2)The transport part:
\begin{align}
\mathcal{G}\Big(\partial_t\chi_{1}+\sum_{l\in Z^3}v_0\Big(\frac{l}{\mu_{1}}\Big)\cdot\nabla\chi_{1l}\Big).\nonumber
\end{align}
(3)The error part:
\begin{align}
w_{1c}\chi_1+w_{1o}\chi_{1c}+f_0-f_{0\ell}+\sum\limits_{l\in Z^3}\Big(v_0-v_0\Big(\frac{l}{\mu_{1}}\Big)\Big)\chi_{1l}+\sum_{l\in Z^3}w_{1l}\Big(\theta_0-\theta_0\Big(\frac{l}{\mu_2}\Big)\Big).\nonumber
\end{align}

\begin{Lemma}[The oscillation part]\label{e:osci}
\begin{align}
\|\mathcal{G}({\rm div}K_1)\|_0\leq C_0(\varepsilon)\frac{\delta\mu_1}{\lambda_1},\qquad\|\mathcal{G}({\rm div}K_1)\|_{C^1_{t,x}}\leq C_0(\varepsilon)\delta\mu_1.
\end{align}
\begin{proof}
Recalling the notations of $K_1$ and $[l]$, we have
\begin{align}
K_1=\sum_{j=0}^7\sum_{[l]=j}\beta_{1l}b_{1l}k_1\Big(e^{i\lambda_1 2^j k_1^{\perp}\cdot \big(x-v_{0}(\frac{l}{\mu_1})t\big)}+e^{-i\lambda_1 2^jk_1^{\perp}\cdot \big(x-v_{0}(\frac{l}{\mu_1})t\big)}\Big{)}\Big)
+\sum\limits_{l,l'\in Z^3 ,l\neq l'}w_{1ol}\chi_{1ol'},\nonumber
\end{align}
thus
\begin{align}
\hbox{div}K_1=K_{11}+K_{12},\nonumber
\end{align}
where
\begin{align}
K_{11}=&\sum_{j=0}^7\sum_{[l]=j}\nabla(\beta_{1l}b_{1l})\cdot k_1\Big(e^{i\lambda_1 2^{|l|} k_1^{\perp}\cdot \big(x-v_{0}(\frac{l}{\mu_1})t\big)}+e^{-i\lambda_1 2^{|l|}k_1^{\perp}\cdot \big(x-v_{0}(\frac{l}{\mu_1})t\big)}\Big{)},\nonumber\\
K_{12}=&\sum\limits_{l,l'\in Z^3 ,l\neq l'}\hbox{div}(w_{1ol}\chi_{1ol'}).\nonumber
  \end{align}
By (\ref{e:estimate on main perturbation}),  (\ref{e:zero estimate on main perturbation}), Proposition (\ref{p:inverse 2}) with $m=\Big[1+\frac{1}{\varepsilon}\Big]+1$ and parameter assumption (\ref{a:assumption on parameter}) , we have
\begin{align}
\|\mathcal{R}(K_{11})\|_0
\leq& C_m\sum\limits_{j=0}^7\Big(\sum_{i=0}^{m-1}\frac{\|\sum_{[l]=j}\nabla(\beta_{1l}b_{1l})\|_i}{(\lambda_12^j)^{i+1}}
+\frac{\|\sum_{[l]=j}\nabla(\beta_{1l}b_{1l})\|_m}{(\lambda_12^j)^{m}}\Big)\nonumber\\
\leq& C_m\sum\limits_{j=0}^7\delta\Big(\frac{\mu_1}{\lambda_12^j}+\frac{\mu_1^{m+1}+\mu_1\ell^{-m}}{(\lambda_12^j)^m}\Big)
\leq C_m\frac{\delta\mu_1}{\lambda_1}.\nonumber
\end{align}
Moreover, we have
\begin{align}
K_{12}=&\sum\limits_{j=0}^7\sum_{[l]=j}\sum\limits_{l'\in Z^3,1\leq |l'-l|<2}k_1\cdot\nabla(b_{1l}\beta_{1l'})\Big(e^{i\lambda_1(2^j+2^{[l']})k_1^{\perp}\cdot x-ig_{1,l,l'}(t)}
+e^{i\lambda_1(2^j-2^{[l']})k_1^{\perp}\cdot x-i\overline{g}_{1,l,l'}(t)}\nonumber\\
&+e^{i\lambda_1(2^{[l']}-2^{j})k_1^{\perp}\cdot x+i\overline{g}_{1,l,l'}(t)}
+e^{-i\lambda_1(2^j+2^{[l']})k_1^{\perp}\cdot x+ig_{1,l,l'}(t)}\Big),\nonumber
\end{align}
as estimate (\ref{e:estimate stragety for m12}) on $M_{12}$, by (\ref{e:estimate on main perturbation}),  (\ref{e:zero estimate on main perturbation}) and Proposition (\ref{p:inverse 2}), we have
\begin{align}
\|\mathcal{G}(K_{12})\|_0
\leq C_m\frac{\delta\mu_1}{\lambda_1}.\nonumber
\end{align}
A straightforward computation gives
\begin{align}
\partial_tK_{11}=&\sum_{j=0}^7\sum_{[l]=j}\nabla\partial_t(\beta_{1l}b_{1l})\cdot k_1\Big(e^{i\lambda_1 2^j k_1^{\perp}\cdot \big(x-v_{0}(\frac{l}{\mu_1})t\big)}+e^{-i\lambda_1 2^jk_1^{\perp}\cdot \big(x-v_{0}(\frac{l}{\mu_1})t\big)}\Big)\nonumber\\
&-\sum_{j=0}^7\sum_{[l]=j}\nabla(\beta_{1l}b_{1l})\cdot k_{1}i\lambda_1 2^j k_1^{\perp}\cdot v_{0}\Big(\frac{l}{\mu_1}\Big)\Big(e^{i\lambda_1 2^j k_1^{\perp}\cdot \big(x-v_{0}(\frac{l}{\mu_1})t\big)}
-e^{-i\lambda_1 2^jk_1^{\perp}\cdot \big(x-v_{0}(\frac{l}{\mu_1})t\big)}\Big{)},\nonumber
\end{align}
thus, by the same argument
\begin{align}
\|\partial_t\mathcal{G}(K_{11})\|_0\leq C_m\delta\mu_1.\nonumber
\end{align}
A similar argument give $\|\partial_t\mathcal{G}(K_{12})\|_0\leq C_m\delta\mu_1$.\\
By Proposition (\ref{p:inverse 2}), we deduce that
\begin{align}
\|\nabla\mathcal{G}(K_{11})\|_0\leq C_m\delta\mu_1,\qquad\|\nabla\mathcal{G}( K_{12})\|_0\leq C_m\delta\mu_1.\nonumber
\end{align}
We complete the proof of this lemma.
\end{proof}
\end{Lemma}

\begin{Lemma}[The transport part]\label{e:trans}
\begin{align}
\Big\|\mathcal{G}\Big(\partial_t\chi_{1}+\sum_{l\in Z^3}v_0\Big(\frac{l}{\mu_{1}}\Big)\cdot\nabla\chi_{1l}\Big)\Big\|_0\leq C_0(\varepsilon) \frac{\sqrt{\delta}\mu_1}{\lambda_1},\nonumber\\
\Big\|\mathcal{G}\Big(\partial_t\chi_{1}+\sum_{l\in Z^3}v_0\Big(\frac{l}{\mu_{1}}\Big)\cdot\nabla\chi_{1l}\Big)\Big\|_{C^1_{t,x}}\leq C_0(\varepsilon)\sqrt{\delta} \mu_1.
\end{align}
\begin{proof}
Recalling the notation of $\chi_1$, we have
\begin{align}
&\partial_t\chi_{1}+\sum_{l\in Z^3}v_0\Big(\frac{l}{\mu_{1}}\Big)\cdot\nabla\chi_{1l}\nonumber\\
=&\sum_{j=0}^7\sum_{[l]=j}\Big\{\Big(\partial_t+v_0\Big(\frac{l}{\mu_{1}}\Big)\cdot\nabla\Big)h_{1l}e^{i\lambda_1 2^j k_1^{\perp}\cdot \big(x-v_{0}(\frac{l}{\mu_1})t\big)}+\Big(\partial_t+v_0\Big(\frac{l}{\mu_{1}}\Big)\cdot\nabla\Big)h_{-1l}e^{-i\lambda_1 2^j k_1^{\perp}\cdot \big(x-v_{0}(\frac{l}{\mu_1})t\big)}\Big\}.\nonumber
\end{align}
By Proposition \ref{p:inverse 2} with $m=\Big[1+\frac{1}{\varepsilon}\Big]+1$ , (\ref{e:estimate on transport amp}), (\ref{e:estimate on transport time derivative}), (\ref{e:zero estimate on transport time derivative}) and parameter assumption(\ref{a:assumption on parameter}), as estimate in (\ref{e:estimate method for velocity transport term}) , we have
\begin{align}
\Big\|\mathcal{G}\Big(\partial_t\chi_{1}+\sum_{l\in Z^3}v_0\Big(\frac{l}{\mu_{1}}\Big)\cdot\nabla\chi_{1l}\Big)\Big\|_0\leq  C_m\sum\limits_{j=0}^7\sqrt{\delta}\Big(\frac{\mu_1}{\lambda_12^j}+\frac{\mu_1^{m+1}+\mu_1\ell^{-m}}{(\lambda_12^j)^m}\Big)
\leq C_m\frac{\sqrt{\delta}\mu_1}{\lambda_1}.\nonumber
\end{align}
And since
\begin{align}
\partial_t\Big(\partial_t\chi_{1}+\sum_{l\in Z^3}v_{0}\Big(\frac{l}{\mu_1}\Big)\cdot\nabla \chi_{1l}\Big)\nonumber
=&\sum_{j=0}^7\sum_{[l]=j}\Big(\Big(\partial_t+v_{0}\Big(\frac{l}{\mu_1}\Big)\cdot\nabla\Big)\partial_th_{1l}e^{i\lambda_1 2^j k_1^{\perp}\cdot \big(x-v_{0}(\frac{l}{\mu_1})t\big)}\nonumber\\
&+\Big(\partial_t+v_0\Big(\frac{l}{\mu_1}\Big)\cdot\nabla\Big)\partial_th_{-1l}e^{-i\lambda_1 2^j k_1^{\perp}\cdot \big(x-v_{0}(\frac{l}{\mu_1})t\big)}\nonumber\\
&-\Big(\partial_t+v_{0}\Big(\frac{l}{\mu_1}\Big)\cdot\nabla\Big)h_{1l}i\lambda_1 2^j k_1^{\perp}\cdot v_{0}\Big(\frac{l}{\mu_1}\Big)e^{i\lambda_1 2^j k_1^{\perp}\cdot \big(x-v_{0}(\frac{l}{\mu_1})t\big)}\nonumber\\
&+\Big(\partial_t+v_{0}\Big(\frac{l}{\mu_1}\Big)\cdot\nabla\Big)h_{-1l}i\lambda_1 2^j k_1^{\perp}\cdot v_{0}\Big(\frac{l}{\mu_1}\Big)e^{i\lambda_1 2^j k_1^{\perp}\cdot \big(x-v_{0}(\frac{l}{\mu_1})t\big)}\Big),\nonumber
\end{align}
then, by the same argument, we obtain
\begin{align}
\Big\|\partial_t\mathcal{G}\Big(\partial_t\chi_{1}+\sum_{l\in Z^3}v_{0}\Big(\frac{l}{\mu_1}\Big)\cdot\nabla \chi_{1l}\Big)\Big\|_0
\leq C_m\sqrt{\delta}\mu_1.\nonumber
\end{align}
Similarly, by Proposition (\ref{p:inverse 2}), we have
\begin{align}
\Big\|\nabla\mathcal{G}\Big(\partial_t\chi_{1}+\sum_{l\in Z^3}v_{0}\Big(\frac{l}{\mu_1}\Big)\cdot\nabla \chi_{1l}\Big)\Big\|_0
\leq C_m\sqrt{\delta}\mu_1.\nonumber
\end{align}
Then we complete our proof of this Lemma.
\end{proof}
\end{Lemma}

\begin{Lemma}[The error part]\label{e:est 1}
\begin{align}
&\Big\|w_{1c}\chi_1+w_{1o}\chi_{1c}+f_0-f_{0\ell}+\sum_{l\in Z^3}\Big(v_0-v_0\Big(\frac{l}{\mu_{1}}\Big)\Big)\chi_{1l}
+\sum_{l\in Z^3}w_{1l}\Big(\theta_0-\theta_0\Big(\frac{l}{\mu_1}\Big)\Big)\Big\|_0\nonumber\\
\leq& C_0\Big(\frac{\delta\mu_1}{\lambda_1}+\sqrt{\delta}\frac{\Lambda}{\mu_1}+\Lambda\ell\Big),\nonumber\\
&\Big\|w_{1c}\chi_1+w_{1o}\chi_{1c}+f_0-f_{0\ell}+\sum_{l\in Z^3}\Big(v_0-v_0\Big(\frac{l}{\mu_{1}}\Big)\Big)\chi_{1l}
+\sum_{l\in Z^3}w_{1l}\Big(\theta_0-\theta_0\Big(\frac{l}{\mu_1}\Big)\Big)\Big\|_{C^1_{t,x}}\nonumber\\
\leq& C_0\lambda_1\Big(\frac{\delta\mu_1}{\lambda_1}+\sqrt{\delta}\frac{\Lambda}{\mu_1}+\Lambda\ell\Big).
\end{align}
\begin{proof}
Using  Lemma \ref{e:estimate on perturbation and correction}, it's easy to obtain
\begin{align}
\|w_{1c}\chi_1\|_0\leq C_0\frac{\delta\mu_1}{\lambda_1},\quad\|w_{1o}\chi_{1c}\|_0\leq C_0\frac{\delta\mu_1}{\lambda_1},\quad
\|w_{1c}\chi_1\|_{C^1_{t,x}}\leq C_0\delta\mu_1
,\quad\|w_{1o}\chi_{1c}\|_{C^1_{t,x}}\leq C_0\delta\mu_1.\nonumber
\end{align}
Then,
\begin{align}
&\sum_{l\in Z^3}\Big(v_0-v_0\Big(\frac{l}{\mu_{1}}\Big)\Big)\chi_{1l}\nonumber\\
=&\sum\limits_{l\in Z^3}\Big(v_0-v_0\Big(\frac{l}{\mu_{1}}\Big)\Big)h_{1l}e^{i\lambda_1 2^{[l]} k_1^{\perp}\cdot \big(x-v_0(\frac{l}{\mu_{1}})t\big)}+\sum\limits_{l\in Z^3}\Big(v_0-v_0\Big(\frac{l}{\mu_{1}}\Big)\Big)h_{-1l}e^{-i\lambda_1 2^{[l]} k_1^{\perp}\cdot \big(x-v_0(\frac{l}{\mu_{1}})t\big)}.\nonumber
\end{align}
Obviously, $h_{1l}(x,t), h_{-1l}\neq0$ implies $|(\mu_1t,\mu_1 x)-l|\leq1.$
Therefore, following the same strategy as estimate (\ref{e:first part estimate}), we obtain
\begin{align}
\Big\|\sum_{l\in Z^3}\Big(v_0-v_0\Big(\frac{l}{\mu_{1}}\Big)\Big)\chi_{1l}\Big\|_0\leq C_0\sqrt{\delta}\frac{\Lambda}{\mu_1}.\nonumber
\end{align}
Similarly, we have
\begin{align}
\Big\|\sum_{l\in Z^3}w_{1l}\Big(\theta_0-\theta_0\Big(\frac{l}{\mu_1}\Big)\Big)\Big\|_0\leq C_0\sqrt{\delta}\frac{\Lambda}{\mu_1}.\nonumber
\end{align}
By calculation we have
\begin{align}
&\partial_t\Big(\sum_{l\in Z^3}\Big(v_0-v_0\Big(\frac{l}{\mu_{1}}\Big)\Big)\chi_{1l}\Big)\nonumber\\
=&\sum\limits_{l\in Z^3}\Big(v_0-v_0\Big(\frac{l}{\mu_{1}}\Big)\Big)\partial_th_{1l}e^{i\lambda_1 2^{[l]} k_1^{\perp}\cdot \big(x-v_0(\frac{l}{\mu_{1}})t\big)}+\sum\limits_{l\in Z^3}\Big(v_0-v_0\Big(\frac{l}{\mu_{1}}\Big)\Big)\partial_th_{-1l}e^{-i\lambda_1 2^{[l]} k_1^{\perp}\cdot \big(x-v_0(\frac{l}{\mu_{1}})t\big)}\nonumber\\
&-\sum\limits_{l\in Z^3}\Big(v_0-v_0\Big(\frac{l}{\mu_{1}}\Big)\Big)h_{1l}i\lambda_1 2^{[l]} k_1^{\perp}\cdot v_0\Big(\frac{l}{\mu_{1}}\Big)e^{i\lambda_1 2^{[l]} k_1^{\perp}\cdot \big(x-v_0(\frac{l}{\mu_{1}})t\big)}\nonumber\\
&+\sum\limits_{l\in Z^3}\Big(v_0-v_0\Big(\frac{l}{\mu_{1}}\Big)\Big)h_{-1l}i\lambda_1 2^{[l]} k_1^{\perp}\cdot v_0\Big(\frac{l}{\mu_{1}}\Big)e^{i\lambda_1 2^{[l]} k_1^{\perp}\cdot \big(x-v_0(\frac{l}{\mu_{1}})t\big)}\nonumber\\
&+\sum\limits_{l\in Z^3}\partial_tv_0h_{1l}e^{i\lambda_1 2^{[l]} k_1^{\perp}\cdot \big(x-v_0(\frac{l}{\mu_{1}})t\big)}+\sum\limits_{l\in Z^3}\partial_tv_0h_{-1l}e^{-i\lambda_1 2^{[l]} k_1^{\perp}\cdot \big(x-v_0(\frac{l}{\mu_{1}})t\big)}.\nonumber
\end{align}
Therefore, by (\ref{e:zero estimate on transport amp}), (\ref{e:zero estimate on transport time derivative}) and parameter assumption(\ref{a:assumption on parameter})
\begin{align}
\Big\|\partial_t\Big(\sum_{l\in Z^3}\Big(v_0-v_0\Big(\frac{l}{\mu_{1}}\Big)\Big)\chi_{1l}\Big)\Big\|_0\leq C_0\sqrt{\delta}\lambda_1\frac{\Lambda}{\mu_1}.\nonumber
\end{align}
Similarly, we have
\begin{align}
\Big\|\nabla\Big(\sum_{l\in Z^3}\Big(v_0-v_0\Big(\frac{l}{\mu_{1}}\Big)\Big)\chi_{1l}\Big)\Big\|_0\leq C_0\sqrt{\delta}\lambda_1\frac{\Lambda}{\mu_1}.\nonumber
\end{align}
Applying the similar argument, we have
\begin{align}
\Big\|\sum_{l\in Z^3}w_{1l}\Big(\theta_0-\theta_0\Big(\frac{l}{\mu_1}\Big)\Big)\Big\|_{C^1_{t,x}}\leq C_0\sqrt{\delta}\lambda_1\frac{\Lambda}{\mu_1}.\nonumber
\end{align}
By (\ref{e:estimate different})
\begin{align}
\|f_0-f_{0\ell}\|_0 \leq C_0\Lambda\ell,\qquad\|f_0-f_{0\ell}\|_{C^1_{t,x}} \leq C_0\Lambda.\nonumber
\end{align}
Collecting the above estimates together, we complete our proof.
\end{proof}
\end{Lemma}
Combining lemma \ref{e:osci}, lemma \ref{e:trans} and lemma \ref{e:est 1}, we conclude
\begin{align}\label{e:first terperture stress estimate}
\|\delta f_{01}\|_0\leq  C_0(\varepsilon)\Big(\frac{\sqrt{\delta}\mu_1}{\lambda_1}+\sqrt{\delta}\frac{\Lambda}{\mu_1}+\Lambda\ell\Big),\quad
\|\delta f_{01}\|_{C^1_{t,x}}\leq  C_0(\varepsilon)\lambda_1\Big(\frac{\sqrt{\delta}\mu_1}{\lambda_1}+\sqrt{\delta}\frac{\Lambda}{\mu_1}+\Lambda\ell\Big).
\end{align}
Finally, by (\ref{f:form R1}), (\ref{f:form g1}), Corollary \ref{e:differerce first estimate}, (\ref{e:estimate on first velocity error}) and (\ref{e:first terperture stress estimate}), we conclude that $(v_{01},p_{01},\theta_{01}, R_{01}, f_{01})\in C_c^{\infty}(Q_{r+\delta})$ solves system (\ref{d:boussinesq reynold}) and satisfies
\begin{align}
 R_{01}(t,x)=-\sum_{i=2}^3 a^2_i(t,x)k_i\otimes k_i+\delta R_{01},\qquad
    f_{01}(t,x):=-c_2(t,x)k_2+\delta f_{01}\nonumber
\end{align}
with
\begin{align}
\|v_{01}-v_0\|_0\leq \frac{M\sqrt{\delta}}{6}+C_0\frac{\sqrt{\delta}\mu_1}{\lambda_1},\quad
\|p_{01}-p_0\|_0\leq& M\delta,\quad\|\theta_{01}-\theta_0\|_0\leq \frac{M\sqrt{\delta}}{4}+C_0\frac{\sqrt{\delta}\mu_1}{\lambda_1},\nonumber\\
\|v_{01}-v_0\|_{C^1_{t,x}}\leq C_0\lambda_1\sqrt{\delta},\qquad
\|p_{01}-p_0\|_{C^1_{t,x}}\leq& C_0,\qquad \|\theta_{01}-\theta_0\|_{C^1_{t,x}}\leq C_0\lambda_1\sqrt{\delta},\nonumber\\
\|\delta R_{01}\|_0\leq C_0(\varepsilon)\Big(\sqrt{\delta}\frac{\mu_1}{\lambda_1}+\sqrt{\delta}\frac{\Lambda}{\mu_1}+\Lambda\ell\Big),\quad&
\|\delta R_{01}\|_{C^1_{t,x}}\leq C_0(\varepsilon)\lambda_1\Big(\sqrt{\delta}\frac{\mu_1}{\lambda_1}+\sqrt{\delta}\frac{\Lambda}{\mu_1}+\Lambda\ell\Big),\nonumber\\
\|\delta f_{01}\|_0\leq C_0(\varepsilon)\Big(\sqrt{\delta}\frac{\mu_1}{\lambda_1}+\sqrt{\delta}\frac{\Lambda}{\mu_1}+\Lambda\ell\Big),\quad&
\|\delta f_{01}\|_{C^1_{t,x}}\leq C_0(\varepsilon)\lambda_1\Big(\sqrt{\delta}\frac{\mu_1}{\lambda_1}+\sqrt{\delta}\frac{\Lambda}{\mu_1}+\Lambda\ell\Big).\nonumber
\end{align}
Thus, we complete the first step.

\setcounter{equation}{0}
\section{Constructions of $(v_{0n}, p_{0n}, \theta_{0n}, R_{0n}, f_{0n})$, $2\leq n\leq 3$}

In this section, we suppose $2\leq n \leq 3$ and construct $(v_{0n}, p_{0n}, \theta_{0n}, R_{0n}, f_{0n})$ by inductions.

Suppose that for $1\leq m<n\leq 3$, $(v_{0m}, p_{0m}, \theta_{0m}, R_{0m}, f_{0m})\in C_c^{\infty}(Q_{r+\delta})$ solves system (\ref{d:boussinesq reynold}) and satisfies
\begin{align}\label{e:form R0m}
 R_{0m}=-\sum_{i=m+1}^3 a^2_ik_i\otimes k_i+\sum_{i=i}^m\delta R_{0i},\quad
    f_{0m}:=-\sum_{i=m+1}^3c_ik_i+\sum_{i=i}^m\delta f_{0i}
\end{align}
with
\begin{align}\label{e:induction difference estimate}
 \|v_{0m}-v_{0(m-1)}\|_0\leq  \frac{M\sqrt{\delta}}{6}+C_0\frac{\sqrt{\delta}\mu_m}{\lambda_m}& ,\quad \|v_{0m}-v_{0(m-1)}\|_{C^1_{t,x}}\leq C_0\lambda_m\sqrt{\delta},\nonumber\\
\|p_{0m}-p_{0(m-1)}\|_0\leq \left\{
\begin{array}{ll}
M\delta,  m=1\\
 0, \quad m=2
 \end{array}
 \right.&\quad
 \|p_{0m}-p_{0(m-1)}\|_{C^1_{t,x}}\leq \left\{
\begin{array}{cc}
C_0,  m=1\\
 0, m=2
 \end{array}
 \right.\nonumber\\
 \|\theta_{0m}-\theta_{0(m-1)}\|_0\leq
  \frac{M\sqrt{\delta}}{6}+C_0\frac{\sqrt{\delta}\mu_m}{\lambda_m},&\quad
 \|\theta_{0m}-\theta_{0(m-1)}\|_{C^1_{t,x}}\leq
 C_0\lambda_m\sqrt{\delta} .
  \end{align}
and
  \begin{align}\label{e:induction error estimate}
\|\delta R_{0m}\|_0\leq C_0(\varepsilon)\Big(\frac{\sqrt{\delta}\mu_m}{\lambda_m}+\sqrt{\delta}\frac{\sqrt{\delta}\lambda_{m-1}}{\mu_m}\Big),\quad&
\|\delta f_{0m}\|_0\leq C_0(\varepsilon)\Big(\frac{\sqrt{\delta}\mu_m}{\lambda_m}+\sqrt{\delta}\frac{\sqrt{\delta}\lambda_{m-1}}{\mu_m}\Big),\nonumber\\
\|\delta R_{0m}\|_{C^1_{t,x}}\leq C_0(\varepsilon)\lambda_m\Big(\frac{\sqrt{\delta}\mu_m}{\lambda_m}+\sqrt{\delta}\frac{\sqrt{\delta}\lambda_{m-1}}{\mu_m}\Big),\quad&
\|\delta f_{0m}\|_{C^1_{t,x}}\leq C_0(\varepsilon)\lambda_m\Big(\frac{\sqrt{\delta}\mu_m}{\lambda_m}+\sqrt{\delta}\frac{\sqrt{\delta}\lambda_{m-1}}{\mu_m}\Big).
\end{align}
Here $(v_{00}, p_{00}, \theta_{00})=(v_0, p_0, \theta_0)$ and the parameter $\mu_m, \lambda_m$ satisfies
\begin{align}\label{a:assumption on parameter m sequence}
  \lambda_m\geq\max\{\mu_m^{1+\varepsilon},\ell^{-(1+\varepsilon)}\}, \quad \mu_m>\mu_{m-1}
  \end{align}
  and $\lambda_0=\Lambda\delta^{-\frac{1}{2}}+\frac{\mu_1\Lambda\ell}{\delta}, \mu_0=1$.
Next, we perform the n-th step.

\subsection{Construction of n-th perturbation on velocity}

   \indent

\subsubsection{Main perturbation $w_{no}$}
For any $l\in Z^3$, we denote $b_{nl}$ by
   \begin{align}\label{d:l amp}
    b_{nl}(t,x):=&\frac{a_n(t,x)\alpha_l(\mu_nt,\mu_nx)}{\sqrt{2}}
    \end{align}
and main $l$-perturbation $w_{1ol}$ by
   \begin{align}\label{d:w 1ol}
    w_{nol}:=b_{nl}k_n\Big(e^{i\lambda_n 2^{[l]} k_n^{\perp}\cdot \big(x-v_{0(n-1)}(\frac{l}{\mu_n})t\big)}+e^{-i\lambda_n 2^{[l]}k_n^{\perp}\cdot \big(x-v_{0(n-1)}(\frac{l}{\mu_n})t\big)}\Big),
   \end{align}
where two parameters $\mu_n$ and $\lambda_n$  will be chosen with
\begin{align}\label{a:n step parameter asssumption}
 \lambda_n\geq\max\{\mu_n^{1+\varepsilon},\ell^{-(1+\varepsilon)}\}, \quad \mu_n>\mu_{n-1}.
  \end{align}
Then we denote n-th main perturbation $w_{no}$ by
   \begin{align}
    w_{no}:=\sum_{l\in Z^3}w_{nol}=\sum_{j=0}^7\sum_{[l]=j}b_{nl}k_n\Big(e^{i\lambda_n 2^j k_n^{\perp}\cdot \big(x-v_{0(n-1)}(\frac{l}{\mu_n})t\big)}+e^{-i\lambda_n 2^jk_n^{\perp}\cdot \big(x-v_{0(n-1)}(\frac{l}{\mu_n})t\big)}\Big).\nonumber
   \end{align}
Obviously, $w_{nol},w_{no}$ are all real 2-dimensional vector-valued functions.

By (\ref{p:unity}), $\hbox{supp} \alpha_l\cap \hbox{supp}\alpha_{l'}=\emptyset$  if $|l-l'|\geq2$, thus the above summation is finite and
   \begin{align}\label{b:bound 3}
  \|w_{no}\|_0\leq \frac{M\sqrt{\delta}}{6}.
  \end{align}
  Moreover, since $a_n(t,x)\in C^{\infty}_c(Q_{r+\delta})$, we know that for any $l\in Z^3$,
  \begin{align}\label{p:support}
  b_{nl}\in C^{\infty}_c(Q_{r+\delta}),\quad w_{nol}\in C^{\infty}_c(Q_{r+\delta})\quad w_{no}\in C^{\infty}_c(Q_{r+\delta}).
\end{align}

 \subsubsection{The correction $w_{nc}$ and the n-th perturbation $w_{n}$}
We denote the $l$-correction $w_{ncl}$ by
   \begin{align}
   w_{ncl}:=&\frac{\nabla^{\perp}b_{nl}}{i\lambda_n 2^{[l]}}\Big(e^{i\lambda_n 2^{[l]} k_1^{\perp}\cdot \big(x-v_{0(n-1)}(\frac{l}{\mu_n})t\big)}-e^{-i\lambda_n 2^{[l]} k_n^{\perp}\cdot \big(x-v_{0(n-1)}(\frac{l}{\mu_n})t\big)}\Big)\nonumber\\
   &-\nabla^{\bot}\Big(\frac{\nabla b_{nl}\cdot k_n^{\perp}}{\lambda_n^22^{2[l]}|k_n|^2}\Big(e^{i\lambda_n 2^{[l]} k_1^{\perp}\cdot \big(x-v_{0(n-1)}(\frac{l}{\mu_n})t\big)}+e^{-i\lambda_n 2^{[l]} k_n^{\perp}\cdot \big(x-v_{0(n-1)}(\frac{l}{\mu_n})t\big)}\Big)\Big).\nonumber
   \end{align}
Then the n-th correction is given by
   \begin{align}
    w_{nc}:=\sum_{l\in Z^3}w_{ncl}.\nonumber
    \end{align}
Finally, we denote n-th perturbation $w_n$ by
   \begin{align}
    w_n:=w_{no}+w_{nc}.\nonumber
   \end{align}
   Thus, if we set
   \begin{align}
   w_{nl}:=w_{nol}+w_{ncl},\nonumber
   \end{align}
   then
   \begin{align}
   w_{nl}=\nabla^{\perp}{\rm div}\Big(-\frac{b_{nl}}{\lambda_n^22^{2[l]}|k_n|^2}k_n^{\perp}2cos\Big(\lambda_n2^{[l]}k_n^{\perp}\cdot \big(x-v_{0(n-1)}(\frac{l}{\mu_n})t\big)\Big)\Big)\nonumber
   \end{align}
and
   \begin{align}
   w_n=\sum\limits_{l\in Z^3}w_{nl},\quad {\rm div}w_{nl}=0,\quad\int_{R^2}w_{nl}dx=0, \quad \int_{R^2}(x_iw_{nlj}-x_jw_{nli})dx=0,\quad i.j=1,2.\nonumber
   \end{align}
In fact
   \begin{align}
   \int_{R^2}(x_1w_{nl2}-x_2w_{nl1})dx=\int_{R^2}(x_1\partial_1{\rm div}\vec{a}_n+x_2\partial_2{\rm div}\vec{a}_n)dx=-2\int_{R^2}{\rm div}\vec{a}_ndx=0,\nonumber
   \end{align}
where $$\vec{a}_n=-\frac{b_{nl}}{\lambda_n^22^{2[l]}|k_n|^2}
k_1^{\perp}2cos\Big(\lambda_n2^{[l]}k_n^{\perp}\cdot \big(x-v_{0(n-1)}(\frac{l}{\mu_n})t\big)\Big)$$
 is a vector-valued smooth function with compact support.
   Obviously, we have
   \begin{align}
   \hbox{div}w_n=0.\nonumber
   \end{align}
Moreover, if we set
\begin{align}
k_{nl}:=&b_{nl}k_n+\frac{\nabla^{\perp}b_{nl}}{i\lambda_n 2^{[l]}}-\nabla^{\bot}\Big(\frac{\nabla b_{nl}\cdot k_n^{\perp}}{\lambda_n^22^{2[l]}|k_n|^2}\Big)+\frac{\nabla b_{nl}\cdot k_n^{\perp}}{i\lambda_n2^{[l]}|k_n|^2}\cdot k_n,\nonumber\\
k_{-nl}:=&b_{nl}k_n+\frac{\nabla^{\perp}b_{nl}}{-i\lambda_n 2^{[l]}}-\nabla^{\bot}\Big(\frac{\nabla b_{nl}\cdot k_n^{\perp}}{\lambda_n^22^{2[l]}|k_n|^2}\Big)+\frac{\nabla b_{nl}\cdot k_n^{\perp}}{-i\lambda_n2^{[l]}|k_n|^2}\cdot k_n,\nonumber
\end{align}
then
\begin{align}\label{d:another definition for n perturbation}
w_{nl}=k_{nl}e^{i\lambda_n 2^{[l]} k_n^{\perp}\cdot \big(x-v_{0(n-1)}(\frac{l}{\mu_n})t\big)}+k_{-nl}e^{-i\lambda_n 2^{[l]}k_n^{\perp}\cdot \big(x-v_{0(n-1)}(\frac{l}{\mu_n})t\big)}.
\end{align}

Furthermore, we have
   \begin{align}
   w_{nol},~w_{ncl},~w_{nl},~w_{no},~w_{nc},~w_n,~k_{nl},~k_{-nl}\in C^{\infty}_c(Q_{r+\delta}).\nonumber
   \end{align}
Thus, we complete the construction of perturbation $w_n$.

\subsection{Construction of n-th perturbation on temperature}

\indent

To construct $\chi_n$, we denote $\beta_{nl}$ by
\begin{align}\label{d:definition nl}
    \beta_{nl}(t,x):=\left \{
    \begin {array}{cc}
    \frac{c_{2\ell}(t,x)\alpha_l(\mu_2t,\mu_2x)}{\sqrt{2e(t,x)}\gamma_{k_2}\big(Id-
    \frac{R_{0\ell}(t,x)}{e(t,x)}\big)}, \quad n=2\\
    0,\quad \quad \quad \quad n=3.
    \end{array}
    \right.
    \end{align}
Since ${\rm supp}c_{2\ell}\subseteq {\rm supp}e$, then $\beta_{2l}(x,t)$ is well-defined.

Then we denote main $l$-perturbation $ \chi_{nol}$ by
    \begin{align}
    \chi_{nol}(t,x):=\left \{
    \begin {array}{cc}
    \beta_{2l}(t,x)\Big(e^{i\lambda_2 2^{[l]} k_2^{\perp}\cdot \big(x-v_{01}(\frac{l}{\mu_2})t\big)}+e^{-i\lambda_2 2^{[l]} k_2^{\perp}\cdot \big(x-v_{01}(\frac{l}{\mu_2})t\big)}\Big), \quad n=2\\
    0,\quad \quad \quad \quad n=3.
    \end{array}
    \right.\nonumber
    \end{align}
    and $l$-correction $\chi_{ncl}$ by
    \begin{align}
    \chi_{ncl}(t,x):=\left \{
    \begin {array}{cc}
    \triangle\beta_{2l}(t,x)\Big(\frac{e^{i\lambda_2 2^{[l]} k_2^{\perp}\cdot \big(x-v_{01}(\frac{l}{\mu_2})t\big)}}{-\lambda_2^22^{2[l]}|k_2|^2}
    +\frac{e^{-i\lambda_2 2^{[l]} k_2^{\perp}\cdot \big(x-v_{01}(\frac{l}{\mu_2})t\big)}}{-\lambda_2^22^{2[l]}|k_2|^2}\Big)
    +2\nabla\beta_{2l}(t,x)\cdot\nonumber\\
    \nabla\Big(\frac{e^{i\lambda_2 2^{[l]} k_2^{\perp}\cdot \big(x-v_{01}(\frac{l}{\mu_2})t\big)}}{-\lambda_2^22^{2[l]}|k_2|^2}
    +\frac{e^{-i\lambda_2 2^{[l]} k_2^{\perp}\cdot \big(x-v_{01}(\frac{l}{\mu_2})t\big)}}{-\lambda_2^22^{2[l]}|k_2|^2}\Big),\quad n=2\\
    0,\quad \quad \quad \quad n= 3.
     \end{array}
    \right.\nonumber
    \end{align}
 Finally, we introduce $\chi_{nl}$ by
 \begin{align}
 \chi_{nl}:=\chi_{nol}+\chi_{ncl}=\left\{
 \begin{array}{cc}
 \triangle\Big(\beta_{2l}\Big(\frac{e^{i\lambda_2 2^{[l]} k_2^{\perp}\cdot \big(x-v_{01}(\frac{l}{\mu_2})t\big)}}{-\lambda_2^22^{2[l]}|k_2|^2}+\frac{e^{-i\lambda_2 2^{[l]} k_2^{\perp}\cdot \big(x-v_{01}(\frac{l}{\mu_2})t\big)}}{-\lambda_2^22^{2[l]}|k_2|^2}\Big)\Big),\quad n=2\\
 0,\quad \quad \quad \quad n=3.
 \end{array}
 \right.\nonumber
 \end{align}
 and $\chi_{no}, \chi_{nc}, \chi_n$ by, respersively,
   \begin{align}
    \chi_{no}:=\sum_{l\in Z^3}\chi_{nol},\quad \chi_{nc}:=\sum_{l\in Z^3}\chi_{ncl},\quad \chi_n:=\sum_{l\in Z^3}\chi_{nl}.\nonumber
   \end{align}
Then $\chi_{nol}, ~\chi_{ncl}, ~\chi_{nl}$ and $\chi_n$ are all real scalar functions and as the perturbations of $w_n$, the summation in their definitions is finite.

Now we set
   \begin{align}
   h_{nl}:=\left\{
   \begin{array}{ll}
   \beta_{2l}-\frac{\triangle\beta_{2l}}{\lambda_2^22^{2[l]}|k_2|^2}
   +2\frac{\nabla\beta_{2l}\cdot k_2^{\perp}}{i\lambda 2^{[l]}|k_2|^2},~~~ n=2\\
   0,\quad \quad \quad \quad n=3.
   \end{array}
   \right.\quad
   h_{-nl}:=\left\{
   \begin{array}{ll}
   \beta_{2l}-\frac{\triangle\beta_{2l}}{\lambda_2^22^{2[l]}|k_2|^2}
   +2\frac{\nabla\beta_{2l}\cdot k_2^{\perp}}{-i\lambda 2^{[l]}|k_2|^2},~~~ n=2\\
    0,\quad \quad \quad \quad n=3.
   \end{array}
   \right.\nonumber
   \end{align}
   then
   \begin{align}
   \chi_{nl}=h_{nl}e^{i\lambda_n 2^{[l]} k_n^{\perp}\cdot \big(x-v_{0(n-1)}(\frac{l}{\mu_n})t\big)}+h_{-nl}e^{-i\lambda_n 2^{[l]} k_n^{\perp}\cdot \big(x-v_{0(n-1)}(\frac{l}{\mu_n})t\big)}.\nonumber
   \end{align}
Moreover, since ${\rm supp} c_{2\ell}\subseteq Q_{r+\delta} $, we know that for all $l\in Z^3$,
 \begin{align}\label{p:support 2}
 \beta_{nl}\in C^{\infty}_c(Q_{r+\delta}),\quad h_{nl}\in C^{\infty}_c(Q_{r+\delta})
 \end{align}
 and
 \begin{align}
 \chi_{nol},\quad \chi_{ncl},\quad \chi_{nl},\quad \chi_{no},\quad \chi_{nc},\quad \chi_{n}\in C^{\infty}_c(Q_{r+\delta}).\nonumber
 \end{align}
Then, by (\ref{p:property on cil}), (\ref{d:definition on p}), (\ref{d:difinition on e}), (\ref{d:definition on m}) and (\ref{d:definition nl}), we know that
  \begin{align}
  \|\beta_{nl}\|_0\leq\left\{
  \begin{array}{ll}
   \frac{M\sqrt{\delta}}{300},~~~n=2\\
   0,~~~~n=3.
   \end{array}
   \right.\nonumber
  \end{align}
Therefore, by (\ref{p:unity})
  \begin{align}\label{e:bound x}
 \|\chi_{no}\|_0\leq\left\{
  \begin{array}{cc}
   \frac{M\sqrt{\delta}}{4},~~~n=2\\
   0,~~~~n=3.
   \end{array}
   \right.
  \end{align}

\subsection{The construction of  $v_{0n}$,~$p_{0n}$,~$\theta_{0n}$,~$f_{0n}$,~$R_{0n}$ }.

\indent
First, we denote
$M_n$ by
\begin{align}
M_n:=&\sum\limits_{l\in Z^3}b^2_{nl}k_n\otimes k_n\Big(e^{2i\lambda_n 2^{[l]} k_n^{\perp}\cdot \big(x-v_{0(n-1)}(\frac{l}{\mu_n})t\big)}
+e^{-2i\lambda_n 2^{[l]} k_n^{\perp}\cdot \big(x-v_{0(n-1)}(\frac{l}{\mu_n})t\big)}\Big)\nonumber\\
&+\sum\limits_{l,l'\in Z^3 ,l\neq l'}w_{nol}\otimes w_{nol'}\nonumber
\end{align}
and
$N_n,K_n$ by
 \begin{align}
   N_n=&\sum_{l\in Z^3}\Big[w_{nl}\otimes \Big(v_{0(n-1)}-v_{0(n-1)}\Big(\frac{l}{\mu_{n}}\Big)\Big)
   +\Big(v_{0(n-1)}-v_{0(n-1)}\Big(\frac{l}{\mu_{n}}\Big)\Big{)}\otimes w_{nl}\Big],\nonumber\\
   K_n=&\sum\limits_{l\in Z^3}\beta_{nl}b_{nl}k_n\Big(e^{2i\lambda_n 2^{[l]} k_n^{\perp}\cdot \big(x-v_{0(n-1)}(\frac{l}{\mu_n})t\big)}
   +e^{-2i\lambda_n 2^{[l]} k_n^{\perp}\cdot \big(x-v_{0(n-1)}(\frac{l}{\mu_n})t\big)}\Big{)}
   +\sum\limits_{l,l'\in Z^3 ,l\neq l'}w_{nol}\chi_{nol'}.\nonumber
    \end{align}
  Then we set
   \begin{align}\label{d:difinition of n sequence}
    v_{0n}:=v_{0(n-1)}+w_n,\quad
    p_{0n}:=&p_{0(n-1)},\quad
    \theta_{0n}:=\theta_{0(n-1)}+\chi_n,\nonumber\\
     R_{0n}:=R_{0(n-1)}+2\sum\limits_{l\in Z^3}b^2_{nl}k_n\otimes k_n+\delta R_{0n},&\quad
    f_{0n}:=f_{0(n-1)}+2\sum\limits_{l\in Z^3}\beta_{nl}b_{nl}k_n+\delta f_{0n},
   \end{align}
where
   \begin{align}
   \delta R_{0n}=&\mathcal{R}(\hbox{div}M_n)+N_n-\mathcal{R}(\chi_ne_2)+\mathcal{R}\Big\{\partial_tw_{n}
   +\hbox{div}\Big[\sum_{l\in Z^3}\Big(w_{nl}\otimes v_{0(n-1)}\Big(\frac{l}{\mu_{n}}\Big)\nonumber\\
   &+v_{0(n-1)}\Big(\frac{l}{\mu_{n}}\Big)\otimes w_{nl}\Big)\Big]\Big\}
   +(w_{no}\otimes w_{nc}+w_{nc}\otimes w_{no}+w_{nc}\otimes w_{nc}),\nonumber
   \end{align}
and
   \begin{align}\label{d:definition f0n}
   \delta f_{0n}=\left\{
   \begin{array}{ll}
   \mathcal{G}(\hbox{div}K_2)
   +\mathcal{G}\Big(\partial_t\chi_{2}+\sum_{l\in Z^3}v_{01}\Big(\frac{l}{\mu_2}\Big)\cdot\nabla\chi_{2l}\Big)+w_{2o}\chi_{2c}
   +\sum_{l\in Z^3}\Big(v_{01}-v_{01}\Big(\frac{l}{\mu_2}\Big)\Big)\chi_{2l}\\
   +w_{2c}\chi_2+\sum_{l\in Z^3}w_{2l}\Big(\theta_{01}-\theta_{01}\Big(\frac{l}{\mu_2}\Big)\Big),\qquad n=2\\
   \sum_{l\in Z^3}w_{3l}\Big(\theta_{02}-\theta_{02}\Big(\frac{l}{\mu_3}\Big)\Big), \qquad n=3.
   \end{array}
   \right.
  \end{align}
Since $$\hbox{div}M_n,~\chi_ne_2, ~\partial_tw_{n},~~
\hbox{div}\Big[\sum_{l\in Z^3}\Big(w_{nl}\otimes v_{0(n-1)}\Big(\frac{l}{\mu_{n}}\Big)
   +v_{0(n-1)}\Big(\frac{l}{\mu_{n}}\Big)\otimes w_{nl}\Big)\Big]\in\Xi,$$
   so $\delta R_{0n}$ is well-defined. Moreover,
   $$\hbox{div}K_2,~~
   \partial_t\chi_{2}+\sum_{l\in Z^3}v_{01}\Big(\frac{l}{\mu_2}\Big)\cdot\nabla\chi_{2l}\in\Psi,$$
thus $\delta f_{0n}$ is well-defined.
 By Proposition \ref{p:proposition R} and Corollary \ref{pro:support of inverse operator}, we know that $\delta R_{0n}$ is a symmetric matrix and $\delta R_{0n}\in
    C_c^{\infty}(Q_{r+\delta})$. Also, by Proposition \ref{p:inverse 2} and Corollary \ref{pro:support of inverse operator}, we have $\delta f_{0n}\in
    C_c^{\infty}(Q_{r+\delta})$.
Obviously,
   \begin{align*}
    \hbox{div}v_{0n}=\hbox{div}v_{0(n-1)}+\hbox{div}w_{n}=0.
   \end{align*}
 Moreover, from the definition (\ref{d:difinition of n sequence}) on $(v_{0n}, p_{0n}, \theta_{0n}, R_{0n}, f_{0n})$ and the fact that $(v_{0(n-1)}, p_{0(n-1)},\theta_{0(n-1)}, \\
 R_{0(n-1)}, f_{0(n-1)})$ solves the system (\ref{d:boussinesq reynold}), Proposition \ref{p:proposition R}, we know that
\[\begin{aligned}
    \hbox{div}R_{0n}=&\hbox{div}R_{0(n-1)}+\partial_tw_n-\chi_ne_2
    +\hbox{div}(w_{no}\otimes w_{no}+w_{n}\otimes v_{0(n-1)}+v_{0(n-1)}\otimes w_{n}\nonumber\\
    &+w_{no}\otimes w_{nc}+w_{nc}\otimes w_{no}+w_{nc}\otimes w_{nc})\nonumber\\
    =&\partial_tv_{0(n-1)}+\hbox{div}(v_{0(n-1)}\otimes v_{0(n-1)})+\nabla p_{0(n-1)}-\theta_{0(n-1)}e_2+\partial_tw_n
    -\chi_ne_2\nonumber\\
    &+\hbox{div}(w_{no}\otimes w_{no}+w_{n}\otimes v_{0(n-1)}+v_{0(n-1)}\otimes w_{n}
    +w_{no}\otimes w_{nc}+w_{nc}\otimes w_{no}+w_{nc}\otimes w_{nc})\nonumber\\
    =&\partial_tv_{0n}+\hbox{div}(v_{0n}\otimes v_{0n})+\nabla p_{0n}-\theta_{0n}e_2.
    \end{aligned}\]
Where we used
$$\hbox{div}(M_n)+\hbox{div}\Big(2\sum\limits_{l\in Z^3}b^2_{nl}k_n\otimes k_n\Big)=\hbox{div}(w_{no}\otimes w_{no}).$$
  Furthermore, by (\ref{d:difinition of n sequence}) and (\ref{d:definition f0n}),  we have
  \[ \begin{aligned}
    f_{02}=&f_{01}+2\sum\limits_{l\in Z^3}\beta_{2l}b_{2l}k_2+\mathcal{G}(\hbox{div}K_2)
   +w_{2o}\chi_{2c}
   +\mathcal{G}\Big(\partial_t\chi_{2}+\sum_{l\in Z^3}v_{01}\Big(\frac{l}{\mu_2}\Big)\cdot\nabla\chi_{2l}\Big)\nonumber\\
   &+\sum_{l\in Z^3}\Big(v_0-v_{01}\Big(\frac{l}{\mu_2}\Big)\Big)\chi_{2l}+w_{2c}\chi_2+\sum_{l\in Z^3}w_{2l}\Big(\theta_{01}-\theta_{01}\Big(\frac{l}{\mu_2}\Big)\Big).\nonumber
   \end{aligned}\]
From the fact that $(v_{01}, p_{01}, \theta_{01}, R_{01}, f_{01})$ solves the system (\ref{d:boussinesq reynold}) and Proposition \ref{p:inverse 2} we have
 \begin{align}
 \hbox{div}f_{02}=&\hbox{div}f_{01}+\partial_{t}\chi_2
 +\hbox{div}(w_{2o}\chi_2+w_{2c} \chi_2+v_{01} \chi_2+w_2 \theta_{01})\nonumber\\
 =&\hbox{div}(v_{01}\theta_{01}+w_{2o}\chi_2+w_{2c} \chi_2+v_{01} \chi_2+w_2 \theta_{01})
 +\partial_t(\theta_{01}+\chi_2)\nonumber\\
 =&\partial_t\theta_{02}+\hbox{div}(v_{02}\theta_{02}),\nonumber
 \end{align}
 where we used
 \begin{align}
 \hbox{div}K_2+\hbox{div}(2\sum\limits_{l\in Z^3}\beta_{2l}b_{2l}k_2)=\hbox{div}(w_{2o}\chi_{2o}).\nonumber
 \end{align}
 Thus, the functions $(v_{02},p_{02},\theta_{02},R_{02},f_{02})$ satisfies the system (\ref{d:boussinesq reynold}).
And from the definition (\ref{d:difinition of n sequence}) on $\delta f_{03}$, we have
\begin{align}
\hbox{div}f_{03}=\hbox{div}f_{02}+\hbox{div}\delta f_{03}=\partial_t \theta_{02}+v_{02}\cdot\nabla\theta_{02}+w_3\cdot\nabla \theta_{02}=\partial_t \theta_{03}+v_{03}\cdot\nabla\theta_{03}.\nonumber
\end{align}
 Thus the functions $(v_{03},p_{03},\theta_{03},R_{03},f_{03})$ also solves the system (\ref{d:boussinesq reynold}).

\setcounter{equation}{0}
\section{The n-th Representation}

\indent

In this section, we will calculate the form of
$$R_{0n}+2\sum\limits_{l\in Z^3}b^2_{nl}k_n\otimes k_n={\tilde I}$$
and
$$f_{0n}+2\sum\limits_{l\in Z^3}\beta_{nl}b_{nl}k_n=\tilde {II}.$$

\subsection{The term $\tilde I$}

\indent

First, by (\ref{d:l amp}), we have
\begin{align}
2\sum\limits_{l\in Z^3}b^2_{nl}k_n\otimes k_n
=\sum\limits_{l\in Z^3}\alpha_l^2(\mu_nt,\mu_n x)a^2_nk_n\otimes k_n
=a^2_nk_n\otimes k_n.\nonumber
\end{align}
Where we used $\sum\limits_{l\in Z^3}\alpha_l^2=1$.
Therefore, by (\ref{e:form R0m}), we have
\begin{align}
R_{0(n-1)}+2\sum\limits_{l\in Z^3}b^2_{nl}k_n\otimes k_n
=-\sum_{i=n+1}^3 a^2_ik_i\otimes k_i+\sum_{i=i}^{n-1}\delta R_{0i}.\nonumber
\end{align}
Meanwhile, by (\ref{d:difinition of n sequence}), we have
\begin{align}
R_{0n}=-\sum_{i=n+1}^3 a^2_i(t,x)k_i\otimes k_i+\sum_{i=i}^n\delta R_{0i}.\nonumber
\end{align}
In particular,
\begin{align}
R_{03}=\sum_{i=1}^3\delta R_{0i}.\nonumber
\end{align}
In next section, we will prove that $\delta R_{0n}$ is small.\\

\subsection{The term $\tilde {II}$}

\indent

Then, by (\ref{d:l amp}) and (\ref{d:definition nl}), we have
\begin{align}
2\sum\limits_{l\in Z^3}\beta_{2l}b_{2l}k_2
=\sum\limits_{l\in Z^3}\alpha_l^2(\mu_2t,\mu_2 x)c_2k_2
=c_2k_2.\nonumber
\end{align}
From the identity (\ref{f:form g1}), we have
\begin{align}
&f_{01}+2\sum\limits_{l\in Z^3}\beta_{2l}b_{2l}k_2=\delta f_{01}.\nonumber
\end{align}
Meanwhile, by (\ref{d:difinition of n sequence}), we have
\begin{align}
f_{02}=&f_{01}+2\sum\limits_{l\in Z^3}\beta_{2l}b_{2l}k_2+\delta f_{02}
=\sum_{i=1}^2\delta f_{0i}.\nonumber
\end{align}
Since $\beta_{nl}=0$ when  $n=3$, then
\begin{align}
f_{03}=&f_{02}+2\sum\limits_{l\in Z^3}\beta_{3l}b_{3l}k_3+\delta f_{03}
=\sum_{i=1}^3\delta f_{0i}.\nonumber
\end{align}
In next section, we will prove that $\delta g_{0n}$ is small.\\
\setcounter{equation}{0}
 \section{Estimates on $\delta R_{0n}$ and $\delta g_{0n}$}

 \indent

We summarize the main estimates of $b_{nl}$ and $\beta_{nl}$.
\begin{Lemma}\label{l:n step estimate osc}
For any $l\in Z^3$ and any integer $m\geq 1$, we have
\begin{align}
\|b_{nl}\|_m+\|\beta_{nl}\|_m\leq& C_m\sqrt{\delta}(\mu_n^m+\mu_n\ell^{-(m-1)}),\label{e:n step estimate on main perturbation}\\
\|\partial_tb_{nl}\|_m+\|\partial_t\beta_{nl}\|_m\leq& C_m\sqrt{\delta}(\mu_n^{m+1}+\mu_n\ell^{-m}),\\
\|\partial_{tt}b_{nl}\|_m+\|\partial_{tt}\beta_{nl}\|_m\leq& C_m\sqrt{\delta}(\mu_n^{m+2}+\mu_n\ell^{-m-1}),\\
\|k_{\pm nl}\|_m+\|h_{\pm nl}\|_m\leq& C_m\sqrt{\delta}(\mu_n^m+\mu_n\ell^{-(m-1)}),\\
\|\partial_tk_{\pm nl}\|_m+\|\partial_th_{\pm nl}\|_m\leq& C_m\sqrt{\delta}(\mu_n^{m+1}+\mu_n\ell^{-m}),\\
\|\partial_{tt}k_{\pm nl}\|_m+\|\partial_{tt}h_{\pm nl}\|_m\leq &C_m\sqrt{\delta}(\mu_n^{m+2}+\mu_n\ell^{-m-1})\label{e:n step estimate on b1l}
\end{align}
and
\begin{align}
\|b_{nl}\|_0+\|\beta_{nl}\|_0\leq& C_0\sqrt{\delta},\label{e:zero n step estimate on main perturbation}\\
\|\partial_tb_{nl}\|_0+\|\partial_t\beta_{nl}\|_0\leq& C_0\sqrt{\delta}\mu_n,\\
\|\partial_{tt}b_{nl}\|_0+\|\partial_{tt}\beta_{nl}\|_0\leq& C_0\sqrt{\delta}(\mu_n^{2}+\mu_n\ell^{-1}),\\
\|k_{\pm nl}\|_0+\|h_{\pm nl}\|_0\leq& C_0\sqrt{\delta},\\
\|\partial_tk_{\pm nl}\|_0+\|\partial_th_{\pm nl}\|_0\leq& C_0\sqrt{\delta}\mu_n,\\
\|\partial_{tt}k_{\pm nl}\|_0+\|\partial_{tt}h_{\pm nl}\|_0\leq &C_0\sqrt{\delta}(\mu_n^{2}+\mu_n\ell^{-1}).\label{e:zero n step estimate on b1l}
\end{align}
\begin{proof}
The proof is similar to Lemma \ref{l:estimate osc}, we omit the detail here.
\end{proof}
\end{Lemma}
Next, we give estimates on perturbations $w_{no}, w_{nc}, \chi_{no}, \chi_{nc}$.
 \begin{Lemma}[Estimate on perturbation]\label{e:n step estimate on perturbation and correction}
 \begin{align}
 \|w_{no}\|_0\leq &C_0\sqrt{\delta},\quad  \|w_{no}\|_{C^1_{t,x}}\leq C_0\sqrt{\delta}\lambda_n,\quad  \|\chi_{no}\|_0\leq C_0\sqrt{\delta},\quad  \|\chi_{no}\|_{C^1_{t,x}}\leq C_0\sqrt{\delta}\lambda_n,\nonumber\\
 \|w_{nc}\|_0\leq& C_0\frac{\sqrt{\delta}\mu_n}{\lambda_n},\quad  \|w_{nc}\|_{C^1_{t,x}}\leq C_0\sqrt{\delta}\mu_n,\quad
 \|\chi_{nc}\|_0\leq C_0\frac{\sqrt{\delta}\mu_n}{\lambda_n},\quad  \|\chi_{nc}\|_{C^1_{t,x}}\leq C_0\sqrt{\delta}\mu_n.
\end{align}
 \begin{proof}
 The proof is similar to Lemma \ref{e:estimate on perturbation and correction}, we omit the detail here.
 \end{proof}
 \end{Lemma}
\begin{Corollary}\label{e:n sequence difference estimate}
\begin{align}
&\|v_{0n}-v_{0(n-1)}\|_0\leq \frac{M\sqrt{\delta}}{6}+C_0\frac{\sqrt{\delta}\mu_n}{\lambda_n},\quad \|v_{0n}-v_{0(n-1)}\|_{C^1_{t,x}}\leq C_0\lambda_n\sqrt{\delta},\quad p_{0n}-p_{0(n-1)}=0,\nonumber\\
&\|\theta_{02}-\theta_{01}\|_0\leq \frac{M\sqrt{\delta}}{4}+C_0\frac{\sqrt{\delta}\mu_2}{\lambda_2},\quad
  \|\theta_{02}-\theta_{01}\|_{C^1_{t,x}}\leq C_0\sqrt{\delta}\lambda_2,\quad \theta_{03}-\theta_{02}=0.\nonumber
\end{align}
\end{Corollary}
\subsection{Estimates on $\delta R_{0n}$}

 \indent

 Recalling that
  \begin{align}
   \delta R_{0n}=&\mathcal{R}(\hbox{div}M_n)+N_n-\mathcal{R}(\chi_ne_2)+\mathcal{R}\Big\{\partial_tw_{n}
   +\hbox{div}\Big[\sum_{l\in Z^3}\Big(w_{nl}\otimes v_{0(n-1)}\Big(\frac{l}{\mu_{n}}\Big)\nonumber\\
   &+v_{0(n-1)}\Big(\frac{l}{\mu_{n}}\Big)\otimes w_{nl}\Big)\Big]\Big\}
   +(w_{no}\otimes w_{nc}+w_{nc}\otimes w_{no}+w_{nc}\otimes w_{nc}).\nonumber
   \end{align}
 Again, we split the stress into three parts: \\
  (1)The oscillation part $$\mathcal{R}(\hbox{div}M_n)-\mathcal{R}(\chi_ne_2).$$
  (2)The transport part
  \begin{align}
  &\mathcal{R}\Big\{\partial_tw_{n}
   +\hbox{div}\Big[\sum_{l\in Z^3}\Big(w_{nl}\otimes v_{0(n-1)}\Big(\frac{l}{\mu_{n}}\Big)+v_{0(n-1)}\Big(\frac{l}{\mu_{n}}\Big)\otimes w_{nl}\Big)\Big]\Big\}\nonumber\\
   =&\mathcal{R}\Big(\partial_tw_{n}+\sum_{l\in Z^3}v_{0(n-1)}\Big(\frac{l}{\mu_{n}}\Big)\cdot\nabla w_{nl}\Big).\nonumber
   \end{align}
  (3)The error part
  \begin{align}
  N_n+(w_{no}\otimes w_{nc}+w_{nc}\otimes w_{no}+w_{nc}\otimes w_{nc}).\nonumber
   \end{align}
In the following we will estimate each term separately. Beside the estimates of $N_n$, the proof of other estimates are same as in Section 6. We only give the proof of the estimates on $N_n$ and omit the others here.

\begin{Lemma}[The oscillation part]\label{e:oscillation estimate n step}
\begin{align}
\|\mathcal{R}({\rm div}M_n)\|_0\leq C_0(\varepsilon)\frac{\delta\mu_n}{\lambda_n},\quad
\|\mathcal{R}(\chi_ne_2)\|_0\leq C_0(\varepsilon)\frac{\sqrt{\delta}}{\lambda_n},\nonumber\\
\|\mathcal{R}({\rm div}M_n)\|_{C^1_{t,x}}\leq C_0(\varepsilon)\delta\mu_n,\quad
\|\mathcal{R}(\chi_ne_2)\|_{C^1_{t,x}}\leq C_0(\varepsilon)\sqrt{\delta}.
\end{align}
\end{Lemma}

\begin{Lemma}[The transport part]\label{e:transport estimate n}
\begin{align}
\Big\|\mathcal{R}\Big(\partial_tw_{n}+\sum_{l\in Z^3}v_{0(n-1)}\Big(\frac{l}{\mu_{n}}\Big)\cdot\nabla w_{nl}\Big)\Big\|_0\leq C_0(\varepsilon)\frac{\sqrt{\delta}\mu_n}{\lambda_n},\nonumber\\
\Big\|\mathcal{R}\Big(\partial_tw_{n}+\sum_{l\in Z^3}v_{0(n-1)}\Big(\frac{l}{\mu_{n}}\Big)\cdot\nabla w_{nl}\Big)\Big\|_{C^1_{t,x}}\leq C_0(\varepsilon)\sqrt{\delta}\mu_n.
\end{align}
\end{Lemma}

\begin{Lemma}[Estimate on error part I]\label{e:error n1}
\begin{align}
\|N_n\|_0\leq C_0 \delta\frac{\lambda_{(n-1)}}{\mu_n},\qquad \|N_n\|_{C^1_{t,x}}\leq C_0\lambda_n\delta\frac{\lambda_{(n-1)}}{\mu_n}.
\end{align}
\begin{proof} First, we have
 \begin{align}
 N_{n}=&\sum_{l\in Z^3}\Big[w_{nl}\otimes \Big(v_{0(n-1)}-v_{0(n-1)}\Big(\frac{l}{\mu_n}\Big)\Big)
   +\Big(v_{0(n-1)}-v_{0(n-1)}\Big(\frac{l}{\mu_n}\Big)\Big)\otimes w_{nl}\Big{)}\Big].\nonumber
\end{align}
By (\ref{d:another definition for n perturbation}) , we have
\begin{align}
&\sum_{l\in Z^3}w_{nl}\otimes\Big( v_{0(n-1)}-v_{0(n-1)}\Big(\frac{l}{\mu_n}\Big)\Big)\nonumber\\
=&\sum_{l\in Z^3}\Big(k_{nl}e^{i\lambda_n 2^{[l]} k_n^{\perp}\cdot \big(x-v_{0(n-1)}(\frac{l}{\mu_n})t\big)}+k_{-nl}e^{-i\lambda_1 2^{[l]}k_n^{\perp}\cdot \big(x-v_{0(n-1)}(\frac{l}{\mu_n})t\big)}\Big{)}\otimes\Big( v_{0(n-1)}-v_{0(n-1)}\Big(\frac{l}{\mu_n}\Big)\Big).\nonumber
\end{align}
Obviously, $k_{nl}(x,t)\neq0$ implies $|(\mu_nt,\mu_n x)-l|\leq1$. Moreover, by (\ref{e:induction difference estimate}) and parameter assumption (\ref{a:assumption on parameter}), we get $\|\nabla_{t,x} v_{0(n-1)}\|_0\leq C_0\sqrt{\delta}\lambda_{n-1}$, therefore
$$\Big|k_{nl}\Big(v_{0(n-1)}-v_{0(n-1)}\Big(\frac{l}{\mu_n}\Big)\Big)\Big|\leq C_0\sqrt{\delta}\frac{\|\nabla_{t,x} v_{0(n-1)}\|_0}{\mu_n}\leq C_0\delta\frac{\lambda_{(n-1)}}{\mu_n}.$$
Similarly,
$$\Big|k_{-nl}\Big(v_{0(n-1)}-v_{0(n-1)}\Big(\frac{l}{\mu_n}\Big)\Big)\Big|\leq C_0\sqrt{\delta}\frac{\|\nabla_{t,x} v_{0(n-1)}\|_0}{\mu_n}\leq C_0\delta\frac{\lambda_{(n-1)}}{\mu_n}.$$
Together with (\ref{p:unity}), it is easy to see
\begin{align}
\Big\|\sum_{l\in Z^3}w_{nl}\otimes \Big(v_{0(n-1)}-v_{0(n-1)}\Big(\frac{l}{\mu_n}\Big)\Big)\Big\|_0\leq C_0\delta\frac{\lambda_{(n-1)}}{\mu_n}.\nonumber
\end{align}
Applying the same argument
\begin{align}
\Big\|\sum_{l\in Z^3}\Big(v_{0(n-1)}-v_{0(n-1)}\Big(\frac{l}{\mu_n}\Big)\Big)\otimes w_{nl}\Big\|_0\leq C_0\delta\frac{\lambda_{(n-1)}}{\mu_n}.\nonumber
\end{align}
Thus, we arrive at the first estimate in this lemma.
A straightforward computation gives
\begin{align}
&\partial_t\Big(\sum_{l\in Z^3}w_{nl}\otimes\Big( v_{0(n-1)}-v_{0(n-1)}\Big(\frac{l}{\mu_n}\Big)\Big)\Big)\nonumber\\
=&\sum_{l\in Z^3}\partial_tw_{nl}\otimes\Big( v_{0(n-1)}-v_{0(n-1)}\Big(\frac{l}{\mu_n}\Big)\Big)+\sum_{l\in Z^3}w_{nl}\otimes \partial_tv_{0(n-1)}.\nonumber
\end{align}
Thus, by (\ref{e:induction difference estimate}), parameter assumption (\ref{a:assumption on parameter}) and Corrollary \ref{e:n sequence difference estimate}
\begin{align}
\Big\|\partial_t\Big(\sum_{l\in Z^3}w_{nl}\otimes \Big(v_{0(n-1)}-v_{0(n-1)}\Big(\frac{l}{\mu_n}\Big)\Big)\Big)\Big\|_0\leq C_0\lambda_n\delta\frac{\lambda_{(n-1)}}{\mu_n}.\nonumber
\end{align}
The same argument gives
\begin{align}
\Big\|\nabla\Big(\sum_{l\in Z^3}w_{nl}\otimes \Big(v_{0(n-1)}-v_{0(n-1)}\Big(\frac{l}{\mu_n}\Big)\Big)\Big)\Big\|_0\leq C_0\lambda_n\delta\frac{\lambda_{(n-1)}}{\mu_n},\nonumber\\
\Big\|\partial_t\Big(\sum_{l\in Z^3}\Big(v_{0(n-1)}-v_{0(n-1)}\Big(\frac{l}{\mu_n}\Big)\Big)\otimes w_{nl}\Big)\Big\|_0\leq C_0\lambda_n\delta\frac{\lambda_{(n-1)}}{\mu_n},\nonumber\\
\Big\|\nabla\Big(\sum_{l\in Z^3}\Big(v_{0(n-1)}-v_{0(n-1)}\Big(\frac{l}{\mu_n}\Big)\Big)\otimes w_{nl}\Big)\Big\|_0\leq C_0\lambda_n\delta\frac{\lambda_{(n-1)}}{\mu_n}.\nonumber
\end{align}
Finally, collecting these estimates, we arrive at
\begin{align}
\|N_{n}\|_{C^1_{t,x}}\leq C_0\lambda_n\delta\frac{\lambda_{(n-1)}}{\mu_n}.\nonumber
\end{align}
Thus, the proof of this lemma is complete.
\end{proof}
\end{Lemma}

\begin{Lemma}[Estimates on error part II]\label{e: error n2}
\begin{align}
\|w_{no}\otimes w_{nc}+w_{nc}\otimes w_{no}+w_{nc}\otimes w_{nc}\|_0\leq C_0\frac{\delta\mu_n}{\lambda_n},\nonumber\\
\|w_{no}\otimes w_{nc}+w_{nc}\otimes w_{no}+w_{nc}\otimes w_{nc}\|_{C^1_{t,x}}\leq C_0\delta\mu_n.\nonumber
\end{align}
\end{Lemma}

From Lemma \ref{e:oscillation estimate n step}, Lemma \ref{e:transport estimate n}, Lemma \ref{e:error n1} and Lemma  \ref{e: error n2}, we conclude that
\begin{align}
\|\delta R_{0n}\|_0\leq C_0(\varepsilon)\Big(\frac{\sqrt{\delta}\mu_n}{\lambda_n}+\delta\frac{\lambda_{n-1}}{\mu_n}\Big),\qquad
\|\delta R_{0n}\|_{C^1_{t,x}}\leq C_0(\varepsilon)\lambda_n\Big(\frac{\sqrt{\delta}\mu_n}{\lambda_n}+\delta\frac{\lambda_{n-1}}{\mu_n}\Big).\nonumber
\end{align}

\subsection{Estimate on $\delta f_{0n}$}

 \indent

We first deal with $\delta f_{02}$. Recalling that
\begin{align}
   \delta f_{02}=&\mathcal{G}(\hbox{div}K_2)
   +\mathcal{G}\Big(\partial_t\chi_{2}+\sum_{l\in Z^3}v_{01}\Big(\frac{l}{\mu_{2}}\Big)\cdot\nabla\chi_{2l}\Big)+w_{2o}\chi_{2c}
   +\sum_{l\in Z^3}\Big(v_{01}-v_{01}\Big(\frac{l}{\mu_{2}}\Big)\Big)\chi_{2l}\nonumber\\
   &+w_{2c}\chi_2+\sum_{l\in Z^3}w_{2l}\Big(\theta_{01}-\theta_{01}\Big(\frac{l}{\mu_2}\Big)\Big).\nonumber
\end{align}
As before, we split $\delta f_{02}$ into three parts:\\
(1)The oscillation part:
\begin{align}
\mathcal{G}(\hbox{div}K_2).\nonumber
\end{align}
(2)The transport part:
\begin{align}
\mathcal{G}\Big(\partial_t\chi_{2}+\sum_{l\in Z^3}v_{01}\Big(\frac{l}{\mu_{2}}\Big)\cdot\nabla\chi_{2l}\Big).\nonumber
\end{align}
(3)The error part:
\begin{align}
w_{2c}\chi_2+w_{2o}\chi_{2c}+\sum\limits_{l\in Z^3}\Big(v_{01}-v_{01}\Big(\frac{l}{\mu_{2}}\Big)\Big)\chi_{2l}+\sum_{l\in Z^3}w_{2l}\Big(\theta_{01}-\theta_{01}\Big(\frac{l}{\mu_2}\Big)\Big).\nonumber
\end{align}

\begin{Lemma}[The oscillation part]
\begin{align}
\|\mathcal{G}({\rm div}K_2)\|_0\leq C_0(\varepsilon)\frac{\delta\mu_2}{\lambda_2},\quad\|\mathcal{G}({\rm div}K_2)\|_{C^1_{t,x}}\leq C_0(\varepsilon)\delta\mu_2.\nonumber
\end{align}
\end{Lemma}

\begin{Lemma}[The transport part]
\begin{align}
\Big\|\mathcal{G}\Big(\partial_t\chi_{2}+\sum_{l\in Z^3}v_{01}\Big(\frac{l}{\mu_{2}}\Big)\cdot\nabla\chi_{2l}\Big)\Big\|_0\leq C_0(\varepsilon)\frac{\sqrt{\delta}\mu_2}{\lambda_2},\nonumber\\
\Big\|\mathcal{G}\Big(\partial_t\chi_{2}+\sum_{l\in Z^3}v_{01}\Big(\frac{l}{\mu_{2}}\Big)\cdot\nabla\chi_{2l}\Big)\Big\|_{C^1_{t,x}}\leq C_0(\varepsilon)\sqrt{\delta}\mu_2.\nonumber
\end{align}
\end{Lemma}
Their proofs are same as in section 6.
\begin{Lemma}[The error part]
\begin{align}
\Big\|w_{2c}\chi_2+w_{2o}\chi_{2c}+\sum_{l\in Z^3}\Big(\theta_{01}-\theta_{01}\Big(\frac{l}{\mu_{2}}\Big)\Big)\chi_{2l}\Big\|_0\leq C_0\delta\Big(\frac{\mu_2}{\lambda_2}+\frac{\lambda_1}{\mu_2}\Big),\nonumber\\
\Big\|w_{2c}\chi_2+w_{2o}\chi_{2c}+\sum_{l\in Z^3}\Big(\theta_{01}-\theta_{01}\Big(\frac{l}{\mu_{2}}\Big)\Big)\chi_{2l}\Big\|_{C^1_{t,x}}\leq C_0\lambda_2\delta\Big(\frac{\mu_2}{\lambda_2}+\frac{\lambda_1}{\mu_2}\Big).\nonumber
\end{align}
\begin{proof}
First, by Lemma \ref{e:n step estimate on perturbation and correction}
\begin{align}
\|w_{2c}\chi_2+w_{2o}\chi_{2c}\|_0\leq C_0\delta\frac{\mu_2}{\lambda_2}.\nonumber
\end{align}
As the argument in Lemma \ref{e:error n1} , we have
\begin{align}
\Big\|\sum_{l\in Z^3}\Big(v_{01}-v_{01}\Big(\frac{l}{\mu_{2}}\Big)\Big)\chi_{2l}\Big\|_0\leq C_0\delta\frac{\lambda_1}{\mu_2},\qquad
\Big\|\sum_{l\in Z^3}w_{2l}\Big(\theta_{01}-\theta_{01}\Big(\frac{l}{\mu_2}\Big)\Big)\Big\|_0\leq C_0\delta\frac{\lambda_1}{\mu_2}.\nonumber
\end{align}
The $C^1_{t,x}$ estimate is similar to that of Lemma \ref{e:error n1}.
\end{proof}
\end{Lemma}
From the above three lemmas, we conclude
\begin{align}
\|\delta f_{02}\|_0\leq  C_0(\varepsilon)\Big(\frac{\sqrt{\delta}\mu_2}{\lambda_2}+\delta\frac{\lambda_1}{\mu_2}\Big),\qquad
\|\delta f_{02}\|_{C^1_{t,x}}\leq  C_0(\varepsilon)\lambda_2\Big(\frac{\sqrt{\delta}\mu_2}{\lambda_2}+\delta\frac{\lambda_1}{\mu_2}\Big).\nonumber
\end{align}

Now we consider the estimates of $\delta f_{03}$. From definition (\ref{d:definition f0n}) on $\delta f_{03}$, applying the same argument as in Lemma \ref{e:est 1}, we have
\begin{align}
\|\delta f_{0n}\|_0\leq  C_0 \delta\frac{\lambda_2}{\mu_3},\qquad
\|\delta f_{0n}\|_{C^1_{t,x}}\leq  C_0 \lambda_3\delta\frac{\lambda_2}{\mu_3}.\nonumber
\end{align}
By induction, we know that, for any $1\leq n \leq 3$, $(v_{0n}, p_{0n}, \theta_{0n}, R_{0n}, f_{0n})\in C_c^{\infty}(Q_{r+\delta})$  solves system (\ref{d:boussinesq reynold}) and satisfies
\begin{align}
 R_{0n}=-\sum_{i=n+1}^3 a^2_ik_i\otimes k_i+\sum_{i=i}^n\delta R_{0i},\qquad
    g_{0n}:=\sum_{i=1}^n\delta g_{0i}\nonumber
\end{align}
with the estimates
\begin{align}
 \|v_{0n}-v_{0(n-1)}\|_0\leq \frac{M\sqrt{\delta}}{6}+C_0\frac{\sqrt{\delta}\mu_n}{\lambda_n} ,\quad\|v_{0n}-&v_{0(n-1)}\|_{C^1_{t,x}}\leq C_0\lambda_n\sqrt{\delta},\quad
\|p_{0n}-p_{0(n-1)}\|_0=0,\nonumber\\
 \|\theta_{0n}-\theta_{0(n-1)}\|_0\leq \left\{
 \begin{array}{cc}
 \frac{M\sqrt{\delta}}{4}+C_0\frac{\sqrt{\delta}\mu_2}{\lambda_2} ,~~~n=2\nonumber\\
 0,~~~~~~~~~~ n= 3.
 \end{array}
 \right.&\quad
 \|\theta_{0n}-\theta_{0(n-1)}\|_{C^1_{t,x}}\leq \left\{
 \begin{array}{cc}
 C_0\lambda_2\sqrt{\delta} ,~~~n=2\nonumber\\
 0,~~~~~~~~~~ n= 3.
 \end{array}
 \right.\\
\|\delta R_{0n}\|_0\leq C_0(\varepsilon)\Big(\frac{\sqrt{\delta}\mu_n}{\lambda_n}+\delta\frac{\lambda_{n-1}}{\mu_n}\Big),&\qquad
\|\delta f_{0n}\|_0\leq C_0(\varepsilon)\Big(\frac{\sqrt{\delta}\mu_n}{\lambda_n}+\delta\frac{\lambda_{n-1}}{\mu_n}\Big),\nonumber\\
\|\delta R_{0n}\|_{C^1_{t,x}}\leq C_0(\varepsilon)\lambda_n\Big(\frac{\sqrt{\delta}\mu_n}{\lambda_n}+\delta\frac{\lambda_{n-1}}{\mu_n}\Big),&\qquad
\|\delta f_{0n}\|_{C^1_{t,x}}\leq C_0(\varepsilon)\lambda_n\Big(\frac{\sqrt{\delta}\mu_n}{\lambda_n}+\delta\frac{\lambda_{n-1}}{\mu_n}\Big).\nonumber
\end{align}

Finally, we obtain $(v_{03},~p_{03},~\theta_{03},~R_{03},~f_{03})\in C_c^{\infty}(Q_{r+\delta})$ which solves system (\ref{d:boussinesq reynold}) and satisfies
\begin{align}
\|R_{03}\|_0\leq C_0(\varepsilon)\sum\limits_{n=1}^{3}\Big(\frac{\sqrt{\delta}\mu_n}{\lambda_n}+\delta\frac{\lambda_{n-1}}{\mu_n}\Big),&\quad
\|f_{03}\|_0\leq C_0(\varepsilon)\sum\limits_{n=1}^{3}\Big(\frac{\sqrt{\delta}\mu_n}{\lambda_n}+\delta\frac{\lambda_{n-1}}{\mu_n}\Big),\nonumber\\
\|R_{03}\|_{C^1_{t,x}}\leq C_0(\varepsilon)\sum\limits_{n=1}^{3}\lambda_n\Big(\frac{\sqrt{\delta}\mu_n}{\lambda_n}+\delta\frac{\lambda_{n-1}}{\mu_n}\Big),&\qquad
\|f_{03}\|_{C^1_{t,x}}\leq C_0(\varepsilon)\sum\limits_{n=1}^{3}\lambda_n\Big(\frac{\sqrt{\delta}\mu_n}{\lambda_n}+\delta\frac{\lambda_{n-1}}{\mu_n}\Big),\nonumber\\
\|v_{03}-v_0\|_0\leq\frac{M\sqrt{\delta}}{2}+
C_0\sum\limits_{n=1}^{3}\frac{\sqrt{\delta}\mu_n}{\lambda_n},&\qquad\|v_{03}-v_0\|_{C^1_{t,x}}\leq C_0\sum\limits_{n=1}^{3}\lambda_n\sqrt{\delta},\nonumber\\
\|p_{03}-p_0\|_0\leq M\delta,&\qquad\|p_{03}-p_0\|_{C^1_{t,x}}\leq C_0, \nonumber\\
\|\theta_{03}-\theta_0\|_0\leq\frac{M\sqrt{\delta}}{2}+
C_0\sum\limits_{n=1}^{2}\frac{\sqrt{\delta}\mu_n}{\lambda_n},&\qquad
\|\theta_{03}-\theta_0\|_{C^1_{t,x}}\leq C_0\sum\limits_{n=1}^{2}\lambda_n\sqrt{\delta}.\nonumber
\end{align}

\setcounter{equation}{0}

\section{Proof of Proposition 2.1}
In this section, we prove Proposition \ref{p: iterative 1} by choosing the appropriate parameters
$\ell, \mu_n, \lambda_n$ for $1\leq n\leq 3$.
\begin{proof}
From the results of Section 9, we have $(v_{03},~p_{03},~\theta_{03},~R_{03},~f_{03})\in C_c^{\infty}(Q_{r+\delta})$ which solves system (\ref{d:boussinesq reynold}) and satisfies
\begin{align}\label{e:final summation estimate}
\|R_{03}\|_0\leq C_0(\varepsilon)\sum\limits_{n=1}^{3}\Big(\frac{\sqrt{\delta}\mu_n}{\lambda_n}+\delta\frac{\lambda_{n-1}}{\mu_n}\Big),&\quad
\|f_{03}\|_0\leq C_0(\varepsilon)\sum\limits_{n=1}^{3}\Big(\frac{\sqrt{\delta}\mu_n}{\lambda_n}+\delta\frac{\lambda_{n-1}}{\mu_n}\Big),\nonumber\\
\|R_{03}\|_{C^1_{t,x}}\leq C_0(\varepsilon)\sum\limits_{n=1}^{3}\lambda_n\Big(\frac{\sqrt{\delta}\mu_n}{\lambda_n}+\delta\frac{\lambda_{n-1}}{\mu_n}\Big),&\qquad
\|f_{03}\|_{C^1_{t,x}}\leq C_0(\varepsilon)\sum\limits_{n=1}^{3}\lambda_n\Big(\frac{\sqrt{\delta}\mu_n}{\lambda_n}+\delta\frac{\lambda_{n-1}}{\mu_n}\Big),\nonumber
\end{align}
\begin{align}
\|v_{03}-v_0\|_0\leq\frac{M\sqrt{\delta}}{2}+
C_0\sum\limits_{n=1}^{3}\frac{\sqrt{\delta}\mu_n}{\lambda_n},&\qquad\|v_{03}-v_0\|_{C^1_{t,x}}\leq C_0\sum\limits_{n=1}^{3}\lambda_n\sqrt{\delta},\nonumber\\
\|p_{03}-p_0\|_0\leq M\delta,&\qquad\|p_{03}-p_0\|_{C^1_{t,x}}\leq C_0, \nonumber\\
\|\theta_{03}-\theta_0\|_0\leq\frac{M\sqrt{\delta}}{2}+
C_0\sum\limits_{n=1}^{2}\frac{\sqrt{\delta}\mu_n}{\lambda_n},&\qquad
\|\theta_{03}-\theta_0\|_{C^1_{t,x}}\leq C_0\sum\limits_{n=1}^{2}\lambda_n\sqrt{\delta}.
\end{align}
where $\lambda_0=\Lambda\delta^{-\frac{1}{2}}+\frac{\mu_1\Lambda\ell}{\delta}$.
We divide the remainder proof into four steps:\\
{\bf Step 1}.
We now specify the choice of the parameters. First choose:
\begin{align}\label{d:choose of parameter l}
\ell=\frac{1}{L_v}\frac{\bar{\delta}}{\Lambda},
\end{align}
with $L_v$ being a sufficiently large constant, which depends only on $\|v\|_0$ and will be chosen in Step 3 below.\\
Next, we impose
\begin{align}
 \mu_1=L_v\frac{\sqrt{\delta}}{\bar{\delta}}\Lambda,\quad \lambda_1=L_v\frac{\sqrt{\delta}}{\bar{\delta}}\mu_1^{1+\varepsilon},\quad \mu_i=L_v\frac{\delta\lambda_{i-1}}{\bar{\delta}},\quad
\lambda_i=L_v\frac{\sqrt{\delta}}{\bar{\delta}}\mu_i^{1+\varepsilon},\quad i=2,3.
\end{align}
{\bf Step 2. Compatibility condition.} We check that all the conditions in (\ref{a:assumption on parameter}), (\ref{a:assumption on parameter m sequence}) are satisfied by our choice of the parameters.

We first check the triple $(\ell, \mu_1, \lambda_1)$. By (\ref{d:choose of parameter l}) $$\ell^{-1}=L_v\frac{\Lambda}{\bar{\delta}}\geq\frac{\Lambda}{\eta\delta}$$
if we take $L_v\geq\frac{1}{\eta}$.\\
Since $\bar{\delta}\leq \delta^{\frac{3}{2}}$, $$\mu_1\geq \frac{\Lambda}{\delta}$$
and
 $$\lambda_1\geq L_v^{1+\varepsilon}\Lambda^{1+\varepsilon}\Big(\frac{\sqrt{\delta}}{\bar{\delta}}\Big)^{2+\varepsilon}\geq \Big(L_v\frac{\Lambda}{\bar{\delta}}\Big)^{1+\varepsilon}\geq \ell^{-(1+\varepsilon)}.$$
It's obvious that
$$\lambda_1\geq{\mu_1^{1+\varepsilon}}.$$
Thus, (\ref{a:assumption on parameter}) is satisfied.

Next, for $i=2,3$, it's obvious that $$\frac{\mu_i}{\mu_{i-1}}=\frac{\lambda_{i-1}}{\lambda_{i-2}}>1.$$
A straightforward computation yield
 $$\lambda_i\geq\frac{\sqrt{\delta}}{\bar{\delta}}\Big(\frac{\delta}{\bar{\delta}}\Big)^{1+\varepsilon}\lambda_{i-1}^{1+\varepsilon}\geq \lambda_{i-1}\geq \ell^{-(1+\varepsilon)}.$$
It's obvious that
$$\lambda_i\geq\mu_i^{1+\varepsilon}.$$
Thus, the relationship (\ref{a:assumption on parameter m sequence}) is satisfied.\\
{\bf Step 3. $C^0$ estimates} In the following, $\varepsilon$ ia a parameter that is small, but fixed, and our constants will be allowed to depend on $\varepsilon$. Thus, (\ref{e:final summation estimate}) implies
\begin{align}
\|R_{03}\|_0\leq C_0(\varepsilon) \bar{\delta}L_v^{-1},\quad&
\|f_{03}\|_0\leq C_0(\varepsilon) \bar{\delta}L_v^{-1},\nonumber\\
\|v_{03}-v_0\|_0\leq \frac{M\sqrt{\delta}}{2}+C_0\bar{\delta}L_v^{-1},\quad
\|\theta_{03}-\theta_0\|_0&\leq \frac{M\sqrt{\delta}}{2}+C_0\bar{\delta}L_v^{-1},\quad
\|p_{03}-p_0\|_0\leq M\delta.\nonumber
\end{align}
Choosing $L_v$ sufficiently large which depending on $\|v\|_0$ linearly,  we can achieve the desired inequalities
(\ref{e:iterative stress estimate})-(\ref{e:iterative pressure difference estimate}).\\
{\bf Step 4. $C^1$ estimates.}
By the specified choices of parameters we have
\begin{align}
\|R_{03}\|_{C^1}\leq \lambda_3\bar{\delta},\quad
\|f_{03}\|_{C^1}\leq \lambda_3\bar{\delta},\quad
\|v_{03}-v_0\|_{C^1_{t,x}}\leq C_0\lambda_3\sqrt{\delta},\quad
\|\theta_{03}-\theta_0\|_{C^1_{t,x}}\leq C_0\lambda_3\sqrt{\delta}.\nonumber
\end{align}
Notice that for $i=2,3$, there holds
$$\lambda_i=L_v^{2+\varepsilon}\frac{\sqrt{\delta}}{\bar{\delta}}\Big(\frac{\delta}{\bar{\delta}}\Big)^{1+\varepsilon}\lambda_{i-1}^{1+\varepsilon}
=\frac{1}{\sqrt{\delta}}\Big(\frac{L_v\delta}{\bar{\delta}}\Big)^{2+\varepsilon}\lambda_{i-1}^{1+\varepsilon}.$$
Thus, we conclude
\begin{align}
&\max\{1,\|R_{03}\|_{C^1}, \|f_{03}\|_{C^1}, \|v_{03}\|_{C^1}, \|\theta_{03}\|_{C^1}\}\leq \Lambda+C_0(\varepsilon)\sqrt{\delta}\lambda_3\nonumber\\
\leq& C_0(\varepsilon)L_v^{(1+\varepsilon)^{2}(2+\varepsilon)+(2+\varepsilon)^2}(\sqrt{\delta})^{\varepsilon^2+3\varepsilon+3}
\Big(\frac{\sqrt{\delta}}{\bar{\delta}}\Big)^{(1+\varepsilon)^{2}(2+\varepsilon)+(2+\varepsilon)^2}
\Lambda^{(1+\varepsilon)^{3}}.\nonumber
\end{align}
Setting $A=C_0(\varepsilon)L_v^{(1+\varepsilon)^{2}(2+\varepsilon)+(2+\varepsilon)^2}$, we conclude estimate (\ref{e:first derivative estimate for stress term}).\\
More precisely, we have
\begin{align}
\|\theta_{03}\|_{C^1_{t,x}}\leq \Lambda+C_0(\varepsilon)\sqrt{\delta}\lambda_2&\leq
C_0(\varepsilon)L_v^{4+4\varepsilon+\varepsilon^2}
(\sqrt{\delta})^{2+\varepsilon}
\Big(\frac{\sqrt{\delta}}{\bar{\delta}}\Big)^{4+4\varepsilon+\varepsilon^2}
\Lambda^{(1+\varepsilon)^2}\nonumber\\
&\leq A\delta^{\frac{2+\varepsilon}{2}}
\Big(\frac{\sqrt{\delta}}{\bar{\delta}}\Big)^{4+4\varepsilon+\varepsilon^2}
\Lambda^{(1+\varepsilon)^2
},\nonumber
\end{align}
this is the second estimate in (\ref{e:terperture derivative estimates}).

Finally, we set
$$\tilde{V}:=v_{03},~\tilde{p}:=p_{03},~\tilde{\theta}:=\theta_{03},~\tilde{R}:=R_{03},~\tilde{f}:=f_{03},$$
then $\tilde{V},~\tilde{p},~\tilde{\theta},~\tilde{R},~\tilde{f}$ are what we need in our Proposition (\ref{p: iterative 1}).
\end{proof}

\setcounter{equation}{0}

\setcounter{equation}{0}
\section{Proof of theorem 1.1}

\indent

We first construct a non-trival solution $v_0, p_0, \theta_0, R_0, f_0$ with compact support both in space and time for system (\ref{d:boussinesq reynold}).

\subsection{Construction of compactly supported solution ($v_0, p_0, \theta_0, R_0, f_0$) for system (\ref{d:boussinesq reynold})}

\indent

We first set $k_1:=(1,0)^T$ and let
$$0\leq \varphi(t,x)\in C_c^{\infty}(Q_{r};R),\quad \varphi(t,x)=10M \quad{\rm in}\quad Q_{\frac{r}{2}},$$
 where $r,M$ is the constant appeared in Proposition (\ref{p: iterative 1}). Then set
\begin{align}
\bar{p}(t,x):=-2\varphi^2(t,x),\quad a_{1l}(t,x):=\varphi(t,x)\alpha_l(\mu_1t,\mu_1x),\quad \bar{R}(t,x):=\left(\begin{array}{cc}
-2\varphi^2(t,x) & 0 \\
0     &    -2\varphi^2(t,x)
\end{array}\right).\nonumber
\end{align}
Here $\alpha_l$ is the partition of unity in section 4. Obviously, $\nabla\bar{p}={\rm div}\bar{R}.$\\
We first set
\begin{align}\label{d:definition on first perturbation appromation}
v_{01ol}(t,x):=&-a_{1l}(t,x) k_1\Big(ie^{i\lambda_12^{|l|}k_1^{\perp}\cdot x}-ie^{-i\lambda_12^{|l|}k_1^{\perp}\cdot x}\Big),\nonumber\\
v_{01cl}(t,x):=&\nabla^{\perp}(\nabla a_{1l}(t,x)\cdot k_1^{\perp})\Big(\frac{ie^{i\lambda_12^{|l|}k_1^{\perp}\cdot x}
-ie^{-i\lambda_12^{|l|}k_1^{\perp}\cdot x}}{\lambda_1^22^{2|l|}}\Big)
-\nabla^{\perp}a_{1l}(t,x)\Big(\frac{e^{i\lambda_12^{|l|}k_1^{\perp}\cdot x}+e^{-i\lambda_12^{|l|}k_1^{\perp}\cdot x}}{\lambda_12^{|l|}}\Big)\nonumber\\
&-\nabla a_{1l}(t,x)\cdot k_1^{\perp}\Big(\frac{k_1e^{i\lambda_12^{|l|}k_1^{\perp}\cdot x}+k_1e^{-i\lambda_12^{|l|}k_1^{\perp}\cdot x}}{\lambda_12^{|l|}}\Big),
\end{align}
where $|l|$ is the length of $l$, $\mu_1\ll\lambda_1$ are two positive numbers which will be chosen quite large, depending on appropriate norms of $\varphi$.\\
Then set
\begin{align}
v_{01l}:=v_{01ol}+v_{01cl},\quad v_{01o}:=\sum\limits_{l\in Z^3}v_{01ol},\quad v_{01c}:=\sum\limits_{l\in Z^3}v_{01cl},\quad v_{01}:=\sum\limits_{l\in Z^3}v_{01l}.\nonumber
\end{align}
Thus, a straightforward computation gives
\begin{align}
v_{01}(t,x):=&\sum\limits_{l\in Z^3}\nabla^{\perp}\hbox{div}\Big(a_{1l}(t,x)\Big(\frac{ik_1^{\perp}e^{i\lambda_12^{|l|}k_1^{\perp}\cdot x}-ik_1^{\perp}e^{-i\lambda_12^{|l|}k_1^{\perp}\cdot x}}{\lambda_1^22^{2|l|}}\Big)\Big).\nonumber
\end{align}
Let $b(t,x)\in C_c^{\infty}(Q_r;R)$ and set
\begin{align}
\theta_{01o}(t,x):=&-b(t,x)(e^{i\lambda_1k_1^{\perp}\cdot x}+e^{-i\lambda_1k_1^{\perp}\cdot x}),\nonumber\\
\theta_{01c}(t,x):=&\triangle b(t,x)\frac{e^{i\lambda_1k_1^{\perp}\cdot x}+e^{-i\lambda_1k_1^{\perp}\cdot x}}{\lambda_1^2}
+2\nabla b(t,x)\cdot k_1^{\perp}\frac{ie^{i\lambda_1k_1^{\perp}\cdot x}-ie^{-i\lambda_1k_1^{\perp}\cdot x}}{\lambda_1}.\nonumber
\end{align}
Then denote $\theta_{01}$ by
\begin{align}
\theta_{01}:=&\theta_{01c}+\theta_{01o}.\nonumber
\end{align}
Thus
\begin{align}
\theta_{01}(t,x)=\triangle\Big(b(t,x)\Big(\frac{e^{i\lambda_1k_1^{\perp}\cdot x}+e^{-i\lambda_1k_1^{\perp}\cdot x}}{\lambda_1^2}\Big)\Big).\nonumber
\end{align}
Finally, we set
\begin{align}
p_{01}:=&\bar{p},\quad
R_{01}:=\bar{R}+2\sum\limits_{l\in Z^3}a^2_{1l}k_1\otimes k_1+\delta R_{01},\nonumber\\
f_{01}:=&v_{01o}\theta_{01c}+v_{01c}\theta_{01o}+v_{01c}\theta_{01c}
+\mathcal{G}\Big(\partial_t\theta_{01}+{\rm div}(v_{01o}\theta_{01o})\Big),\nonumber
\end{align}
where
\begin{align}
\delta R_{01}=&v_{01o}\otimes v_{01c}+v_{01c}\otimes v_{01o}+v_{01c}\otimes v_{01c}+\mathcal{R}\Big(\partial_tv_{01}\nonumber\\
&+{\rm div}\Big(-\sum\limits_{l\in Z^3}a^2_{1l}k_1\otimes k_1\Big(e^{2i\lambda_12^{|l|}k_1^{\perp}\cdot x}+e^{-i\lambda_12^{|l|}k_1^{\perp}\cdot x}\Big)+\sum\limits_{l,l'\in Z^3,l\neq l'}v_{01ol}\otimes v_{01ol'}\Big)
-\theta_{01}e_2\Big).\nonumber
\end{align}
Obviously, ${\rm div}v_{01}=0$ and $\partial_tv_{01},~~~ \theta_{01}e_2\in\Xi,~~~ \partial_t\theta_{01}+{\rm div}(v_{01o}\theta_{01o})\in\Psi$, thus $R_{01}, f_{01}$ is well-defined. By Proposition \ref{p:proposition R} and Proposition \ref{p:inverse 2}, we know that $(v_{01}, p_{01}, \theta_{01}, R_{01}, f_{01})\in C^\infty_c(Q_r)$ solves Boussinesq-stress system (\ref{d:boussinesq reynold}).  In fact, by Proposition \ref{p:proposition R}, we have
\begin{align}
{\rm div}R_{01}=\partial_t v_{01}+{\rm div}(v_{01}\otimes v_{01})+\nabla p_{01}-\theta_{01}e_2,\nonumber
\end{align}
where we use the identity
\begin{align}
&{\rm div}\Big\{-\sum\limits_{l\in Z^3}a^2_{1l}k_1\otimes k_1\Big(e^{2i\lambda_12^{|l|}k_1^{\perp}\cdot x}+e^{-i\lambda_12^{|l|}k_1^{\perp}\cdot x}\Big)+2\sum\limits_{l\in Z^3}a^2_{1l}k_1\otimes k_1+\sum\limits_{l,l'\in Z^3,l\neq l'}v_{01ol}\otimes v_{01ol'}\Big\}\nonumber\\
=&{\rm div}(v_{01o}\otimes v_{01o}),\quad {\rm div}\bar{R}=\nabla p_{01}.\nonumber
\end{align}
Using Proposition \ref{p:inverse 2}, we have
\begin{align}
{\rm div}f_{01}=\partial_t \theta_{01}+{\rm div}(v_{01}\theta_{01}).\nonumber
\end{align}
Thus, $(v_{01}, p_{01}, \theta_{01}, R_{01}, f_{01})$ solves Boussinesq-stress system (\ref{d:boussinesq reynold}).
Furthermore, we have
\begin{align}
\bar{R}+2\sum\limits_{l\in Z^3}a^2_{1l}k_1\otimes k_1=\left(\begin{array}{cc}
-2\varphi^2 & 0 \\
0     &    -2\varphi^2
\end{array}\right)+\left(\begin{array}{cc}
2\varphi^2 & 0 \\
0     &    0
\end{array}\right)=\left(\begin{array}{cc}
0 & 0 \\
0     &    -2\varphi^2
\end{array}\right),\nonumber
\end{align}
therefore
\begin{align}\label{f:form R01}
R_{01}=\left(\begin{array}{cc}
0 & 0 \\
0     &    -2\varphi^2
\end{array}\right)+\delta R_{01}.
\end{align}
We claim $\delta R_{01},~g_{01}$ can be arbitrarily small by choosing appropriate $\mu_1$ and $\lambda_1$. In fact,
\begin{align}
\|v_{01c}\|_0\leq C_1\frac{\mu_1}{\lambda_1},\quad\|\theta_{01c}\|_0
\leq C_1\frac{\mu_1}{\lambda_1},\quad\|v_{01o}\|_0\leq C_1,\quad\|\theta_{01o}\|_0\leq C_1.\nonumber
\end{align}
Here and sebsequent, $C_1$ is an absolute constant which depends on functions $b,~~ \varphi$. Therefore
\begin{align}\label{e:bound correct}
\|v_{01o}\otimes v_{01c}+v_{01c}\otimes v_{01o}+v_{01c}\otimes v_{01c}\|_0\leq C_1\frac{\mu_1}{\lambda_1},\quad \|v_{01o}\theta_{01c}+v_{01c}\theta_{01o}+v_{01c}\theta_{01c}\|_0\leq C_1\frac{\mu_1}{\lambda_1}.
\end{align}
 Moreover,
by Proposition \ref{p:proposition R} and the same argument as Lemma \ref{e:oscillate estimate}, we have
 \begin{align}\label{b:bound RR}
 &\Big\|\mathcal{R}\Big(\partial_tv_{01}+{\rm div}\Big(-\sum\limits_{l\in Z^3}a^2_{1l}k_1\otimes k_1\Big(e^{2i\lambda_12^{|l|}k_1^{\perp}\cdot x}+e^{-i\lambda_12^{|l|}k_1^{\perp}\cdot x}\Big)+\sum\limits_{l,l'\in Z^3,l\neq l'}v_{01ol}\otimes v_{01ol'}\Big)
-\theta_{01}e_2\Big)\Big\|_0\nonumber\\
\leq& C_1\frac{\mu_1}{\lambda_1}.
 \end{align}
Similarly, by Proposition \ref{p:inverse 2}, we obtain
 \begin{align}\label{e:bound G}
 \Big\|\mathcal{G}\Big(\partial_t\theta_{01}+{\rm div}(v_{01o}\theta_{01o})\Big)\Big\|_0\leq C_1\frac{\mu_1}{\lambda_1}.
 \end{align}
Thus, combining (\ref{e:bound correct}), (\ref{b:bound RR}) and (\ref{e:bound G}), we arrive at
 \begin{align}
 \|f_{01}\|_0\leq C_1 \frac{\mu_1}{\lambda_1},~~\|\delta R_{01}\|_0\leq C_1 \frac{\mu_1}{\lambda_1}.\nonumber
 \end{align}
Hence $\delta R_{01},~g_{01}$ can be arbitrarily small by choosing appropriate $\mu_1,~\lambda_1$.

 Take $k_2:=(0,1)^T$, $a_{2l}(t,x):=\varphi(t,x)\alpha_l(\mu_2t,\mu_2x)$ and set
 \begin{align}\label{d:definition on second perturbation appromation}
 w_{2ol}(t,x):=&-a_{2l}(t,x)k_2\Big(ie^{i\lambda_22^{|l|}k_2^{\perp}\cdot \big(x-v_{01}(\frac{l}{\mu_2})t\big)}-ie^{-i\lambda_22^{|l|}k_1^{\perp}\cdot \big(x-v_{01}(\frac{l}{\mu_2})t\big)}\Big),\nonumber\\
 w_{2cl}(t,x):=&\nabla^{\perp}(\nabla a_{2l}(t,x)\cdot k_2^{\perp})\Big(\frac{ie^{i\lambda_22^{|l|}k_2^{\perp}\cdot \big(x-v_{01}(\frac{l}{\mu_2})t\big)}-ie^{-i\lambda_22^{|l|}k_2^{\perp}\cdot \big(x-v_{01}(\frac{l}{\mu_2})t\big)}}{\lambda_2^22^{2|l|}}\Big)\nonumber\\
&-\nabla^{\perp}a_{2l}(t,x)\Big(\frac{e^{i\lambda_22^{|l|}k_2^{\perp}\cdot \big(x-v_{01}(\frac{l}{\mu_2})t\big)}+e^{-i\lambda_22^{|l|}k_2^{\perp}\cdot x}}{\lambda_22^{|l|}}\Big)\nonumber\\
&-\nabla a_{2l}(t,x)\cdot k_2^{\perp}\Big(\frac{k_2e^{i\lambda_22^{|l|}k_2^{\perp}\cdot \big(x-v_{01}(\frac{l}{\mu_2})t\big)}+k_2e^{-i\lambda_22^{|l|}k_2^{\perp}\cdot \big(x-v_{01}(\frac{l}{\mu_2})t\big)}}{\lambda_22^{|l|}}\Big),\nonumber\\
w_{2l}:=&w_{2ol}+w_{2cl},\quad w_{2o}:=\sum\limits_{l\in Z^3}w_{2ol},\quad w_{2c}:=\sum\limits_{l\in Z^3}w_{2cl},\quad w_{2}:=w_{2o}+w_{2c},
 \end{align}
 where $\mu_2\ll\lambda_2$ are two positive numbers which will be chosen quite large, depending on appropriate norms of $v_{01}, \theta_{01}$.
Then, a straightforward computation gives
 \begin{align}
w_{2}:=&\sum\limits_{l\in Z^3}\nabla^{\perp}\hbox{div}\Big(a_{2l}\Big(\frac{ik_2^{\perp}e^{i\lambda_22^{|l|}k_2^{\perp}\cdot \big(x-v_{01}(\frac{l}{\mu_2})t\big)}-ik_2^{\perp}e^{-i\lambda_22^{|l|}k_2^{\perp}\cdot \big(x-v_{01}(\frac{l}{\mu_2})t\big)}}{\lambda_2^22^{2|l|}}\Big)\Big).\nonumber
\end{align}

Finally, we set
\begin{align}
v_{02}:=v_{01}+w_{2},\quad
\theta_{02}:=\theta_{01},&\quad
p_{02}:=p_{01}.\nonumber\\
R_{02}:=R_{01}+2\sum\limits_{l\in Z^3}a^2_{2l}k_2\otimes k_2+\delta R_{02},&\quad
f_{02}:=f_{01}+\delta f_{02},\nonumber
\end{align}
where
\begin{align}\label{r:representation of error term}
\delta R_{02}=&w_{2o}\otimes w_{2c}+w_{2c}\otimes w_{2o}+w_{2c}\otimes w_{2c}+w_{2c}\otimes v_{01}+v_{01}\otimes w_{2c}+\mathcal{R}\Big(\partial_tw_{2}+{\rm div}\Big(w_{2o}\otimes v_{01}\Big(\frac{l}{\mu_2}\Big)\nonumber\\
&+v_{01}\Big(\frac{l}{\mu_2}\Big)\otimes w_{2o} \Big)\Big)
+{\rm div}\Big(-\sum\limits_{l\in Z^3}a^2_{2l}k_2\otimes k_2\Big(e^{2i\lambda_22^{|l|}k_2^{\perp}\cdot x}+e^{-i\lambda_22^{|l|}k_2^{\perp}\cdot x}\Big)\nonumber\\
&+\sum\limits_{l,l'\in Z^3,l\neq l'}w_{2ol}\otimes w_{2ol'}\Big)
+w_{2o}\otimes \Big(v_{01}-v_{01}\Big(\frac{l}{\mu_2}\Big)\Big)+\Big(v_{01}-v_{01}\Big(\frac{l}{\mu_2}\Big)\Big)\otimes w_{2o},\nonumber\\
\delta f_{02}=&\mathcal{G}(w_{2}\cdot\nabla \theta_{01}).
\end{align}
Obviously, ${\rm div}v_{02}=0$ and $\partial_tw_{2}\in\Xi$,
thus by Proposition \ref{p:proposition R} and Proposition \ref{p:inverse 2}, $R_{02}, f_{02}$ are well-defined and $(v_{02}, p_{02}, \theta_{02}, R_{02}, f_{02})\in C^\infty_c(Q_r)$ solves Boussinesq-stress system (\ref{d:boussinesq reynold}).  In fact,
\begin{align}
{\rm div}R_{02}=&{\rm div}R_{01}+{\rm div}(w_{2o}\otimes w_{2o}+w_{2o}\otimes w_{2c}+w_{2c}\otimes w_{2o}+w_{2c}\otimes w_{2c}+w_{2o}\otimes v_{01}\nonumber\\
&+v_{01}\otimes w_{2o}+w_{2c}\otimes v_{01}+v_{01}\otimes w_{2c})+\partial_tw_{2}\nonumber\\
=&\partial_tv_{02}+{\rm div}(v_{02}\otimes v_{02})+\nabla p_{02}-\theta_{02}e_2,\nonumber
\end{align}
where we use
\begin{align}
&{\rm div}\Big\{-\sum\limits_{l\in Z^3}a^2_{2l}k_2\otimes k_2\Big(e^{2i\lambda_22^{|l|}k_2^{\perp}\cdot x}+e^{-i\lambda_22^{|l|}k_2^{\perp}\cdot x}\Big)+2\sum\limits_{l\in Z^3}a^2_{2l}k_2\otimes k_2+\sum\limits_{l,l'\in Z^3,l\neq l'}w_{2ol}\otimes w_{2ol'}\Big\}\nonumber\\
=&{\rm div}(w_{2o}\otimes w_{2o}).\nonumber
\end{align}
And
\begin{align}
{\rm div}f_{02}={\rm div}f_{01}+{\rm div}(w_{2}\theta_{01})=\partial_t \theta_{01}+{\rm div}(v_{01}\theta_{01}+w_{2}\theta_{01})=\partial_t \theta_{02}+{\rm div}(v_{02}\theta_{02}).\nonumber
\end{align}
We claim $\delta R_{02},~\delta g_{02}$ can be arbitrarily small by choosing appropriate $\mu_2,$ and $\lambda_2$. In fact,
\begin{align}
\|w_{2c}\|_0\leq C_2\frac{\mu_2}{\lambda_2},~~\|w_{2o}\|_0\leq C_2.\nonumber
\end{align}
Here and subsequent, $C_2$ is an constant which depends on appropriate norms of $v_{01}, \theta_{01}$. Therefore
\begin{align}\label{e:bound correct 2}
\|w_{2o}\otimes w_{2c}+w_{2c}\otimes w_{2o}+w_{2c}\otimes w_{2c}+w_{2c}\otimes v_{01}+v_{01}\otimes w_{2c}\|_0\leq C_2\frac{\mu_2}{\lambda_2}.
\end{align}
By the same argument as lemma \ref{e:oscillate estimate}, Lemma \ref{e:trans estimate}, we obtain
 \begin{align}
 &\Big\|\mathcal{R}\Big\{{\rm div}\Big(-\sum\limits_{l\in Z^3}a^2_{2l}k_2\otimes k_2\Big(e^{2i\lambda_22^{|l|}k_2^{\perp}\cdot x}+e^{-i\lambda_22^{|l|}k_2^{\perp}\cdot x}\Big)+\sum\limits_{l,l'\in Z^3,l\neq l'}w_{2ol}\otimes w_{2ol'}\Big)\Big\}
\Big\|_0\leq C_2\frac{\mu_2}{\lambda_2},\nonumber\\
&\Big\|\mathcal{R}\Big(\partial_tw_{2o}+v_{01}\Big(\frac{l}{\mu_1}\Big)\cdot\nabla w_{2o}\Big)
\Big\|_0\leq C_2\frac{\mu_2}{\lambda_2},\quad
\|\mathcal{R}(\partial_tw_{2c})\|_0\leq C_2\frac{\mu_2}{\lambda_2}.\nonumber
 \end{align}
Moreover, $a_{2l}(t,x)\neq0$ implies $|(\mu_2t,\mu_2x)-l|\leq 1$, therefore
\begin{align}
\Big\|w_{2o}\otimes \Big(v_{01}-v_{01}\Big(\frac{l}{\mu_2}\Big)\Big)+\Big(v_{01}-v_{01}\Big(\frac{l}{\mu_2}\Big)\Big)\otimes w_{2o}\Big\|_0\leq\frac{C_2}{\mu_2}.\nonumber
\end{align}
Collecting the above estimates, we arrive at
\begin{align}
\|\delta R_{02}\|_0\leq\frac{C_2}{\mu_2}+C_2\frac{\mu_2}{\lambda_2}.
\end{align}
By Proposition \ref{p:inverse 2} and (\ref{r:representation of error term}), we have
\begin{align}
\|\delta f_{02}\|_0\leq \frac{C_2}{\lambda_2}.
\end{align}
Moreover, it's obvious that
\begin{align}
2\sum\limits_{l\in Z^3}a^2_{2l}(t,x)k_2\otimes k_2=\left(\begin{array}{cc}
0 & 0 \\
0     &    2\varphi^2(t,x)
\end{array}\right).\nonumber
\end{align}
Thus, by (\ref{f:form R01}),
\begin{align}
R_{01}+2\sum\limits_{l\in Z^3}a^2_{2l}k_2\otimes k_2=\delta R_{01}.\nonumber
\end{align}
Finally, we have
\begin{align}
R_{02}=\delta R_{01}+\delta R_{02},\quad f_{02}=f_{01}+\delta f_{02}\nonumber
\end{align}
and
\begin{align}
\|R_{02}\|_0\leq\frac{C_2}{\mu_2}+ C_2\frac{\mu_2}{\lambda_2}+C_1 \frac{\mu_1}{\lambda_1},\qquad
\|f_{02}\|_0\leq C_1 \frac{\mu_1}{\lambda_1}+\frac{C_2}{\lambda_2}.\nonumber
\end{align}
Next, we claim $\|v_{02}\|_0\geq 10M$. In fact,
\begin{align}
v_{02}=v_{01}+w_2=v_{01o}+w_{2o}+v_{01c}+w_{2c}.\nonumber
\end{align}
By (\ref{d:definition on first perturbation appromation}) and (\ref{d:definition on second perturbation appromation})
\begin{align}
v_{01o}
=2\sum\limits_{l\in Z^3}a_{1l}\Big(\sin(\lambda_12^{|l|}x_2),0\Big)^T,\quad
w_{2o}
=-2\sum\limits_{l\in Z^3}a_{2l}\Big(0,\sin\Big(\lambda_22^{|l|}(x_1-
v^1_{01}\Big(\frac{l}{\mu_2}\Big)t\Big)\Big)^T,\nonumber
\end{align}
where $v^1_{01}$ is the first component of $v_{01}$.
If we set $(t,x_1,x_2):=(0, \frac{\pi}{2\lambda_2}, 0)$ and take $1\ll\mu_2\ll\lambda_2$, then
\begin{align}
v_{01o}(t,x_1,x_2)=&0,\nonumber\\
w_{2o}(t,x_1,x_2)
=&-2\sum\limits_{l\in Z^3}\varphi\Big(0,\frac{\pi}{2\lambda_2},0\Big)\alpha_l\Big(0,\frac{\mu_2\pi}{2\lambda_2}, 0\Big)\Big(0,sin\Big(\frac{2^{|l|}\pi}{2}\Big)\Big)^T=(0,-20M)^T.\nonumber
\end{align}
Moreover, we can take $1\ll\mu_1\ll\lambda_1\ll\mu_2\ll\lambda_2$ such that
\begin{align}
\|v_{01c}\|_0+\|w_{2c}\|_0\leq M.\nonumber
\end{align}
Therefore, we conclude that
\begin{align}
\|v_{02}\|_0\geq 10M.\nonumber
\end{align}
Finally, we set $(v_0,~p_0,~\theta_0,~R_0,~f_0):=(v_{02},~p_{02},~\theta_{02},~R_{02},~f_{02})$.

In conclusion, for any $ M>0,~ r>0$, we can construct function  $(v_0,~p_0,~\theta_0,~R_0,~f_0)\in C^{\infty}_{c}(Q_{r})$ which solves Boussinesq-stress system (\ref{d:boussinesq reynold}) and satisfies the following estimates
\begin{align}
\|R_{0}\|_0\leq\frac{C_2}{\mu_2}+ C_2\frac{\mu_2}{\lambda_2}+C_1 \frac{\mu_1}{\lambda_1},\quad
\|f_{0}\|_0\leq C_1 \frac{\mu_1}{\lambda_1}+\frac{C_2}{\lambda_2},\quad
\|v_{0}\|_0\geq 10M.\nonumber
\end{align}

\subsection{Proof of Theorem 1.1}
\begin{proof}
From subsection 11.1, we know that for any $r>0$, there exists $(v_0,~p_0,~\theta_0,~R_0,~f_0)\in C^{\infty}_{c}(Q_{r})$ which solves Boussinesq-stress system (\ref{d:boussinesq reynold}) and satisfies the following estimates:
\begin{align}\label{e:initial data estimate}
\|R_{0}\|_0\leq\frac{C_2}{\mu_2}+ C_2\frac{\mu_2}{\lambda_2}+C_1 \frac{\mu_1}{\lambda_1},\quad
\|f_{0}\|_0\leq C_1 \frac{\mu_1}{\lambda_1}+\frac{C_2}{\lambda_2},\quad
\|v_{0}\|_0\geq 10M.\nonumber
\end{align}
Take $a,b \geq\frac{3}{2}$ such that $\frac{1}{a}\leq\min\{\frac{r}{2},\frac{\varepsilon^2}{16M^2}\}$ and set $\delta_n:=a^{-b^n}: n=0,1,2,\cdot\cdot\cdot$ , and we have
\beno
\delta_{n+1}=a^{-b^{n+1}}=(a^{-b^{n}})^b=(\delta^n)^b\leq \frac12 (\delta^n)^{\frac32},
\eeno
where we used $b\geq \frac32$ and $\delta^n\ll 1.$
 Then taking $\mu_1,\mu_2,\lambda_1,\lambda_2$ with $1\ll\mu_1\ll\lambda_1\ll\mu_2\ll\lambda_2$ such that
\begin{align}
\|R_{0}\|_0\leq\eta\delta_0,~~
\|f_{0}\|_0\leq\eta\delta_0.\nonumber
\end{align}
Applying Proposition \ref{p: iterative 1} iteratively, we can construct
 $$(v_{n},~p_{n},~\theta_{n},~R_{n},~f_{n})\in C_c^\infty(Q_{r+\sum_{i=0}^n\delta_i}), \quad n=1,2,\cdot\cdot\cdot$$
such that they solve system \ref{d:boussinesq reynold} and satisfy the following estimates
\begin{align}
\|R_{n}\|_0\leq&\eta\delta_n,\\
\|f_{n}\|_0\leq &\eta\delta_n,\\
\|v_{n+1}-v_{n}\|_0 \leq& M\sqrt{\delta_n},\label{e:velocity_final}\\
\|\theta_{n+1}-\theta_{n}\|_0 \leq& M\sqrt{\delta_n}\label{e:temperature_final},\\
\|p_{n+1}-p_{n}\|_0 \leq& M\delta_n\label{e:pressure_final},\\
\Lambda_{n+1}:=\max\{1, \|R_{n+1}\|_{C^1_{t,x}},& \|f_{n+1}\|_{C^1_{t,x}}, \|v_{n+1}\|_{C^1_{t,x}}, \|\theta_{n+1}\|_{C^1_{t,x}}\}\nonumber\\
\leq A(\sqrt{\delta_n})^{\varepsilon^2+3\varepsilon+3}&\Big(\frac{\sqrt{\delta_n}}{\delta_{n+1}}\Big)^{(1+\varepsilon)^{2}(2+\varepsilon)+(2+\varepsilon)^2}
\Lambda_n^{(1+\varepsilon)^{3}} .
\end{align}
In particular,
\begin{align}\label{e:final pressure and termperture estimate}
\|p_n\|_{C^1_{t,x}}&\leq C_0,\qquad\|\theta_{n+1}\|_{C^1_{t,x}}\leq A
(\sqrt{\delta_n})^{2+\varepsilon}
\Big(\frac{\sqrt{\delta_n}}{\delta_{n+1}}\Big)^{4+4\varepsilon+\varepsilon^2}
\Lambda_n^{(1+\varepsilon)^2
}.
\end{align}
where $A$ depends on $r+\delta_n, \varepsilon, \|v_n\|$ linearly and $\varepsilon$. It is obvious that $\sum_{i=0}^\infty\delta_i<r$.
Thus, by (11.10)-(\ref{e:pressure_final}), we know that $(v_{n},~p_{n},~\theta_{n},~R_{n},~f_{n})$ are Cauchy sequence in $C_c(Q_{2r})$, therefore there exists
\[\begin{aligned}
    (v,p,\theta)\in C_c(Q_{2r})
\end{aligned}\]
such that
\[\begin{aligned}
    v_{n}\rightarrow v,\quad p_{n}\rightarrow p,\quad \theta_{n}\rightarrow \theta,
    \quad R_{n}\rightarrow 0,\quad f_{n}\rightarrow 0
\end{aligned}\]
in $C_c(Q_{2r})$ and in particular,
\beno
\|v_n\|_0\leq \|v_0\|+M\sum_{j=0}^{\infty}a^{-\frac12b^j}\leq\|v_0\|+M\sum_{j=0}^{\infty}a^{-\frac12(\frac32)^j}.
\eeno
Thus, $\|v_n\|_0$ and $r+\delta_n$ are both bounded uniformly, hence we can assume that the constant $A$ only depend on $\varepsilon$.
Passing into the limit in (\ref{d:boussinesq reynold}), we conclude that $(v,~p,~\theta)$ solve (\ref{e:boussinesq equation}) in the sense of distribution. Moreover, since $\|v_{0}\|_0\geq 10M$, thus
\beno
\|v_n\|_0\geq \|v_0\|-M\sum_{j=0}^{\infty}a^{-\frac12b^j}\geq 10M-M\sum_{j=0}^{\infty}a^{-\frac12(\frac32)^j}\geq M,
\eeno
 hence the solution is non-trivial.

Next, we prove that the solution $v , p, \theta$ is H\"{o}lder continuous. We claim that for a suitable choice of $a, b$, there exist a constant $c> 1$ such that $$\Lambda_n\leq a^{cb^n}.$$
We prove this claim by induction.\\
Indeed, for $n=0$, it's obvious. Assuming that we have proved $\Lambda_n\leq a^{cb^n}$, then
\begin{align}
\Lambda_{n+1}\leq& A(\sqrt{\delta_n})^{\varepsilon^2+3\varepsilon+3}\Big(\frac{\sqrt{\delta_n}}{\delta_{n+1}}\Big)^{(1+\varepsilon)^{2}(2+\varepsilon)+(2+\varepsilon)^2}
\Lambda_n^{(1+\varepsilon)^{3}}\nonumber\\
\leq& Aa^{-\frac{\varepsilon}{2}b^n}a^{cb^{n+1}}a^{\big((b-\frac{1}{2})d-\frac{\varepsilon^2
+3\varepsilon+3}{2}+\frac{\varepsilon}{2}+c(1+\varepsilon)^3-cb\big)b^n},\nonumber
\end{align}
where $d=(1+\varepsilon)^{2}(2+\varepsilon)+(2+\varepsilon)^2$.\\
Take $b=\frac{3}{2}$ and $c=\frac{2d-(\varepsilon^2+2\varepsilon+3)}{3-2(1+\varepsilon)^3}$ , we arrive at
$$\Lambda_{n+1}\leq Aa^{-\frac{\varepsilon}{2}}a^{cb^{n+1}}.$$
Then choosing $a\geq A^{\frac{2}{\varepsilon}}$, we have $\Lambda_{n+1}\leq a^{cb^{n+1}}.$  Finally, we take
$$a:=\max\Big\{A^{\frac{2}{\varepsilon}},~~ \Lambda_0,~~~ \frac{3}{2},~~~\frac{1}{\min\{\frac{r}{2},\frac{\varepsilon^2}{16M^2}\}}\Big\},$$ then $a$
satisfies all the needed conditions.

Now we consider the approximate sequence $v_n, p_n, \theta_n$. By (\ref{e:velocity_final}), we have
\begin{align}
\|v_{n+1}-v_n\|_0\leq Ma^{-\frac{1}{2}b^n}.\nonumber
\end{align}
Moreover, we have
\begin{align}
\|v_{n+1}-v_n\|_{C^1_{t,x}}\leq \Lambda_n+\Lambda_{n+1}\leq 2a^{cb^{n+1}}.\nonumber
\end{align}
Therefore, for any $\alpha\in (0,1)$
\begin{align}
\|v_{n+1}-v_n\|_{C^\alpha_{t,x}}\leq  2Ma^{\big(\alpha cb-\frac{(1-\alpha)}{2}\big)b^n}.\nonumber
\end{align}
If $\alpha<\frac{1}{1+2bc}$, then $\alpha cb-\frac{(1-\alpha)}{2}< 0$, thus $v_n$ are Cauchy sequence in $C^\alpha_{t,x}.$ Take the value of $b,c$, we kwon that $v\in C^\alpha_{t,x}$ for any $\alpha<\frac{3-2(1+\varepsilon)^3}{3-2(1+\varepsilon)^3+6d-3(\varepsilon^2+2\varepsilon+3)}$. When $\varepsilon\rightarrow 0$, we have
$\frac{3-2(1+\varepsilon)^3}{3-2(1+\varepsilon)^3+6d-3(\varepsilon^2+2\varepsilon+3)}\rightarrow \frac{1}{28}$.
By (\ref{e:pressure_final}) and (\ref{e:final pressure and termperture estimate}), we know that $p\in C^\beta_{t,x}$ for any $\beta\in (0,1)$.\\
By (\ref{e:temperature_final}) and (\ref{e:final pressure and termperture estimate}), we have
\begin{align}
\|\theta_{n+1}-\theta_n\|_0\leq Ma^{-\frac{1}{2}b^n}\nonumber
\end{align}
and
\begin{align}
\|\theta_{n+1}-\theta_n\|_{C^1_{t,x}}\leq 2a^{(3+4\varepsilon+\varepsilon^2+c(1+\varepsilon)^2)b^n}.\nonumber
\end{align}
By interpolation, for any $\gamma\in (0,1)$, we have
\begin{align}
\|\theta_{n+1}-\theta_n\|_{C^\gamma_{t,x}}\leq 2Ma^{\big(\gamma(3+4\varepsilon+\varepsilon^2+c(1+\varepsilon)^2)-\frac{1-\gamma}{2}\big)b^n}.\nonumber
\end{align}
Take $\gamma<\frac{1}{1+2(3+4\varepsilon+\varepsilon^2+c(1+\varepsilon)^2)}$, then $\theta_n$ converge in $C^\gamma$, which implies that $\theta\in C^\gamma_{t, x}$ for any $\gamma<\frac{1}{1+2(3+4\varepsilon+\varepsilon^2+c(1+\varepsilon)^2)}$. When $\varepsilon\rightarrow 0$, we have $\frac{1}{1+2(3+4\varepsilon+\varepsilon^2+c(1+\varepsilon)^2)}\rightarrow \frac{1}{25}$.\\
Thus, we complete our proof of Theorem \ref{t:main 1}.
\end{proof}

{\bf Acknowledgments.}
The research is partially supported by the Chinese NSF under grant 11471320.
We thank Tianwen Luo for very valuable discussions and deeply grateful to the referee for his/her careful reading of the manuscript and the numerous very helpful suggestions.

\end{document}